\documentclass[11pt]{article}

\usepackage{amssymb,amsmath,amscd,amsthm,xy,enumitem,mathrsfs,caption,pstricks}
\xyoption{all}

\newtheorem{thm}{Theorem}[section]

\newtheorem{prp}[thm]{Proposition}

\theoremstyle{definition}
\newtheorem{dfn}[thm]{Definition}

\theoremstyle{remark}
\newtheorem{rmk}[thm]{Remark}

\newcounter{saveenumi}

\numberwithin{equation}{section}

\def\lra{\longrightarrow}

\def\xlra#1{\xrightarrow{#1}}

\def\BE#1{\begin{equation}\label{#1}}
\def\EE{\end{equation}}
\def\lr#1{\langle#1\rangle}

\def\blr#1{\big\langle#1\big\rangle}

\def\wt#1{\widetilde{#1}}
\def\wh#1{\widehat{#1}}
\def\ov#1{\overline{#1}}
\def\eref#1{(\ref{#1})}
\def\tn#1{\textnormal{#1}}
\def\sf#1{\textsf{#1}}

\def\sm#1{\begin{small}#1\end{small}}

\def\De{\Delta}

\def\La{\Lambda}
\def\Om{\Omega}
\def\Si{\Sigma}

\def\ga{\gamma}

\def\io{\iota}

\def\la{\lambda}

\def\om{\omega}
\def\si{\sigma}
\def\th{\theta}
\def\ve{\varepsilon}
\def\vph{\varphi}
\def\vp{\varpi}

\def\C{\mathbb C}
\def\cC{\mathcal C}  

\def\D{\mathbb D}
\def\fD{\mathfrak D}
\def\bE{\mathbb E}

\def\nd{\tn{d}}

\def\cF{\mathcal F}

\def\fI{\mathfrak i}
\def\cI{\mathcal I}
\def\cJ{\mathcal J}

\def\cL{\mathcal L}

\def\cM{\mathcal M}

\def\fM{\mathfrak M}

\def\cN{\mathcal N}
\def\cO{\mathcal O}
\def\fO{\mathfrak O}
\def\P{\mathbb P}
\def\cP{\mathcal P}

\def\R{\mathbb{R}}
\def\fR{\mathfrak R}

\def\bS{\mathbb S}

\def\cU{\mathcal U}

\def\cW{\mathcal W}

\def\Z{\mathbb{Z}}

\def\Q{\mathbb Q}

\def\fd{\mathfrak d}

\def\u{\mathbf u}

\def\fc{\mathfrak c}
\def\ft{\mathfrak t}
\def\fj{\mathfrak j}
\def\ff{\mathfrak f}
\def\fo{\mathfrak o}
\def\fs{\mathfrak s}

\def\ev{\tn{ev}}
\def\FS{\tn{FS}}

\def\id{\tn{id}}

\def\PD{\tn{PD}}
\def\pt{\tn{pt}}
\def\rk{\tn{rk}}

\def\top{\tn{top}}

\def\vir{\tn{vir}}

\def\hb{\hbar}
\def\dbar{\bar\partial}
\def\prt{\partial}
\def\eset{\emptyset}
\def\i{\infty}

\def\bu{\bullet}

\def\0{\mathbf 0}

\begin{document}

\title{Algebraic Properties of Real Gromov-Witten Invariants}
\author{Penka Georgieva\thanks{Partially supported by ERC Grant ROGW-864919} 
$~$and 
Aleksey Zinger\thanks{Partially supported by NSF grant DMS 2301493}}
\date{\today}
\maketitle

\begin{abstract}
\noindent
We describe properties of the previously constructed all-genus real Gromov-Witten theory
in the style of Kontsevich-Manin's axioms and other classical equations
and reconstruction results of complex Gromov-Witten theory.
\end{abstract}

\tableofcontents

\section{Introduction}
\label{intro_sec}

\noindent
The theory of $J$-holomorphic maps plays a prominent role in symplectic topology,
algebraic geometry, and string theory.
Kontsevich-Manin's axioms~\cite{KM} governing the associated numerical invariants of 
symplectic manifolds, known as \sf{Gromov-Witten} (or \sf{GW-}) \sf{invariants},
have proved instrumental for studying these invariants via algebraic methods.
For example, they underline the notion of Cohomological Field Theory (or~\sf{CohFT}).
Givental's discovery of symplectic actions on CohFTs \cite{Giv1,Giv2} and 
Teleman's subsequent classification of semi-simple CohFTs \cite{Teleman}
led to many reconstruction results, Virasoro constraints, and mirror symmetry 
as well as strong connections to integrable systems~\cite{DZ1,DZ2}.
On the other hand, the progress in the theory of real $J$-holomorphic maps, 
i.e.~of counts of $J$-holomorphic curves in symplectic manifolds 
preserved by anti-symplectic involutions, has been much slower.
In particular,  the associated numerical invariants of real symplectic manifolds, 
known as \sf{real GW-invariants}, did not even exist in positive genera until~\cite{RealGWsI};
a detailed summary of the construction of these invariants in~\cite{RealGWsI}
appears in~\cite{RealGWsSumm}.
The present paper provides analogues of Kontsevich-Manin's axioms and 
other standard algebraic properties of complex GW-invariants
for the real GW-invariants of~\cite{RealGWsI}.
It is intended to serve as a base for establishing analogues of Givental's and Teleman's results 
and their wide-ranging implications in real GW-theory.\\

\noindent
A \sf{real symplectic manifold} is a triple $(X,\om,\phi)$ consisting 
of a symplectic manifold~$(X,\om)$ and an anti-symplectic involution~$\phi$.
The fixed locus~$X^{\phi}$ of~$\phi$ is then a Lagrangian submanifold of~$(X,\om)$.
A \sf{real bundle pair} $(V,\vph)\!\lra\!(X,\phi)$ consists of a complex vector bundle 
$V\!\lra\!X$ and a conjugation~$\vph$ on $V$ lifting~$\phi$,
i.e.~$\vph^2\!=\!\id_V$ and
$\vph\!:V_x\!\lra\!\ov{V}_{\phi(x)}$ is a $\C$-linear isomorphism for every $x\!\in\!X$.
For a real bundle pair $(V,\vph)\!\lra\!(X,\phi)$, we  denote~by
$$\La_{\C}^{\top}(V,\vph)=\big(\La_{\C}^{\top}V,\La_{\C}^{\top}\vph\big)$$
the top exterior power of $V$ over $\C$ with the induced conjugation.
Direct sums, duals, and tensor products over~$\C$ of real bundle pairs over~$(X,\phi)$
are again real bundle pairs over~$(X,\phi)$. 
For any complex vector bundle $L\!\lra\!X$, 
the homomorphisms 
\begin{alignat*}{2}
\phi_L^{\oplus}\!:L\!\oplus\!\phi^*\ov{L}&\lra L\!\oplus\!\phi^*\ov{L}, 
&\quad\phi_L^{\oplus}(v,w)&=(w,v),  \qquad\hbox{and}\\
\phi_L^{\otimes}\!:L\!\otimes_{\C}\!\phi^*\ov{L}&\lra L\!\otimes_{\C}\!\phi^*\ov{L},
&\quad\phi_L^{\otimes}(v\!\otimes\!w)&=w\!\otimes\!v,
\end{alignat*}
are conjugations covering~$\phi$.

\begin{dfn}\label{realorient_dfn}
A \sf{real orientation} $(L,[\psi],\fs)$ on a real symplectic manifold~$(X,\om,\phi)$ consists~of 
\begin{enumerate}[label=(RO\arabic*),leftmargin=*]

\item\label{LBP_it} a complex line bundle $L\!\lra\!X$ such that 
\BE{realorient_e} w_2(TX^{\phi})=w_2(L)|_{X^{\phi}} \qquad\hbox{and}\qquad
\La_{\C}^{\top}(TX,\nd\phi)\approx\big(L\!\otimes_{\C}\!\phi^*\ov{L},\phi_L^{\otimes}\big),\EE

\item\label{isom_it} a homotopy class~$[\psi]$ of isomorphisms 
of real bundle pairs in~\eref{realorient_e}, and

\item\label{spin_it} a spin structure~$\fs$ on the real vector bundle
$TX^{\phi}\!\oplus\!L^*|_{X^{\phi}}$ over~$X^{\phi}$
compatible with the orientation induced by~$[\psi]$.

\end{enumerate}
\end{dfn}

\vspace{.1in}

\noindent
By~\ref{isom_it}, a real orientation on $(X,\om,\phi)$ determines a homotopy class of isomorphisms
$$\La_{\R}^{\top}\big(TX^{\phi}\big)
\!=\!\La_{\R}^{\top}
\big(\!\big\{\dot{x}\!\in\!TX|_{X^{\phi}}\!:\nd\phi(\dot{x})\!=\!\dot{x}\big\}\!\big)
\approx \big\{rv\!\otimes_{\C}\!v\!\in\!L\!\otimes_{\C}\!\phi^*\ov{L}\!:
v\!\in\!L|_{X^{\phi}},\,r\!\in\!\R\big\}$$
and thus an orientation on~$X^{\phi}$.
In particular, the real vector bundle $TX^{\phi}\!\oplus\!L^*|_{X^{\phi}}$ over~$X^{\phi}$
in~\ref{spin_it} is oriented.
By the first condition in~\eref{realorient_e}, it admits a spin structure.
By \cite[Theorems~1.1,1.2]{RBP}, a real orientation on $(X,\om,\phi)$ also determines 
a homotopy class of isomorphisms
\BE{RBPisom_e}u^*\big(TX\!\oplus\!(L^*\!\oplus\!\phi^*\ov{L}^*),\nd\phi\!\oplus\!\phi_{L^*}^{\oplus}\big)
\approx \big(\Si\!\times\!\C^{n+2},\si\!\times\!\fc\big)\EE
over a symmetric, possibly nodal, surface $(\Si,\si)$ for every real map 
\hbox{$u\!:(\Si,\si)\!\lra\!(X,\phi)$}, where $n$ is half the real dimension of~$X$ 
and $\fc$ is the standard conjugation on~$\C^{n+2}$.
This homotopy class of isomorphisms is one of the key ingredients 
in the construction of real GW-invariants in~\cite{RealGWsI}.

\begin{thm}\label{main_thm}
Let $(X,\om,\phi)$ be a compact connected real symplectic manifold of dimension $2n$ with $n\!\not\in\!2\Z$
and $(L,[\psi],\fs)$ be a real orientation on~$(X,\om,\phi)$.
The associated real GW-invariants~$\lr{\ldots}_{\!g,B}^{\om,\phi}$ in~\eref{RGWdfn_e} 
and the linear maps~$\cI_{g,\ell,B}^{\om,\phi}$ in~\eref{RIdfn_e} are well-defined
and satisfy the properties~\ref{Reff1_it}-\ref{RMotivic_it} in 
Sections~\ref{RGWs_subs1} and~\ref{RGWs_subs2}
as well as Theorems~\ref{Rrecont_thm1} and~\ref{RTRR_thm} in Section~\ref{RRecon_subs}.
\end{thm}

\noindent
A collection of homomorphisms~$\cI_{g,\ell,B}^{\om}$ as in~\eref{CIdfn_e}
which satisfies \ref{Ceff2_it}-\ref{CfcII_it}, \ref{CdivI_it} with $\ga\!=\!1$, 
and \ref{CMapPt_it}-\ref{CMotivic_it} in Section~\ref{CGWs_subs} is called 
a \sf{system of GW-classes} for~$(X,\om)$ in \cite[Definition~2.2]{KM};
a collection~$\{\cI_{0,\ell,B}^{\om}\}$ satisfying the associated properties is 
called a \sf{tree-level system of GW-classes} for~$(X,\om)$ in~\cite{KM}.
In light of the terminology used in the string theory literature, 
such as \cite{Walcher,Adam}, it would thus be appropriate to call 
a collection of homomorphisms~$\cI_{g,\ell,B}^{\om,\phi}$ as in~\eref{RIdfn_e}
which satisfies \ref{Reff2_it}-\ref{RfcII_it}, \ref{RdivI_it} with $\ga\!=\!1$, 
and \ref{RMapPt_it}-\ref{RMotivic_it} in Section~\ref{RGWs_subs2}  
a \sf{$\phi$-extension} of the system~$\{\cI_{g,\ell,B}^{\om}\}$ of GW-classes for~$(X,\om)$
and a collection~$\cI_{0,\ell,B}^{\om,\phi}$ satisfying the associated properties
a \sf{$\phi$-extension} of the tree-level system~$\{\cI_{0,\ell,B}^{\om}\}$ of GW-classes 
for~$(X,\om)$.
Theorem~\ref{Rrecont_thm1}, which is a real analogue of a now classical result of~\cite{KM}, 
describes the structure of the latter extension.
Similarly to the complex case of~\cite{KM} depending on the structure of 
the cohomology ring of the Deligne-Mumford moduli space~$\ov\cM_{0,\ell}$
of rational nodal $\ell$-marked curves,
the proof of Theorem~\ref{Rrecont_thm1} depends on the structure of 
the cohomology ring of the Deligne-Mumford moduli space~$\R\ov\cM_{0,\ell}$
of real rational nodal curves with $\ell$~conjugate pairs of marked points.
This was only recently obtained in~\cite{RDMhomol,RDMbl}.\\

\noindent
The anti-holomorphic involutions 
\begin{equation*}\begin{aligned}
\tau_n\!:\P^{n-1}&\lra\P^{n-1}, &
\tau_n\big([Z_1,\ldots,Z_n]\big)&=\begin{cases}
[\ov{Z}_2,\ov{Z}_1,\ldots,\ov{Z}_n,\ov{Z}_{n-1}],&\hbox{if}~n\!\in\!2\Z;\\
[\ov{Z}_2,\ov{Z}_1,\ldots,\ov{Z}_{n-1},\ov{Z}_{n-2},\ov{Z}_n],&\hbox{if}~n\!\not\in\!2\Z;
\end{cases}\\
\eta_{2m}\!: \P^{2m-1}& \lra\P^{2m-1},&
\eta_{2m}\big([Z_1,\ldots,Z_{2m}]\big)&= 
\big[\ov{Z}_2,-\ov{Z}_1,\ldots,\ov{Z}_{2m},-\ov{Z}_{2m-1}\big],
\end{aligned}\end{equation*}
are anti-symplectic with respect to the Fubini-Study symplectic
forms~$\om_n$ and~$\om_{2m}$ on the complex projective spaces~$\P^{n-1}$ 
and~$\P^{2m-1}$, respectively.
The real symplectic manifolds $(\P^{n-1},\om_n,\tau_n)$ with $n\!\in\!2\Z$ and 
$(\P^{2m-1},\om_{2m},\eta_{2m})$ admit real orientations~$(L,[\psi],\fs)$ 
with \hbox{$L\!=\!\cO_{\P^{n-1}}(n/2)$} and \hbox{$L\!=\!\cO_{\P^{2m-1}}(m)$}, respectively.
Many other examples of compact real symplectic manifolds with real orientations
are described in \cite[Section~1.1]{RealGWsIII}.
These include many projective complete intersections,
such as the real quintic threefolds,
i.e.~the smooth hypersurfaces in~$\P^4$ cut out by degree~5 homogeneous polynomials
on~$\C^5$ with real coefficients;
they play a prominent role in the interactions of symplectic topology 
with string theory and algebraic geometry.

\vspace{.1in}

\begin{rmk}\label{realorient_rmk}
The notion of real orientation provided by \cite[Definition~1.2]{RealGWsI} 
and used in~\cite{RealGWsI} to orient moduli spaces of stable real maps
replaces $(L\!\otimes_{\C}\!\phi^*\ov{L},\phi_L^{\otimes})$ and 
$(L^*\!\oplus\!\phi^*\ov{L}^*,\phi_{L^*}^{\oplus})$ in~\eref{realorient_e} and~\eref{RBPisom_e}
by $(L,\wt\phi)^{\otimes2}$ and $2(L,\wt\phi)^*$, respectively, 
for a real line bundle pair~$(L,\wt\phi)$ over~$(X,\phi)$.
For such a pair, the homomorphism
\BE{GIgen_e} \Phi_L\!:2(L,\wt\phi)\lra\big(L\!\oplus\!\phi^*\ov{L},\phi_L^{\oplus}\big),
\quad \Phi_L(v,w)=\big(v\!+\!\fI w,\wt\phi(v\!-\!\fI w)\!\big),\EE
is an isomorphism of real bundle pairs over~$(X,\phi)$.
Definition~\ref{realorient_dfn}, which is a slightly reworded version of \cite[Definition~A.1]{GI},
thus weakens (broadens) the notion of real orientation of \cite[Definition~1.2]{RealGWsI}.
The compositions of the isomorphisms \hbox{$u^*(\id_{TX}\!\oplus\!\Phi_{L^*})$}
and~\eref{RBPisom_e} yield a homotopy class of isomorphisms
\BE{RBPisom_e3}u^*\big(TX\!\oplus\!2L^*,\nd\phi\!\oplus\!2\wt\phi^*\big)
\approx \big(\Si\!\times\!\C^{n+2},\si\!\times\!\fc\big)\EE
over a symmetric $(\Si,\si)$ for every real map \hbox{$u\!:(\Si,\si)\!\lra\!(X,\phi)$}.
As noted in~\cite{GI} and detailed in~\cite{RealGWsGeom}, 
the construction of orientations of moduli spaces in~\cite{RealGWsI}
goes through almost verbatim with the weaker notion of real orientation 
of Definition~\ref{realorient_dfn}.
\end{rmk}

\vspace{.1in}

\noindent
By \cite[Theorem~4.5]{KM}, a tree-level system of GW-classes determines
a Frobenius structure on $H^*(X;\C)$ in the sense of \cite[Section~4.2]{KM}.
By \cite[Example~6.3]{KM}, a system of GW-classes determines a CohFT in 
the sense of \cite[Definition~6.1]{KM}.
This perspective makes it possible to recover arbitrary-genus GW-invariants 
of some symplectic manifolds, such as~$\P^n$, from the genus~0 GW-invariants;
see \cite{Giv1,Giv2}.
We hope that this perspective can be productively applied to $\phi$-extensions of 
system of GW-classes based on the results of this paper.
Okounkov-Pandharipande's trilogy \cite{OP1,OP2,OP3} determines
the (complex) Gromov-Witten theory of Riemann surfaces, 
in particular showing that the GW-invariants of~$\P^1$ satisfy integrable hierarchies of Toda type
and the GW-invariants of Riemann surfaces satisfy the Witten and Virasoro conjectures.
Guidoni~\cite{Gu} recently showed that the real GW-invariants of~$\P^1$ satisfy integrable 
hierarchies of types~CKP and~KdV. 
However, analogues of the Witten and Virasoro conjectures for real GW-invariants 
of Riemann surfaces remain unknown.
More geometric directions for future research are indicated in~\cite{RealGWsGeom}.\\ 

\noindent
After collecting the most frequently used notation and terminology in Section~\ref{nota_sec},
we recall basic properties of complex GW-invariants, Kontsevich-Manin's axioms
for these invariants, and two reconstruction results in the genus~0 complex GW-theory
in Section~\ref{CGWs_sec}.
We give a direct proof of the second reconstruction result, Proposition~\ref{CTRR_prp},
which readily adapts to the real setting.
After setting up the relevant notation in Section~\ref{RMS_subs},
we state and justify basic properties of 
the real GW-invariants arising from~\cite{RealGWsI} and analogues of Kontsevich-Manin's axioms
for the associated real GW-theory in Sections~\ref{RGWs_subs1} and~\ref{RGWs_subs2}.
Real analogues of the two reconstruction results of Section~\ref{CGWs_sec},
Theorems~\ref{Rrecont_thm1} and~\ref{RTRR_thm}, are established 
in Section~\ref{RRecon_subs}.
Section~\ref{RSplitPf_subs} contains the most technical proofs,
deducing the {\it Genus Reduction} and {\it Splitting} properties of Section~\ref{RGWs_subs2}
from the structural results for the orientations of the moduli spaces of stable real maps
established in~\cite{RealGWsGeom}.

\section{Notation and conventions}
\label{nota_sec}

\noindent
For $\ell\!\in\!\Z^{\ge0}$, we denote by~$\bS_{\ell}$ the $\ell$-th symmetric group.
For $\ell,g\!\in\!\Z^{\ge0}$, define 
\begin{gather*}
[\ell]=\{1,2,\ldots,\ell\big\}, \quad
\cP(g,\ell)=\big\{\!(g_1,g_2;I,J)\!: g_1,g_2\!\in\!\Z^{\ge0},\,g_1\!+\!g_2\!=\!g,\,
[\ell]\!=\!I\!\sqcup\!J\big\},\\
\wt\cP(g,\ell)=\big\{(g',g_0;I,J,K)\!: 
g',g_0\!\in\!\Z^{\ge0},\,2g'\!+\!g_0\!=\!g,\,[\ell]\!=\!I\!\sqcup\!J\!\sqcup\!K\big\}.
\end{gather*}
A decomposition $[\ell]\!=\!I\!\sqcup\!J$ determines a permutation 
\BE{Sperm_e}\big\{1,\ldots,\ell\big\}\lra I\!\sqcup\!J\EE
sending $1,\ldots,|I|$ to the elements of~$I$ in the increasing order
and \hbox{$|I|\!+\!1,\ldots,\ell$} to the elements of~$J$ in the increasing order.
For disjoint subsets $I,J\!\subset\!\Z^+$, let
$$\io_{I,J}\!:\big[|I\!\sqcup\!J|\big]\lra I\!\sqcup\!J$$
be the order-preserving bijection.\\

\noindent
Let $X$ be a compact connected oriented manifold as in the statement of Theorem~\ref{main_thm}. 
We denote by $H^*(X;\Q)$ the direct sum of the groups $H^p(X;\Q)$ over all $p\!\in\!\Z$.
For a homogeneous element~$\mu$ of $H^*(X;\Q)$, 
let $|\mu|$ be its cohomology degree, i.e.~$|\mu|\!=\!p$ if $\mu\!\in\!H^p(X;\Q)$.
In any equation involving an element \hbox{$\mu\!\in\!H^*(X;\Q)$} and its cohomology degree~$|\mu|$,
we implicitly assume that $\mu$ is a homogeneous element.
We fix a basis $\{e_i\}_{i\in[N]}$ of homogeneous elements for~$H^*(X;\Q)$ and
let $(g^{ij})_{i,j\in[N]}$ be the inverse of
the associated matrix 
$$\big(g_{ij}\!\equiv\!\lr{e_ie_j,X}\!\big)_{i,j\in[N]}$$
for the intersection form on~$H^*(X;\Q)$.
If $\De_X\!\in\!H_*(X^2;\Q)$ is the homology class determined by
the diagonal in~$X^2$ with the orientation induced from~$X$, then
\BE{DeXPD_e} \PD_{X^2}^{-1}\big(\De_X\big)=
\sum_{i,j=1}^N\!g_{ij}\,e_i\!\times\!e_j\EE
by~\eref{PDdfn_e}, \eref{cupcap_e}, and~\eref{FCprod_e} below.\\

\noindent
Let $\ell\!\in\!\Z^{\ge0}$, $\mu\!\equiv\!(\mu_i)_{i\in[\ell]}$ 
be a tuple of homogeneous elements of~$H^*(X;\Q)$, and 
$$|\mu|=|\mu_1|\!+\!\ldots\!+|\mu_{\ell}|.$$
For $I\!\equiv\!\{i_1,\ldots,i_m\}\!\subset\![\ell]$  with $i_1\!<\!\ldots\!<\!i_m$,
define 
$$\mu_I=\big(\mu_{i_1},\ldots,\mu_{i_m}\big).$$
For $\vp\!\in\!\bS_{\ell}$ and a partition $[\ell]\!=\!I\!\sqcup\!J$,  
let
$$\ve\big(\vp,\mu\big)
=\sum_{\begin{subarray}{c}i,j\in[\ell]\\ i<j\\ \vp(i)>\vp(j)\end{subarray}}
\hspace{-.2in} \big|\mu_i\big|\!\cdot\!\big|\mu_j\big|,
\quad 
\ve\big(\!(I,J),\mu\big)
=\sum_{\begin{subarray}{c}i\in I\\ j\in J\\ i>j\end{subarray}}|\mu_i|\!\cdot\!|\mu_j|.$$
The parity of the first (resp.~second) number above
is the parity of the permutation on the odd-degree classes~$\mu_i$
induced by the permutation~$\vp$ (resp.~permutation~\eref{Sperm_e}) of~$[\ell]$.
For $\vp\!\in\!\bS_{\ell}$, we define a graded automorphism of~$H^*(X;\Q)^{\otimes\ell}$~by
$$\vp\big(\mu_1\!\otimes\!\ldots\!\otimes\!\mu_{\ell}\big)=
(-1)^{\ve(\vp,(\mu_i)_{i\in[\ell]})}\mu_{\vp(1)}\!\otimes\!\ldots\!\otimes\!\mu_{\vp(\ell)}.$$
This determines an action of~$\bS_{\ell}$ on~$H^*(X;\Q)^{\otimes\ell}$ by 
graded automorphisms.
For 
\BE{PwtPdfn_e}P\equiv(g_1,g_2;I,J)\in\cP(g,\ell), \quad
\wt{P}\equiv\big(g',g_0;I,J,K)\in\wt\cP(g,\ell),\EE
and  $n\!\not\in\!2\Z$, we define
\begin{alignat*}{2}
 \ve_n(P)&=\frac{n\!-\!1}2(g_1\!-\!1)(g_2\!-\!1), &\qquad
\ve\big(P,\mu\big) &= \ve\big(\!(I,J),\mu\big), \\
\ve_n\big(P,\mu\big)&=\ve_n(P)\!+\!\ve\big(P,\mu\big)\!+\!(g_1\!-\!1)|\mu_J|, &\qquad
\ve\big(\wt{P},\mu\big)&=\ve\big(\!(I\!\sqcup\!J,K),\mu\big)\,.
\end{alignat*}   

\vspace{.18in}

\noindent
For a tuple $J\!\equiv\!(j_1,j_2,...)\!\in\!(\Z^{\ge0})^{\i}$
and a complex vector bundle $E\!\lra\!Y$, let
$$
c_J(E)=\prod_{k=1}^{\i}c_k^{j_k}(E)\in H^*(Y;\Q),$$
with $0^0\!\equiv\!1$.
By the Splitting Principle \cite[Problem~7C]{MiSt}, 
there exist universal coefficients \hbox{$C_{J_1,J_2}^{(n_1,n_2)}\!\in\!\Z$}
such~that 
$$e\big(E_1\!\otimes_{\C}\!E_2\big)=\sum_{J_1,J_2\in(\Z^{\ge0})^{\i}} \!\!\!\!\!\!\!\!\!\!
C_{J_1,J_2}^{(n_1,n_2)}c_{J_1}(E_1)c_{J_2}(E_2)$$
for all complex vector bundles $E_1,E_2\!\lra\!Y$ of ranks $n_1$ and $n_2$, respectively.\\

\noindent
For an orbifold~$\ov\cM$, we denote by~$H_*(\ov\cM;\Q)$ and~$H^*(\ov\cM;\Q)$ 
the homology and cohomology of the sheaf of singular chains for
the orbifold charts of~$\ov\cM$ with rational coefficients and
by~$\wt{H}_*(\ov\cM;\Q)$ and~$\wt{H}^*(\ov\cM;\Q)$ 
the homology and cohomology of this sheaf twisted by the orientation system~$\fO_{\ov\cM}$ of~$\ov\cM$,
as in~\cite[Section~3.H]{Hatcher} and \cite[Section~V.10]{Bredon}.
For a homogeneous element~$\mu$ of~$\wt{H}^*(\ov\cM;\Q)$, 
we again denote by~$|\mu|$ its cohomology degree.
We use the conventions for the cup and cap products in the singular (co)homology theory
as in~\S48 and~\S66, respectively, in~\cite{Mu2}, so~that
\BE{cupcap_e} \big(\ga\!\cup\!\mu\big)\cap A=\ga\cap\big(\mu\!\cap\!A\big)
\quad\forall~\ga\!\in\!H^*(\ov\cM;\Q),\,\mu\!\in\!\wt{H}^*(\ov\cM;\Q),
\,A\!\in\!\wt{H}_*(\ov\cM;\Q);\EE
see Theorem~66.2 in~\cite{Mu2}.
An orientation on~$\ov\cM$ identifies~$\wt{H}_*(\ov\cM;\Q)$ and~$\wt{H}^*(\ov\cM;\Q)$
with~$H_*(\ov\cM;\Q)$ and $H^*(\ov\cM;\Q$), respectively, 
intertwining the cup and cap products in the two (co)homology theories.\\

\noindent
If in addition $\ov\cM$ is compact  and either oriented or connected and unorientable,
let \hbox{$[\ov\cM]\!\in\!\wt{H}_*(\ov\cM;\Q)$} be its \sf{twisted fundamental class};
it corresponds to the fundamental class of~$\ov\cM$ in the first case and
to the fundamental class of the orientation double cover of~$\ov\cM$ with its canonical orientation  
in the second case.
In either case, the homomorphisms
\BE{PDdfn_e}\begin{split} 
\PD_{\ov\cM}\!:\wt{H}^*(\ov\cM;\Q)&\lra H_*(\ov\cM;\Q),\\
\PD_{\ov\cM}\!:H^*(\ov\cM;\Q)&\lra \wt{H}_*(\ov\cM;\Q),
\end{split} \qquad \PD_{\ov\cM}(\ga)= \ga\!\cap\![\ov\cM],\EE
are isomorphisms; 
see Theorem~3H.6 in~\cite{Hatcher} and Theorem~9.3 and Corollary~10.2 in~\cite{Bredon}.
In such cases, we also define
\BE{intdfn_e}\int_{\ov\cM}\!:\wt{H}^*(\ov\cM;\Q)\lra\Q, \quad
\int_{\ov\cM}\!\ga=\blr{\ga,[\ov\cM]}
=\big|\ga\!\cap\![\ov\cM]\big|^{\pm},\EE
where $|\cdot|^{\pm}$ is the degree~of 
(the weighted cardinality of the points~in) the $H_0$-part of~$\cdot$.
If~$\ov\cM_1$ and~$\ov\cM_2$ are compact connected orbifolds,
$\ov\cM_1$ is oriented, and~$\ov\cM_2$ is either oriented or unorientable,
then
\BE{FCprod_e} 
\int_{\ov\cM_1\times\ov\cM_2}\!\!\!\ga_1\!\times\!\ga_2
=\bigg(\int_{\ov\cM_1}\!\!\ga_1\!\!\bigg)\!\bigg(\int_{\ov\cM_2}\!\!\ga_2\!\!\bigg)
\quad\forall~\ga_1\!\in\!H^*(\ov\cM_1;\Q),\,\ga_2\!\in\!\wt{H}^*(\ov\cM_2;\Q).\EE

\vspace{.18in}

\noindent
Let $\ff\!:\ov\cM'\!\lra\!\ov\cM$ be a surjective morphism between orbifolds.
Suppose $\cU'\!\subset\!\ov\cM'$ is an open subset so that 
\BE{surjcond_e}\ov\cM'\!-\cU' \subset \ov\cM' \qquad\hbox{and}\qquad
\ff\big(\ov\cM'\!-\cU'\big) \subset \ov\cM\EE
are finite unions of suborbifolds of codimensions at least~2 and~1, respectively,
the restriction of~$\ff$ to~$\cU'$ is a submersion,
and the vector orbi-bundle
\BE{ndffbndl_e}\ker\nd\ff|_{\cU'}\lra \cU'\EE
is oriented. 
The short exact sequence 
\BE{submerses_e}0\lra \ker\nd\ff\big|_{\cU'}\lra T\ov\cM'\big|_{\cU'} \xlra{\,\nd\ff\,}
\ff^*T\ov\cM\big|_{\cU'}\lra0\EE
of vector orbi-bundles over~$\cU'$ and the orientation~$\fo$ of~\eref{ndffbndl_e} 
determine an isomorphism
$$\fO_{\ov\cM'}|_{\cU'}\approx\ff^*\fO_{\ov\cM}|_{\cU'}$$
of local systems over~$\cU'$.
By the codimension~2 assumption above, this isomorphism extends over all of~$\ov\cM'$
and thus determines and a homomorphism
$$H^*(\ov\cM';\Q)\!\otimes\!\wt{H}^*(\ov\cM;\Q) \lra\wt{H}^*(\ov\cM';\Q),
\qquad \ga'\!\otimes\!\ga\lra \ga'(\ff^*\ga) .$$
Furthermore, $\ff^{-1}(u)\!\subset\!\ov\cM'$ is an oriented suborbifold 
for every \hbox{$u\!\in\!\ff(\ov\cM'\!-\cU')$}.
If in addition $\ov\cM,\ov\cM'$ are compact and either oriented so that 
the exact sequence~\eref{submerses_e} respects the three orientations
or connected and unorientable,
then $\ff^{-1}(u)\!\subset\!\ov\cM'$ is compact and 
\BE{fiberint_e}\blr{\ga'(\ff^*\ga),[\ov\cM']}=\blr{\ga',[\ff^{-1}(u)]}\blr{\ga,[\ov\cM]}\EE
for every $\ga\!\in\!\wt{H}^*(\ov\cM;\Q)$ of top degree, 
every $\ga'\!\in\!H^*(\ov\cM';\Q)$, and every $u\!\in\!\ff(\ov\cM'\!-\cU')$.\\

\noindent
Let $\io\!:\ov\cM'\!\lra\!\ov\cM$ be a codimension~$r$ immersion between orbifolds.
Suppose the normal bundle 
$$\pi\!:\cN\io\equiv \frac{\io^*T\ov\cM}{\nd\io(T\ov\cM')}\lra \ov\cM'$$
of~$\io$ is oriented.
The short exact sequence 
\BE{immerses_e}0\lra  T\ov\cM' \xlra{\,\nd\io\,}\io^*T\ov\cM\lra\cN\io\lra0\EE
of vector orbi-bundles over~$\ov\cM'$ and the orientation~$\fo$ of~~$\cN\io$ 
determine an isomorphism \hbox{$\fO_{\ov\cM'}\!\approx\!\io^*\fO_{\ov\cM}$}
of local systems over~$\ov\cM'$ and thus a homomorphism
$$H^*(\ov\cM';\Q)\!\otimes\!\wt{H}^*(\ov\cM;\Q)\lra\wt{H}^*(\ov\cM';\Q),
\qquad \ga'\!\otimes\!\ga\lra \ga'(\io^*\ga).$$
If in addition $\ov\cM,\ov\cM'$ are compact and either oriented so that 
the exact sequence~\eref{immerses_e} respects the three orientations
or connected and unorientable, then we obtain a \sf{pushforward homomorphism}
$$\io_*\!:H^*(\ov\cM';\Q)\lra H^*(\ov\cM;\Q),  \qquad
\io_*(\ga')=\PD_{\ov\cM}^{-1}\big(\io_*\big(\PD_{\ov\cM'}(\ga')\!\big)\!\big).$$
By~\eref{cupcap_e}-\eref{intdfn_e},
\BE{pushintprp_e} 
\blr{\ga'(\io^*\ga),[\ov\cM']}=(-1)^{r|\ga|}\blr{(\io_*\ga')\ga,[\ov\cM]}
\quad\forall~\ga\!\in\!H^*(\ov\cM';\Q),\,\ga'\!\in\!\wt{H}^*(\ov\cM;\Q).\EE
Similarly to the definition on page~120 and Exercise 11-C in~\cite{MiSt},
the orientation~$\fo$ on~$\cN\io$ also determines a class 
\hbox{$u''\!\in\!H^r(\ov\cM;\Q)$} so~that
\BE{pushint_e}
\io_*\big([\ov\cM']\big)=u''\!\cap\![\ov\cM]\in \wt{H}_*(\ov\cM;\Q).\EE

\section{Complex Gromov-Witten theory}
\label{CGWs_sec}

\subsection{Moduli spaces}
\label{CMS_subs}

\noindent
Let $g\!\in\!\Z$ and $\ell\!\in\!\Z^{\ge0}$.
We denote by $\ov\cM_{g,\ell}$ the Deligne-Mumford moduli space
of stable closed connected, but possibly nodal, $\ell$-marked Riemann surfaces~$\Si$ 
of arithmetic genus~$g$.
This space is a compact complex orbifold of complex dimension
\BE{cMdim_e}\dim_{\C}\ov\cM_{g,\ell}=3(g\!-\!1)\!+\!\ell\EE
and is thus oriented; it is empty if $g\!<\!0$ or $2g\!+\!\ell\!<\!3$.
A permutation $\vp\!\in\!\bS_{\ell}$ acts on~$\ov\cM_{g,\ell}$ by 
sending a stable $\ell$-marked Riemann surface~$\Si$
to the stable $\ell$-marked Riemann surface~$\vp(\Si)$ so~that the marked points
of the two surfaces are related~by
\BE{bSellact_e} z_{\vp(i)}\big(\vp(\Si)\!\big)=z_i(\Si) \quad\forall\,i\!\in\![\ell].\EE
This determines an action of~$\bS_{\ell}$ on~$\ov\cM_{g,\ell}$ by 
holomorphic automorphisms.
Let
$$\bE\lra \ov\cM_{g,\ell}$$
be the Hodge vector bundle of holomorphic differentials.\\

\noindent
If $\ell\!>\!0$ and $2g\!+\!\ell\!>\!3$, let
$$\ff_{\ell}\!: \ov\cM_{g,\ell}\lra \ov\cM_{g,\ell-1}$$
be the natural forgetful morphism dropping the last marked point.
The preimage $\cU'\!\subset\!\ov\cM_{g,\ell}$ of the open subspace
\hbox{$\cM_{g,\ell-1}\!\subset\!\ov\cM_{g,\ell-1}$} consisting of smooth curves
satisfies the codimension conditions below~\eref{surjcond_e} and 
the orientation condition above~\eref{fiberint_e}.
We note~that
\BE{bEffl_e} \bE=\ff_{\ell}^*\bE\lra \ov\cM_{g,\ell}\EE
under the above assumptions.\\

\noindent
We denote by
\BE{Ciogldfn_e}\io_{g,\ell}\!:\ov\cM_{g-1,\ell+2}\lra\ov\cM_{g,\ell}\EE
the immersion obtained by identifying the last two marked points of each Riemann surface
in the domain to form a node; this map is generically $2\!:\!1$ onto its image.
For each $P\!\in\!\cP(g,\ell)$ as in~\eref{PwtPdfn_e}, let
\BE{CioPdfn_e}\io_P\!: 
\ov\cM_P\!\equiv\!\ov\cM_{g_1,|I|+1}\!\times\!\ov\cM_{g_2,|J|+1}\lra \ov\cM_{g,\ell}\EE
be the immersion obtained by identifying the last marked points of each pair of Riemann surfaces
in the domain to form a node and 
by re-ordering the remaining pairs of marked points according to the bijection~\eref{Sperm_e}.
These two immersions are illustrated in Figure~\ref{Cimmers_fig}.
The normal bundles to these immersions have canonical orientations that satisfy
the orientation condition above~\eref{pushintprp_e}.\\

\begin{figure}
\begin{pspicture}(-.5,-.2)(10,2)
\psset{unit=.4cm}
\psarc[linewidth=.06](4,2.5){2}{135}{45}\pscircle*(2,2.5){.2}\pscircle*(6,2.5){.2}
\rput(.8,2.5){\sm{$z_{\ell+1}$}}\rput(7.2,2.5){\sm{$z_{\ell+2}$}}
\psline[linewidth=.03]{->}(7,1.5)(11,1.5)\rput(9,2.1){\sm{$\io_{g,\ell}$}}
\psarc[linewidth=.06](13,1.7){1.2}{120}{60}\pscircle*(13,3.09){.2}
\psline[linewidth=.06](12.4,2.74)(14.2,3.78)\psline[linewidth=.06](13.6,2.74)(11.8,3.78)
\psarc[linewidth=.06](20,2.5){2}{-90}{0}\psline[linewidth=.06](22,2.5)(22,3.5)
\psarc[linewidth=.06](25,2.5){2}{180}{270}\psline[linewidth=.06](23,2.5)(23,3.5)
\pscircle*(22,2.5){.2}\pscircle*(23,2.5){.2}
\rput(20.7,2.5){\sm{$z_{|I|+1}$}}\rput(24.5,2.5){\sm{$z_{|J|+1}$}}
\psline[linewidth=.03]{->}(25,1.5)(29,1.5)\rput(27,2.1){\sm{$\io_P$}}
\psarc[linewidth=.06](29,2.5){2}{-90}{-45}\psline[linewidth=.06](30.41,1.09)(32.41,3.09)
\psarc[linewidth=.06](34,2.5){2}{225}{270}\psline[linewidth=.06](32.59,1.09)(30.59,3.09)
\pscircle*(31.5,2.18){.2}
\rput(28.5,.5){\sm{$g_1$}}\rput(34.5,.5){\sm{$g_2$}}
\rput(30.6,.6){\sm{$I$}}\rput(32.2,.6){\sm{$J$}}
\end{pspicture}
\caption{Typical elements in the domains and images of the immersions~\eref{Ciogldfn_e} 
and~\eref{CioPdfn_e},
with the genus and marked points of each irreducible component of an image of~\eref{CioPdfn_e}
indicated next to~it.}
\label{Cimmers_fig}
\end{figure}

\noindent
Let $(X,\om)$ be a compact symplectic manifold of real dimension~$2n$,  
$$H_2(X;\Z)_{\om}=\big\{B\!\in\!H_2(X;\Z)\!:\,\om(B)\!>\!0~\hbox{or}~B\!=\!0\big\},$$
and $\cJ_{\om}$ be the space of $\om$-tamed almost complex structures on~$X$.
For $g,\ell\!\in\!\Z^{\ge0}$, $B\!\in\!H_2(X;\Z)$, and $J\!\in\!\cJ_{\om}$,
we denote by $\ov\fM_{g,\ell}(B;J)$ the moduli space of stable $J$-holomorphic degree~$B$ maps 
from closed connected, but possibly nodal, $\ell$-marked Riemann surfaces of arithmetic genus~$g$.
This space is empty if either \hbox{$B\!\not\in\!H_2(X;\Z)_{\om}$} or
$B\!=\!0$ and $2g\!+\!\ell\!<\!3$.
For each $i\!=\!1,\ldots,\ell$, let 
$$\ev_i\!:\ov\fM_{g,\ell}(B;J)\lra X \qquad\hbox{and}\qquad
\psi_i\in H^2\big(\ov\fM_{g,\ell}(B;J);\Q\big)$$
be the natural evaluation map at the  $i$-th  marked point and
the Chern class of the universal cotangent line bundle for this marked point, 
respectively.
The symmetric group~$\bS_{\ell}$ acts on $\ov\fM_{g,\ell}(B;J)$ 
similarly to~\eref{bSellact_e}.
This action satisfies
\BE{bSellact_e2b}\ev_i\!=\!\ev_{\vp(i)}\!\circ\!\vp \quad\hbox{and}\quad
\psi_i\!=\!\vp^*\psi_{\vp(i)} \qquad\forall\,\vp\!\in\!\bS_{\ell},\,i\!\in\![\ell].\EE

\vspace{.18in}

\noindent
By \cite{LT,FO}, the moduli space $\ov\fM_{g,\ell}(B;J)$ carries 
a natural virtual fundamental class of dimension/degree
\BE{CfMdim_e}\begin{split}
\dim\,[\ov\fM_{g,\ell}(B;J)]^{\vir}
&=2\big((1\!-\!g)(n\!-\!3)\!+\!\ell\!+\!\blr{c_1(X,\om),B}\!\big)\\
&=2\dim_{\C}\ov\cM_{g,\ell}
\!+\!2\big(n(1\!-\!g)\!+\!\blr{c_1(X,\om),B}\!\big).
\end{split}\EE
This class is preserved by the $\bS_{\ell}$-action.
For \hbox{$a_1,\ldots,a_{\ell}\!\in\!\Z^{\ge0}$} and 
\hbox{$\mu_1,\ldots,\mu_{\ell}\!\in\!H^*(X;\Q)$},
let 
\BE{CGWdfn_e}\blr{\tau_{a_1}(\mu_1),\ldots,\tau_{a_{\ell}}(\mu_{\ell})}_{\!g,B}^{\!\om} 
=\int_{[\ov\fM_{g,\ell}(B;J)]^{\vir}}\!\! 
\psi_1^{a_1}\!\big(\ev_1^*\mu_1\big)\ldots\psi_{\ell}^{a_{\ell}}\!\big(\ev_{\ell}^*\mu_{\ell}\big)\EE
be the associated \sf{descendant GW-invariant}.
This number is independent of the choice of~$J\!\in\!\cJ_{\om}$.\\

\noindent
If $2g\!+\!\ell\!\ge\!3$, let
\BE{ffdfn_e}\ff:\ov\fM_{g,\ell}(B;J)\lra \ov\cM_{g,\ell} \EE
be the natural forgetful morphism to the corresponding Deligne-Mumford moduli space.
It satisfies
\BE{ffvp_e}\ff\!\circ\!\vp\!=\!\vp\!\circ\!\ff:
\ov\fM_{g,\ell}(B;J)\lra \ov\cM_{g,\ell}
\qquad\forall\,\vp\!\in\!\bS_{\ell}\,.\EE
We denote by
$$\pi_{\ov\cM_{g,\ell}},\pi_{X^{\ell}}\!:
\ov\cM_{g,\ell}\!\times\!X^{\ell}\lra \ov\cM_{g,\ell},X^{\ell}$$
the component projection maps.
Using Poincar\'e Duality on~$\ov\cM_{g,\ell}$ and $\ov\cM_{g,\ell}\!\times\!X^{\ell}$,
we define 
\BE{CIdfn_e}
\cI_{g,\ell,B}^{\om}\!: H^*(X;\Q)^{\otimes\ell}\lra H^*\big(\ov\cM_{g,\ell};\Q\big)
\quad\hbox{and}\quad
C_{g,\ell,B}^{\om}\in H^*\big(\ov\cM_{g,\ell}\!\times\!X^{\ell};\Q\big)\EE
by requiring that 
\begin{gather*}
\int_{\ov\cM_{g,\ell}}\!\!\ga
\cI_{g,\ell,B}^{\om}\big(\mu_1,\ldots,\mu_{\ell}\big)=
\int_{[\ov\fM_{g,\ell}(B;J)]^{\vir}}\!\!
\big(\ff^*\ga\big)\!  
\big(\ev_1^*\mu_1\big)\ldots\!\big(\ev_{\ell}^*\mu_{\ell}\big)\quad\hbox{and}\\
\int_{\ov\cM_{g,\ell}\times X^{\ell}}\!\! 
\big(\pi_{\ov\cM_{g,\ell}}^*\!\ga\big)C_{g,\ell,B}^{\om}
\big(\pi_{X^{\ell}}^*(\mu_1\!\times\!\ldots\!\times\!\mu_{\ell})\!\big)
=\int_{[\ov\fM_{g,\ell}(B;J)]^{\vir}}\!\! \big(\ff^*\ga\big)\!
\big(\ev_1^*\mu_1\big)\ldots\!\big(\ev_{\ell}^*\mu_{\ell}\big)
\end{gather*}
for all $\mu_i\!\in\!H^*(X;\Q)$  and $\ga\!\in\!H^*(\ov\cM_{g,\ell};\Q)$.
The linear maps~$\cI_{g,\ell,B}^{\om}$ and the correspondences $C_{g,\ell,B}^{\om}$ in~\eref{CIdfn_e} 
are independent of the choice of \hbox{$J\!\in\!\cJ_{\om}$}.
If $g\!\not\in\!\Z^{\ge0}$ or $2g\!+\!\ell\!<\!3$, we set $\cI_{g,\ell,B}^{\om}\!=\!0$.

\subsection{Properties of invariants}
\label{CGWs_subs}

\noindent
The descendant GW-invariants~\eref{CGWdfn_e} satisfy the following properties:
\begin{enumerate}[label=$\C{\arabic*}\!\!$,ref=$\C{\arabic*}$,leftmargin=*]

\item\label{Ceff1_it} ({\it Effectivity I}):
$\lr{\tau_{a_1}(\mu_1),\ldots,\tau_{a_{\ell}}(\mu_{\ell})}_{\!g,B}^{\om}\!=\!0$ 
if $B\!\not\in\!H_2(X;\Z)_{\om}$ or $B\!=\!0$ and $2g\!+\!{\ell}\!<\!3$;

\item\label{Cgrad1_it} ({\it Grading I}): 
$\lr{\tau_{a_1}(\mu_1),\ldots,\tau_{a_{\ell}}(\mu_{\ell})}_{\!g,B}^{\om}\!=\!0$ if
$$\sum_{i=1}^{\ell}\!\big(2a_i\!+\!|\mu_i|\big)
\neq2\big((1\!-\!g)(n\!-\!3)\!+\!\ell\!+\!\blr{c_1(X,\om),B}\!\big);$$

\item\label{Cstr_it}  ({\it String}): if $B\!\neq\!0$ or $(g,\ell)\!\neq\!(0,2)$, 
\begin{equation*}\begin{split}
&\blr{\tau_{a_1}(\mu_1),\ldots,\tau_{a_{\ell}}(\mu_{\ell}),\tau_0(1)}_{\!g,B}^{\!\om}\\
&\qquad=\sum_{\begin{subarray}{c}1\le i \le\ell\\ a_i>0\end{subarray}}
\!\!\!\blr{\tau_{a_1}(\mu_1),\ldots,\tau_{a_{i-1}}(\mu_{i-1}), 
\tau_{a_i-1}(\mu_i),\tau_{a_{i+1}}(\mu_{i+1}),\ldots,\tau_{a_{\ell}}(\mu_{\ell})}_{\!g,B}^{\!\om}; 
\end{split}\end{equation*}

\item\label{Cdil_it}  ({\it Dilaton}): 
if $B\!\neq\!0$ or $(g,\ell)\!\neq\!(1,0)$,
$$\blr{\tau_{a_1}(\mu_1),\ldots,\tau_{a_{\ell}}(\mu_{\ell}),\tau_1(1)}_{\!g,B}^{\!\om}
=\big(2g\!-\!2\!+\!{\ell}\big)\blr{\tau_{a_1}(\mu_1),
\ldots,\tau_{a_{\ell}}(\mu_{\ell})}_{\!g,B}^{\!\om};$$

\item\label{CdivII_it}  ({\it Divisor I}):
if $\mu_{\ell+1}\!\in\!H^2(X;\Q)$ and either $B\!\neq\!0$ or $2g\!+\!\ell\ge3$,
\begin{equation*}\begin{split}
&\blr{\tau_{a_1}(\mu_1),\ldots,\tau_{a_{\ell}}(\mu_{\ell}),\tau_0(\mu_{\ell+1})}_{\!g,B}^{\!\om} 
=\lr{\mu_{\ell+1},B}\blr{\tau_{a_1}(\mu_1),\ldots,\tau_{a_{\ell}}(\mu_{\ell})}_{\!g,B}^{\om}\\
&\qquad+\sum_{\begin{subarray}{c}1\le i\le\ell\\ a_i>0\end{subarray}}
\!\!\!\blr{\tau_{a_1}(\mu_1),\ldots,\tau_{a_{i-1}}(\mu_{i-1}),
\tau_{a_i-1}(\mu_i\mu_{\ell+1}),\tau_{a_{i+1}}(\mu_{i+1}),\ldots, 
\tau_{a_{\ell}}(\mu_{\ell})}_{\!g,B}^{\!\om}.
\end{split}\end{equation*}

\setcounter{saveenumi}{\arabic{enumi}}
\end{enumerate}
The {\it Effectivity} properties above and below follow immediately 
from $\ov\fM_{g,\ell}(B;J)\!=\!\eset$
if either \hbox{$B\!\not\in\!H_2(X;\Z)_{\om}$} or $B\!=\!0$ and $2g\!+\!\ell\!<\!3$.
The ${\it Grading}$ properties above and below are consequences of~\eref{CfMdim_e}.
For \ref{Cstr_it}-\ref{CdivII_it}, see \cite[Section~26.3]{MirSym}.\\

\noindent
The linear maps~\eref{CIdfn_e} satisfy \sf{Kontsevich-Manin's axioms} of 
\cite[Section~2]{KM}:
\begin{enumerate}[label=$\C{\arabic*}\!\!$,ref=$\C{\arabic*}$,leftmargin=22pt]
\setcounter{enumi}{\arabic{saveenumi}}

\item\label{Ceff2_it} ({\it Effectivity II}): 
$\cI_{g,\ell,B}^{\om}\!=\!0$ if $B\!\not\in\!H_2(X;\Z)_{\om}$;

\item\label{Ccov_it} ({\it $\bS_{\ell}$-Covariance}): 
the map $\cI_{g,\ell,B}^{\om}$ is $\bS_{\ell}$-equivariant;

\item\label{Cgrad_it} ({\it Grading II}): $\cI_{g,\ell,B}^{\om}$ is homogeneous of degree
$2(g\!-\!1)n\!-\!2\lr{c_1(X,\om),B}$, i.e.
$${}\hspace{-.5in}
\big|\cI_{g,\ell,B}^{\om}(\mu)\big|=|\mu|+2(g\!-\!1)n\!-\!2\blr{c_1(X,\om),B}
\quad\forall\,\mu\!\in\!H^*(X;\Q)^{\ell};$$

\item\label{CfcI_it} ({\it Fundamental Class I}): for all $\mu_1,\mu_2\!\in\!H^*(X;\Q)$,
$$\cI_{0,3,B}^{\om}\big(\mu_1,\mu_2,1\big)
=\begin{cases} 0, &\hbox{if}~B\!\neq\!0;\\ 
\blr{\mu_1\mu_2,[X]}, &\hbox{if}~B\!=\!0;\end{cases}$$

\item\label{CfcII_it} ({\it Fundamental Class II}): if $2g\!+\!\ell\!\ge\!3$
and $\mu\!\in\!H^*(X;\Q)^{\ell}$, 
$${}\hspace{-.5in}
\cI_{g,\ell+1,B}^{\om}(\mu,1)
=\ff_{\ell+1}^*\big(\cI_{g,\ell,B}^{\om}(\mu)\!\big);$$

\item\label{CdivI_it}  ({\it Divisor II}):  
if $2g\!+\!\ell\!\ge\!3$, $\mu\!\in\!H^*(X;\Q)^{\ell}$, $\mu_{\ell+1}\!\in\!H^2(X;\Q)$, 
and $\ga\!\in\!H^*(\ov\cM_{g,\ell};\Q)$, 
$${}\hspace{-.5in}
\int_{\ov\cM_{g,\ell+1}}\!\!\big(\ff_{\ell+1}^*\ga\big)
\cI_{\!g,\ell+1,B}^{\om}\big(\mu,\mu_{\ell+1}\big)
=\lr{\mu_{\ell+1},B}\!
\int_{\ov\cM_{g,\ell}}\!\!\ga\cI_{g,\ell,B}^{\om}(\mu);$$

\item\label{CMapPt_it} ({\it Mapping to Point}):
for all $\mu_1,\ldots,\mu_{\ell}\!\in\!H^*(X;\Q)$,
$${}\hspace{-.5in}
\cI_{g,{\ell},0}^{\om}\big(\mu_1,\ldots,\mu_{\ell}\big)
=\sum_{J_1,J_2\in(\Z^{\ge0})^{\i}} \!\!\!\!\!\!\!\!\!\!
C_{J_1,J_2}^{(n,g)}\blr{c_{J_1}(X,\om)\mu_1\ldots\mu_{\ell},[X]}c_{J_2}(\bE^*)
\in H^*\big(\ov\cM_{g,\ell};\Q\big);$$

\item\label{CGenusRed_it} ({\it Genus Reduction}): for all $\mu\!\in\!H^*(X;\Q)^{\ell}$,
$$\io_{g,\ell}^{\,*}\big(\cI_{g,\ell,B}^{\om}(\mu)\!\big)
=\sum_{i,j=1}^N\! g^{ij} \cI_{g-1,\ell+2,B}^{\om}\big(\mu,e_i,e_j\big);$$

\item\label{CSplit_it} ({\it Splitting}): for all 
$\mu\!\in\!H^*(X;\Q)^{\ell}$ and $P\!\equiv\!(g_1,g_2;I,J)\!\in\!\cP(g,\ell)$,
$${}\hspace{-.5in}
\io_P^*\big(\cI_{g,\ell,B}^{\om}(\mu)\!\big)=(-1)^{\ve(P,\mu)}
\!\!\!\!\!\!\!\!\sum_{\begin{subarray}{c}B_1,B_2\in H_2(X;\Z)\\
B_1+B_2=B\end{subarray}}\!\sum_{i,j=1}^N\! g^{ij}
\cI_{g_1,|I|+1,B_1}^{\om}\!\big(\mu_I,e_i\big)\!\times\!
\cI_{g_2,|J|+1,B_2}^{\om}\!\big(e_j,\mu_J\big);$$

\item\label{CMotivic_it} ({\it Motivic Axiom}):
$\cI_{g,\ell,B}^{\om}$ is induced by the correspondence~$C_{g,\ell,B}^{\om}$, i.e.
$${}\hspace{-.3in}\cI_{g,\ell,B}^{\om}(\mu_1,\ldots,\mu_{\ell})= 
\PD_{\ov\cM_{g,\ell}}^{\,-1}\!\bigg(\!\!
\big\{\pi_{\ov\cM_{g,\ell}^{\bu}}\big\}_{\!*}
\PD_{\ov\cM_{g,\ell}\times X^{\ell}}\!
\Big(C_{g,\ell,B}^{\om}\big(\!\big\{\pi_{X^{\ell}}\big\}^{\!*}\!
(\mu_1\!\times\!\ldots\!\times\!\mu_{\ell})\!\big)\!\Big)
\!\!\!\bigg)$$
for all $\mu_1,\ldots,\mu_{\ell}\!\in\!H^*(X;\Q)$.
\end{enumerate}

\vspace{.18in}

\noindent
The property~\ref{Ccov_it} follows from the $\bS_{\ell}$-invariance
of the virtual fundamental class of~$\ov\fM_{g,\ell}(B;J)$,
the first identity in~\eref{bSellact_e2b}, and~\eref{ffvp_e}.
The property~\ref{CMotivic_it} is immediate from~\eref{cupcap_e}, \eref{PDdfn_e},  
and the definitions of~$\cI_{g,\ell,B}^{\om}$ 
and~$C_{g,\ell,B}^{\om}$ after~\eref{CIdfn_e}.
The statements \ref{CfcI_it}-\ref{CSplit_it}
are consequences of natural properties of the virtual fundamental class for $\ov\fM_{g,\ell}(B;J)$
constructed in~\cite{LT,FO}.
The analogous properties in the real GW-theory are 
established in Sections~\ref{RGWs_subs2} and~\ref{RSplitPf_subs}.\\

\noindent
The $g\!=\!0$ case of~\ref{CMapPt_it} is equivalent~to 
\BE{g0basic_e} \cI_{0,\ell,0}^{\om}\big(\mu_1,\ldots,\mu_{\ell}\big)
=\blr{\mu_1\!\ldots\!\mu_{\ell},X}\in H^0\big(\ov\cM_{0,\ell};\Q\big).\EE
The $g\!=\!1$ case of~\ref{CMapPt_it} is equivalent~to 
\BE{g1basic_e}
\cI_{1,\ell,0}^{\om}\big(\mu_1,\ldots,\mu_{\ell}\big)
=\begin{cases}
\lr{\mu_1\!\ldots\!\mu_{\ell} c_n(X,\om),X},&\hbox{if}~\mu_1,\ldots,\mu_{\ell}\!\in\!H^0(X;\Q);\\
-\lr{\mu_1\!\ldots\!\mu_{\ell} c_{n-1}(X,\om),X}c_1(\bE),&\hbox{if}~
|\mu_1|\!+\!\ldots\!+\!|\mu_{\ell}|\!=\!2;\\
0,&\hbox{otherwise}.
\end{cases}\EE
By~\ref{CMapPt_it}, \eref{cMdim_e}, and~\eref{bEffl_e}, 
the restriction of $\cI_{g,\ell,B}^{\om}(\mu_1,\ldots,\mu_{\ell})$
to $\ov\cM_{g,\ell}$ vanishes if $g\!\ge\!2$ and $n\!\ge\!4$.

\subsection{Reconstruction of invariants}
\label{CRecon_subs}

\noindent
Following \cite[Definition~2.2]{KM}, we call a collection of homomorphisms~$\cI_{0,\ell,B}^{\om}$
as in~\eref{CIdfn_e} with $\ell\!\ge\!3$ and $B\!\in\!H_2(X;\Z)$ which satisfies
the properties \ref{Ceff2_it}-\ref{CfcII_it}, \ref{CdivI_it} with $\ga\!=\!1$, 
and \ref{CMapPt_it}-\ref{CMotivic_it} a \sf{tree-level system of GW-classes for~$(X,\om)$}.
By Proposition~2.5.2 and Theorem~3.1 in~\cite{KM},
a tree-level system of GW-classes for~$(X,\om)$
can be reconstructed from a small subcollection of the full collection.
This follows readily from the fact that the ring $H^*(\ov\cM_{0,\ell};\Q)$
is generated by the Poincar\'e duals of the ``boundary" classes, 
the images of the immersion~$\io_P$ in~\eref{CioPdfn_e} with $P\!\in\cP(g,\ell)$;
see~(1) and~(4) on~p545 in~\cite{Keel}.

\begin{prp}[{\cite{KM}}]\label{Crecont_prp1}
Let $(X,\om)$ be a compact symplectic manifold of real dimension~$2n$.
A tree-level system $\{\cI_{0,\ell,B}^{\om}\}$ of GW-classes for~$(X,\om)$
is determined by the~numbers
$$\int_{\ov\cM_{0,{\ell}}}\!\!\!\cI_{0,\ell,B}^{\om}(\mu)\in\Q
\qquad\hbox{with}\quad
\ell\in\Z^{\ge0},\,B\!\in\!H_2(X;\Z)_{\om},\,\mu\!\in\!H^*(X;\Q)^{\ell},$$
i.e.~by the codimension~0 classes $\cI_{0,\ell,B}^{\om}(\mu)$.
If in addition $H^*(X;\Q)$ is generated as a ring by~$H^2(X;\Q)$, then 
this system is determined by the codimension~0 classes $\cI_{0,3,B}^{\om}(\mu_1,\mu_2,\mu_3)$ with 
$$B\in H_2(X;\Z)_{\om},\quad \lr{c_1(X,\om),B}\le n\!+\!1, 
\quad |\mu_1|\!+\!|\mu_2|=2n\!+\!2\blr{c_1(X,\om),B}\!-\!2, \quad |\mu_3|=2.$$
\end{prp}

\vspace{.18in}

\noindent
A manifestation of the first statement of this proposition is the renown \sf{WDVV equation} 
for GW-invariants:
\BE{CWDVV_e}\begin{split}
&\sum_{\begin{subarray}{c}[\ell]=I\sqcup J\\ 1,2\in I,\,3,4\in J\end{subarray}}
\!\!\!\sum_{B_1+B_2=B}\sum_{i,j=1}^N
(-1)^{\ve((I,J),\mu)}
\blr{\mu_I,e_i}_{\!0,B_1}^{\!\om} g^{ij} \blr{e_j,\mu_J}_{\!0,B_2}^{\!\om}\\
&\hspace{1in}=
\sum_{\begin{subarray}{c}[\ell]=I\sqcup J\\ 1,3\in I,\,2,4\in J\end{subarray}}
\!\!\!\sum_{B_1+B_2=B}\sum_{i,j=1}^N
(-1)^{\ve((I,J),\mu)}\blr{\mu_I,e_i}_{\!0,B_1}^{\!\om} g^{ij} 
\blr{e_j,\mu_J}_{\!0,B_2}^{\!\om} 
\end{split}\EE
for all $\ell\!\in\!\Z^+$, $B\!\in\!H_2(X;\Z)$, and $\mu\!\in\!H^*(X;\Q)^{\ell}$.
The full collection of relations~\eref{CWDVV_e} is equivalent to the associativity 
of the multiplication in the quantum cohomology of~$(X,\om)$.
For $(X,\om)\!=\!(\P^2,\om_{\FS})$, it is equivalent to Kontsevich's recursion
\cite[Claim~5.2.1]{KM} enumerating rational curves in~$\P^2$.
For $(X,\om)\!=\!(\P^n,\om_{\FS})$, \eref{CWDVV_e} is equivalent to the recursion of 
\cite[Theorem~10.4]{RT} enumerating rational curves in~$\P^n$.\\

\noindent
Proposition~\ref{Crecont_prp1} is an early example of \sf{reconstruction} 
in complex GW-theory.
Another example is Proposition~\ref{CTRR_prp} below that reduces 
the $g\!=\!0$ descendant GW-invariants~\eref{CGWdfn_e} to the \sf{primary} ones,
i.e.~those with $a_i\!=\!0$ for all $i\!\in\![\ell]$.
Its statement for smooth projective varieties is \cite[Corollary~1]{LeePand}.
We give a proof of this proposition assuming instead that 
the strata of $\ov\fM_{0,\ell}(B;J)$ are of the expected dimension.
This proof, motivated by an approach initiated in~\cite{Ionel} and formalized 
in~\cite{g2P2and3,obst},
adapts readily to semi-positive symplectic manifolds via Ruan-Tian's global 
inhomogeneous perturbations and with some technical care to arbitrary symplectic manifolds
via the virtual fundamental class constructions of~\cite{LT,FO}.\\

\noindent
For $\ell\!\in\!\Z^+$, $i,j\!\in\![\ell]$ distinct, and
$B,B_1,B_2\!\in\!H_2(X;\Z)$ with $B_1\!+\!B_2\!=\!B$, we denote~by
\BE{CD2dfn_e} D_{B_1i,B_2j}\subset\ov\fM_{0,\ell}(B;J)\EE
the subspace of maps~$u$ from domains $\Si_1\!\cup\!\Si_2$ such that 
\begin{enumerate}[label=$\bullet$,leftmargin=*]

\item $\Si_1,\Si_2$ are connected genus~0 curves with one point in common,

\item the $i$-th (resp.~$j$-th) marked $z_i$ (resp.~$z_j$) lies on~$\Si_1$ (resp.~$\Si_2$), and

\item the restrictions of~$u$ to~$\Si_1$ and~$\Si_2$ are of degrees~$B_1$ and~$B_2$,
respectively.

\end{enumerate}
Under ideal circumstances, $D_{B_1i,B_2j}$ is a union of smooth divisors in~$\ov\fM_{0,\ell}(B;J)$
intersecting transversely.
By~\cite{LT,FO}, $D_{B_1i,B_2j}$ carries a natural virtual fundamental class.
The meaning of the equation~\eref{CTRR_e} below is that the integral
of the product of the left-hand side with a descendant cohomology class~$\eta$,  
as in the integrand in~\eref{CGWdfn_e}, against $[\ov\fM_{0,\ell}(B;J)]^{\vir}$
equals the integral
of the product of the difference on the right-hand side with~$\eta$ 
against $[\ov\fM_{0,\ell}(B;J)]^{\vir}$ plus
the integral of~$\eta$ against the weighted sum of~$[D_{B_1i,B_2j}]^{\vir}$.

\begin{prp}\label{CTRR_prp}
Let $(X,\om)$ be a compact symplectic manifold.
For $\ell\!\in\!\Z^+$, $B\!\in\!H_2(X;\Z)$, $i,j\!\in\![\ell]$ distinct,
and $\mu\!\in\!H^2(X;\Q)$,
\BE{CTRR_e}\lr{\mu,B}\psi_i=\big(\ev_j^*\mu\!-\!\ev_i^*\mu\big)
+\!\sum_{B_1+B_2=B}\hspace{-.15in}\lr{\mu,B_2}\!\big[D_{B_1i,B_2j}\big]^{\vir}
\in H^2\big(\ov\fM_{0,\ell}(B;J);\Q\big)\,.\EE
\end{prp}

\begin{proof}
We can assume that $\om(B)\!>\!0$.
By the linearity of both sides of~\eref{CTRR_e} in~$\mu$,
we can also assume that $\mu$ can be represented by a (generic) pseudocycle~$M$ in~$X$
as defined in Section~1 in~\cite{pseudo}.
Let $d\!=\!\lr{\mu,B}$ and $\hb\!\in\!\R^+$ be the minimal value of $\lr{\om,u_*[\P^1]}$
for a non-constant $J$-holomorphic map $u\!:\P^1\!\lra\!X$.\\

\noindent
We denote by $\wt\fM$ the set of representatives $u\!:\Si\!\lra\!X$
for the elements of $\ov\fM_{0,\ell}(B;J)$
such~that 
\begin{enumerate}[label=$\bullet$,leftmargin=*]

\item the $i$-th marked point $z_i\!=\!0$ on the irreducible component $\P^1_i\!\subset\!\Si$ 
containing it;

\item if the $j$-th marked point $z_j$ lies on~$\P^1_i$, then $z_j\!=\!\i$;

\item if $z_j$ does not lie on~$\P^1_i$, then the node of~$\P^1_i$ separating
it from the irreducible component of~$\Si$ containing~$z_j$ is the point $\i\!\in\!\P^1_i$;

\item the interior of the closed unit disk $\D\!\subset\!\P_i^1$ centered at $z_i\!=\!0$
contains no marked points other than~$z_i$ and no nodes;

\item either $\int_{\D}u^*\om\!=\!\hb/2$ and $\D\!\subset\!\P_i^1$ 
contains no marked points other than~$z_i$ and no nodes or
$\int_{\D}u^*\om\!<\!\hb/2$ and $\D\!\subset\!\P_i^1$ 
contains a marked point other than~$z_i$ or a node;

\end{enumerate}
see Figure~\ref{wtfM_fig}.\\

\begin{figure}
\begin{pspicture}(-.5,-.2)(10,2)
\psset{unit=.4cm}
\pscircle(8,2.5){2}\pscircle*(8,4.5){.2}\pscircle*(8,.5){.2}
\rput(8.9,5.2){\sm{$z_i\!=\!0$}}\rput(9,-.3){\sm{$z_j\!=\!\i$}}
\rput(8,2.5){$B$}
\pscircle(27,4){1.5}\pscircle(27,1){1.5}\pscircle*(27,2.5){.2}
\pscircle*(27,5.5){.2}\pscircle*(28.5,1){.2}\rput(27,3){\sm{$\i$}}
\rput(27.9,6.2){\sm{$z_i\!=\!0$}}\rput(29.2,.9){\sm{$z_j$}}
\rput(27,4.1){$B_1$}\rput(27,.9){$B_2$}
\end{pspicture}
\caption{A typical element of~$\wt\fM$ and a representative for a typical element 
of~$D_{B_1i,B_2j}$ in~$\wt\fM$.}
\label{wtfM_fig}
\end{figure}

\noindent
We call two elements of $\wt\fM$ \sf{equivalent} if they differ by a reparametrization
of the domain that restricts to the identity on the irreducible component~$\P^1_i$
containing~$z_i$.
Gromov's convergence induces a topology on the set~$\wh\fM$ of the resulting equivalence
classes of elements of~$\wt\fM$.
The action of~$S^1$ on~$\P^1_i$ induces a continuous action on~$\wh\fM$ so~that 
$\ov\fM_{0,\ell}(X,B)$ is the quotient $\wh\fM/S^1$ and
$$L\!\equiv\!\wh\fM\!\times_{S^1}\!\C\lra \ov\fM_{0,\ell}(X,B)$$
is the universal {\it tangent} line bundle for the $i$-th marked point.
Below we define a ``meromorphic" section of~$L^{\otimes d}$ and determine its zero/pole locus.\\

\noindent
A generic element $u\!:\P^1\!\lra\!X$ of $\wh\fM$ intersects $M$ transversely 
at points 
$$y_1^+(u),\ldots,y_{d_+}^+(u)\in\C^* \qquad\hbox{and}\qquad 
y_1^-(u),\ldots,y_{d_-}^-(u)\in\C^*$$
positively and negatively, respectively, so that $d\!=\!d_+\!-\!d_-$.
The resulting element
$$s\big([u]\big)\equiv 
\big[u,y_1^+(u)\ldots y_{d_+}^+(u)\big/y_1^-(u)\ldots y_{d_-}^-(u)\big]
\in L^{\otimes d}$$
depends only on the image $[u]\!\in\!\ov\fM_{0,\ell}(B;J)$ of $u\!\in\!\wh\fM$.
This construction induces a section of~$L^{\otimes d}$ over generic elements
of $\ov\fM_{0,\ell}(B;J)$.
This section extends to a ``meromorphic" section over all of~$\ov\fM_{0,\ell}(B;J)$.\\

\noindent
The section $s$ has a zero (resp.~pole) whenever $u$ meets $M$ positively (resp.~negatively)
at $z_i\!=\!0$ and a pole (resp.~zero) whenever $u$ meets~$M$ positively (resp.~negatively)
at $z_j\!=\!\i$.
It also has a pole (resp.~zero) whenever $u\!\in\!D_{B_1i,B_2j}$ meets~$M$
 at any point of~$\Si_2$ 
(in the terminology around~\eref{CD2dfn_e}) positively (resp.~negatively).
Thus,
$$s^{-1}(0)=\big(\ev_i^{-1}(M)\!-\!\ev_j^{-1}(M)\!\big)
-\!\sum_{B_1+B_2=B}\hspace{-.15in}\lr{\mu,B_2}\!D_{B_1i,B_2j}
\subset \ov\fM_{0,\ell}(B;J)\,.$$
Since $\lr{\mu,B}\psi_i=-c_1(L^{\otimes d})$,
this implies~\eref{CTRR_e}.
\end{proof}

\noindent
It is often convenient to combine GW-invariants into generating functions.
Suppose $\{e_i\}_{i\in[N]}$ is a basis of homogeneous elements for $H^*(X;\Q)$ with $e_1\!=\!1$,
$(g_{ij})_{i,j\in[N]}$ is 
the associated matrix for the intersection form on~$H^*(X;\Q)$, and 
$(g^{ij})_{i,j\in[N]}$ is its inverse as before.
Let $t_i$ and $t_{ai}$ with  $i\!\in\![N]$  and $a\!\in\!\Z^{\ge0}$ be formal variables with
$$t_it_j=(-1)^{|e_i||e_j|}t_jt_i, \qquad t_{ai}t_{bj}=(-1)^{|e_i||e_j|}t_{bj}t_{ai}\,.$$
Define  
\begin{alignat*}{2}
\ft&=\sum_{i=1}^N e_it_i, &\quad
\blr{e_{i_1}t_{i_1},\ldots,e_{i_{\ell}}t_{i_{\ell}}}_{\!g,B}^{\om} 
&=\blr{e_{i_1},\ldots,e_{i_{\ell}}}_{\!g,B}^{\om} t_{i_{\ell}}\!\ldots\!t_{i_1}, \\
\wt\ft&=\sum_{a=0}^{\i}\sum_{i=1}^N \tau_a(e_i)t_{ai},&\quad
\blr{\tau_{a_1}(e_{i_1})t_{a_1i_1},\ldots,\tau_{a_{\ell}}(e_{i_{\ell}})
t_{a_{\ell}i_{\ell}}}_{\!g,B}^{\om} 
&=\blr{\tau_{a_1}(e_{i_1}),\ldots,
\tau_{a_{\ell}}(e_{i_{\ell}})}_{\!g,B}^{\om} t_{a_{\ell}i_{\ell}}\!\ldots\!t_{a_1i_1}\,.
\end{alignat*} 
The (\sf{primary}) \sf{genus~0 GW-potential} and 
the \sf{descendant genus~g GW-potential} of $(X,\om)$ are the formal power series
$$\Phi^{\om}(\ft,q) =\sum_{B\in H_2(X)}\sum_{\ell\ge 0} \frac{1}{\ell!} 
\blr{\underset{\ell}{\underbrace{\ft,\ldots,\ft}}}_{\!0,B}^{\om} q^B
\quad\hbox{and}\quad
F_g^{\om}(\wt\ft,q) =\sum_{B\in H_2(X;\Z)}\sum_{\ell\ge0} \frac{1}{\ell!} 
\blr{\underset{\ell}{\underbrace{\wt\ft,\ldots,\wt\ft}}}_{\!g,B}^{\om} q^B\,,$$
respectively.
The \sf{full descendant GW-potential} is the formal power~series
$$F^{\om}(\wt\ft,q) = \sum_{g\ge0}F_g^{\om}(\wt\ft,q)\la^{2g-2}.$$
Thus, $\Phi^{\om}$ is the coefficient of $\la^{-2}$ in $F^{\om}$ with $t_{0i}\!=\!t_i$ 
and $t_{ai}\!=\!0$ for $a\!>\!0$.\\

\noindent
In light of~\ref{Ceff1_it}, \ref{Cgrad1_it}, and~\eref{g0basic_e},
\ref{Cstr_it} is equivalent to the \sf{string~differential equation}
\BE{CstrODE_e}\frac{\prt F^{\om}}{\prt t_{01}} 
= \frac{1}{2}\la^{-2}\!\!\!\!\!\!\sum_{1\le i,j\le N}\!\!\!\!\!g_{ij}t_{0j}t_{0i} 
+\sum_{a=0}^\infty \sum_{i=1}^N t_{(a+1)i}\frac{\prt F^{\om}}{\prt t_{ai}}\,.\EE
In light of~\ref{Ceff1_it}, \ref{Cgrad1_it}, and~\eref{g1basic_e},
\ref{Cdil_it} is equivalent to  the \sf{dilaton~differential equation}
$$\frac{\prt F^{\om}}{\prt t_{11}}= \frac{\chi(X)}{24}+
\bigg\{\la\frac{\prt }{\prt\la} +  \sum_{a=0}^\infty \sum_{i=1}^N t_{ai}
\frac{\prt}{\prt t_{ai}}\bigg\} F^{\om} \,.$$
The relations~\eref{CWDVV_e} are equivalent to the \sf{WDVV differential equations}
\begin{equation*}\begin{split}
&\sum_{j,k=1}^N\!\prt_{t_{i_1}}\prt_{t_{i_2}}\prt_{t_j}\Phi^{\om}\cdot g^{jk}
\prt_{t_k}\prt_{t_{i_3}}\prt_{t_{i_4}}\Phi^{\om} \\
&\hspace{1.5in}=(-1)^{|e_{i_1}|(|e_{i_2}|+|e_{i_3}|)}
\sum_{j,k=1}^N\!\prt_{t_{i_2}}\prt_{t_{i_3}}\prt_{t_j}\Phi^{\om}\cdot g^{jk} 
\prt_{t_k}\prt_{t_{i_1}}\prt_{t_{i_4}}\Phi^{\om}
\end{split}\end{equation*}
with $i_1,i_2,i_3,i_4\!=\!1,2,\ldots,N$;
see \cite[(4.13)]{KM}.

\section{Real Gromov-Witten theory}
\label{RGWs_sec}

\subsection{Moduli spaces}
\label{RMS_subs}

\noindent
A \sf{symmetric Riemann surface} $(\Si,\si,\fj)$  is a closed, but possibly nodal and disconnected,
Riemann surface~$(\Si,\fj)$ with an anti-holomorphic involution~$\si$. 
For example, there are two topological types of such involutions on~$\P^1$:
\BE{tauetadfn_e}\tau,\eta\!:\P^1\lra\P^1, \qquad \tau(z)=\frac{1}{\bar{z}}, \quad
\eta(z)=-\frac{1}{\bar{z}}\,.\EE
If $\phi$ is an involution on another space~$X$,
a map $u\!:\!\Si\!\lra\!X$ is \sf{$\phi$-real} (or just \sf{real}) 
if $u\!\circ\!\si\!=\!\phi\!\circ\!u$.\\

\noindent
Let $g\!\in\!\Z$ and $\ell\!\in\!\Z^{\ge0}$.
We denote by $\R\ov\cM_{g,\ell}$ the Deligne-Mumford moduli space
of stable closed connected, but possibly nodal, 
symmetric Riemann surfaces~$(\Si,\si,\fj)$ of arithmetic genus~$g$
with $\ell$~conjugate pairs $(z_i^+,z_i^-)$ of marked points.
This space is a smooth compact orbifold of dimension
\BE{RcMdim_e}\dim\R\ov\cM_{g,\ell}=3(g\!-\!1)\!+\!2\ell\,;\EE
it is empty if $g\!<\!0$ or $g\!+\!\ell\!\le\!1$.
The symmetric group~$\bS_{\ell}$ acts on~$\R\ov\cM_{g,\ell}$ by 
smooth automorphisms permuting pairs of conjugate marked points analogously 
to~\eref{bSellact_e}. 
For each $i\!\in\![\ell]$, let 
$$\th_i\!:\R\ov\cM_{g,\ell}\lra\R\ov\cM_{g,\ell}$$
be the automorphism interchanging the marked points
in the $i$-th conjugate pair.\\

\noindent
Similarly to the complex case, let
$$(\bE,\vph_{\bE})\lra \R\ov\cM_{g,\ell}$$
be the Hodge vector bundle of holomorphic differentials 
and the conjugation induced by the real structure on each surface.
We denote~by
\BE{bERdfn_e}\bE^{\R}\lra \R\ov\cM_{g,\ell}\EE
the $\vph_{\bE}$-fixed part of~$\bE$;
this is a real vector orbi-bundle of rank~$g$.
If $\ell\!>\!0$ and $g\!+\!\ell\!\ge\!2$, let
$$\ff_{\ell}^{\R}\!: \R\ov\cM_{g,\ell}\lra \R\ov\cM_{g,\ell-1}$$
be the natural forgetful morphism dropping the last conjugate pair of marked points.
The open subspace $\cU'\!\subset\!\R\ov\cM_{g,\ell}$ of the marked curves~$[\cC]$
such that the irreducible component of~$\cC$ carrying the marked point~$z_{\ell}^+$
also carries either 
\begin{enumerate}[label=$\bullet$,leftmargin=*]

\item at least three nodes and/or marked points $z_i^{\pm}$ with $i\!\in\![\ell\!-\!1]$ or

\item precisely one node and at least one other marked point

\end{enumerate}
satisfies the codimension conditions below~\eref{surjcond_e}.
We orient the kernel of $\nd\ff_{\ell}^{\R}|_{\cU'}$ by the position 
of the marked point~$z_{\ell}^+$ in the fiber (which is a nodal Riemann surface).
Similarly to~\eref{bEffl_e},
\BE{RbEffl_e} 
(\bE,\vph_{\bE})=\big\{\ff_{\ell}^{\R}\big\}\!^*(\bE,\vph_{\bE})\lra \R\ov\cM_{g,\ell}\EE
under the above assumptions.\\

\noindent
We denote by
\BE{Riogldfn_e}\io_{g,\ell}^{\C}\!:\R\ov\cM_{g-2,\ell+2}\lra\R\ov\cM_{g,\ell}
\quad\hbox{and}\quad
\io_{g,\ell}^E\!:\R\ov\cM_{g-1,\ell+1}\lra\R\ov\cM_{g,\ell}\EE
the immersion obtained by identifying the marked points~$z_{\ell+1}^+$ and~$z_{\ell+1}^-$ 
of each curve in the domain with~$z_{\ell+2}^+$ and~$z_{\ell+2}^-$, respectively,
to form a conjugate pair of nodes and
the immersion obtained by identifying the marked point~$z_{\ell+1}^+$
of each curve in the domain with~$z_{\ell+1}^-$ to form 
an isolated real node, called $E$-node in~\cite{RealGWsI} and elsewhere.
The first immersion is generically $4\!:\!1$ onto its image, while
the second is generically $2\!:\!1$.
For each element $P\!\in\!\cP(g\!-\!1,\ell)$
as in~\eref{PwtPdfn_e}, let
\BE{RioPdfn_e}\io_P^{\C}\!: 
\ov\cM_P^{\C}\!\equiv\!\R\ov\cM_{g_1,|I|+1}\!\times\!\R\ov\cM_{g_2,|J|+1}\lra \R\ov\cM_{g,\ell}\EE
be the immersion obtained by identifying the marked points~$z_{|I|+1}^+$ and~$z_{|I|+1}^-$
of the first Riemann surface in the domain with 
the marked points~$z_{|J|+1}^+$ and~$z_{|J|+1}^-$ of the second Riemann surface in the domain
to form a conjugate pair of nodes and by re-ordering the remaining pairs of marked points 
according to the bijection~\eref{Sperm_e}.
These three immersions are illustrated by the first three diagrams in Figure~\ref{Rimmers_fig}.\\

\begin{figure}
\begin{pspicture}(-.5,-6.5)(10,2)
\psset{unit=.4cm}
\psarc[linewidth=.06](4,2.5){2}{30}{330}\pscircle*(2.59,3.91){.2}\pscircle*(2.59,1.09){.2}
\pscircle*(5,4.23){.2}\pscircle*(5,.77){.2}
\rput(1.4,4.1){\sm{$z_{\ell+1}^+$}}\rput(1.6,1.1){\sm{$z_{\ell+1}^-$}}
\rput(6.2,4.4){\sm{$z_{\ell+2}^+$}}\rput(6.1,.55){\sm{$z_{\ell+2}^-$}}
\psline[linewidth=.03]{<->}(4,1)(4,4)\rput(4.5,2.5){\sm{$\si$}}
\psline[linewidth=.03]{->}(7,2.2)(11,2.2)\rput(9,2.9){\sm{$\io_{g,\ell}^{\C}$}}
\psarc[linewidth=.06](14,2.5){2}{135}{225}\pscircle*(12.59,3.91){.2}\pscircle*(12.59,1.09){.2}
\psline[linewidth=.06](12.59,3.91)(13.09,4.41)\psline[linewidth=.06](12.59,1.09)(13.09,.59)
\psline[linewidth=.06](12,4.5)(13.09,3.41)\psline[linewidth=.06](12,.5)(13.09,1.59)
\psarc[linewidth=.06](13.59,3.91){.7}{-135}{135}\psarc[linewidth=.06](13.59,1.09){.7}{-135}{135}
\psline[linewidth=.03]{<->}(14.8,1)(14.8,4)\rput(15.3,2.5){\sm{$\si$}}
\psarc[linewidth=.06](24,2.5){2}{90}{270}\pscircle*(24,4.5){.2}\pscircle*(24,.5){.2}
\psline[linewidth=.03](24,4.5)(25.8,4.5)\psline[linewidth=.03](24,.5)(25.8,.5)
\rput(24.2,5.3){\sm{$z_{\ell+1}^+$}}\rput(24.2,-.3){\sm{$z_{\ell+1}^-$}}
\psline[linewidth=.03]{<->}(24,1)(24,4)\rput(24.5,2.5){\sm{$\si$}}
\psline[linewidth=.03]{->}(27,2.2)(31,2.2)\rput(29,2.9){\sm{$\io_{g,\ell}^E$}}
\psarc[linewidth=.06](33.5,2.5){1.5}{45}{315}
\psline[linewidth=.03](34.56,3.56)(36.6,1.52)\psline[linewidth=.03](34.56,1.44)(36.6,3.48)
\pscircle*(35.62,2.5){.2}
\psline[linewidth=.03]{<->}(33.5,1.3)(33.5,3.7)\rput(34,2.5){\sm{$\si$}}
\psarc[linewidth=.06](1.5,-4.5){2}{-90}{90}\psarc[linewidth=.06](7,-4.5){2}{90}{270}
\pscircle*(2.91,-3.09){.2}\pscircle*(2.91,-5.91){.2}
\pscircle*(5.59,-3.09){.2}\pscircle*(5.59,-5.91){.2}
\rput(1.7,-3.45){\sm{$z_{|I|+1}^+$}}\rput(1.5,-5.45){\sm{$z_{|I|+1}^-$}}
\rput(7,-3.45){\sm{$z_{|J|+1}^+$}}\rput(7.1,-5.65){\sm{$z_{|J|+1}^-$}} 
\psline[linewidth=.03]{<->}(4,-6)(4,-3)\rput(4.5,-4.5){\sm{$\si$}}
\psline[linewidth=.03]{->}(7.5,-4.8)(11,-4.8)\rput(9.5,-4.1){\sm{$\io_P^{\C}$}}
\psarc[linewidth=.06](11.5,-4.5){2}{-90}{90}\psarc[linewidth=.06](14.32,-4.5){2}{90}{270}
\pscircle*(12.91,-3.09){.2}\pscircle*(12.91,-5.91){.2}
\psline[linewidth=.03]{<->}(14.8,-6)(14.8,-3)\rput(15.3,-4.5){\sm{$\si$}}
\rput(11.8,-4.6){\sm{$g_2$}}\rput(14.1,-4.6){\sm{$g_1$}}
\pscircle*(12.02,-2.57){.2}\pscircle*(12.02,-6.43){.2}   
\pscircle*(13.8,-2.57){.2}\pscircle*(13.9,-6.43){.2}
\rput(12.5,-2){\sm{$z_i^+$}}\rput(12.4,-7.2){\sm{$z_i^-$}}
\rput(14.1,-1.8){\sm{$z_j^+$}}\rput(14.1,-7.2){\sm{$z_j^-$}}
\rput(13.1,-8.5){\sm{$i\!\in\!I,\,j\!\in\!J$}}
\psarc[linewidth=.06](24,-2.5){2}{180}{360}\pscircle*(24,-4.5){.2}
\rput(24.4,-5.2){\sm{$z_{\ell+1}$}}
\pscircle*(22.27,-3.5){.2}\pscircle*(25.73,-3.5){.2}
\psline[linewidth=.03]{->}(27,-4.8)(31,-4.8)\rput(29,-4.1){\sm{$\io_{g,(I,J)}^E$}}
\psarc[linewidth=.06](33.5,-2.5){2}{180}{360}\pscircle*(33.5,-4.5){.2}
\psarc[linewidth=.06](33.5,-6.5){2}{0}{180}
\psline[linewidth=.03]{<->}(36,-6)(36,-3)\rput(36.5,-4.5){\sm{$\si$}}
\pscircle*(31.77,-3.5){.2}\pscircle*(31.77,-5.5){.2}
\pscircle*(35.23,-3.5){.2}\pscircle*(35.23,-5.5){.2}
\rput(32.7,-3.3){\sm{$z_i^+$}}\rput(34.5,-3.3){\sm{$z_j^-$}}
\rput(32.6,-5.6){\sm{$z_i^-$}}\rput(34.4,-5.7){\sm{$z_j^+$}}
\rput(33.5,-7.5){\sm{$i\!\in\!I,\,j\!\in\!J$}}
\psarc[linewidth=.06](6.5,-12.5){2}{90}{270}
\pscircle*(5.09,-11.09){.2}\pscircle*(5.09,-13.91){.2}
\psline[linewidth=.03]{<->}(6.5,-14)(6.5,-11)\rput(6,-12.5){\sm{$\si$}}
\psline[linewidth=.06](4,-12.59)(0,-10.59)\pscircle*(3,-12.09){.2}
\pscircle*(2,-11.59){.2}\pscircle*(1,-11.09){.2}
\rput(3.8,-10.3){\sm{$z_{|K|+1}^+$}}\rput(4,-14.5){\sm{$z_{|K|+1}^-$}}
\rput(1.8,-12.9){\sm{$z_{|I\sqcup J|+1}$}}
\psline[linewidth=.03]{->}(7.5,-12.8)(10.5,-12.8)\rput(9,-12){\sm{$\io_{\wt{P}}^{\C}$}}
\psarc[linewidth=.06](15.5,-12.5){2}{90}{270}
\psline[linewidth=.03]{<->}(15.5,-14)(15.5,-11)\rput(15,-12.5){\sm{$\si$}}
\psline[linewidth=.06](14.77,-12)(10.77,-10)\psline[linewidth=.06](14.77,-13)(10.77,-15)
\pscircle*(14.5,-10.77){.2}\pscircle*(14.5,-14.23){.2}
\pscircle*(13.77,-11.5){.2}\pscircle*(12.77,-11){.2}\pscircle*(11.77,-10.5){.2}
\pscircle*(13.77,-13.5){.2}\pscircle*(12.77,-14){.2}\pscircle*(11.77,-14.5){.2}
\rput(14,-10.1){\sm{$z_k^+$}}\rput(13.9,-14.9){\sm{$z_k^-$}}
\rput(12.3,-11.7){\sm{$z_i^+$}}\rput(12.4,-13.3){\sm{$z_i^-$}}
\rput(11.6,-9.6){\sm{$z_j^-$}}\rput(11.7,-15.3){\sm{$z_j^+$}}
\rput(10.4,-9.8){\sm{$g'$}}\rput(10.4,-15){\sm{$g'$}}\rput(16,-10.3){\sm{$g_0$}}
\rput(18,-12.5){\sm{\begin{tabular}{l}$i\!\in\!I$\\ $j\!\in\!J$\\ $k\!\in\!K$\end{tabular}}}
\psarc[linewidth=.06](24,-10.5){2}{180}{360}
\rput(21.6,-12.1){\sm{$z_{\ell+1}$}}\rput(24.8,-11){\sm{$z_{\ell+2}$}}
\pscircle*(22.27,-11.5){.2}\pscircle*(25.73,-11.5){.2}
\psline[linewidth=.03]{->}(27,-12.8)(31,-12.8)\rput(29,-12.1){\sm{$\io_{g,(I,J)}^{\C}$}}
\psarc[linewidth=.06](33.5,-11.09){2}{180}{360}\psarc[linewidth=.06](33.5,-13.91){2}{0}{180}
\pscircle*(32.09,-12.5){.2}\pscircle*(34.91,-12.5){.2}
\psline[linewidth=.03]{<->}(36,-13.5)(38,-11.5)\psline[linewidth=.03]{<->}(38,-13.5)(36,-11.5)
\rput(37.1,-12){\sm{$\si$}}
\pscircle*(31.57,-11.61){.2}\pscircle*(31.57,-13.39){.2}    
\pscircle*(35.43,-11.61){.2}\pscircle*(35.43,-13.39){.2}
\rput(32.4,-11.2){\sm{$z_i^+$}}\rput(34.8,-11.2){\sm{$z_j^-$}}
\rput(32.4,-13.7){\sm{$z_j^+$}}\rput(34.6,-13.7){\sm{$z_i^-$}}
\rput(33.5,-15.5){\sm{$i\!\in\!I,\,j\!\in\!J$}}
\end{pspicture}
\caption{Typical elements in the domains and images of the immersions
\eref{Riogldfn_e}-\eref{RioIJEdfn_e},
with the genus and marked points of each irreducible component of an image of
\eref{RioPdfn_e}-\eref{RioIJEdfn_e}  indicated next to~it.}
\label{Rimmers_fig}
\end{figure}

\noindent
For a Riemann surface $(\Si,\fj)$, let $\fD(\Si,\fj)\!\equiv\!(\wh\Si,\wh\fj,\si)$
be the symmetric Riemann surface given~by
$$\wh\Si=\{+,-\}\!\times\!\Si, \quad \wh\fj|_{\{\pm\}\times\Si}=\pm\fj,\quad
\si(\pm,z)=(\mp,z)~~\forall\,z\!\in\!\Si.$$
For disjoint subsets $I,J\!\subset\!\Z^+$ and a marked curve
$\cC\!\equiv\!(\Si,\fj,(z_i)_{i\in[|I\sqcup J|]})$, define
\BE{fDIJdfn_e1}\fD_{I,J}(\cC)=\big(\fD(\Si,\fj),(z_i^+,z_i^-)_{i\in[|I\sqcup J|]}\big)
\quad\hbox{with}~~
z_i^{\pm}=\begin{cases}(\pm,z_i),&\hbox{if}~i\!\in\!\io_{I,J}^{~-1}(I);\\
(\mp,z_i),&\hbox{if}~i\!\in\!\io_{I,J}^{~-1}(J).\end{cases}\EE
For each $\wt{P}\!\in\!\wt\cP(g,\ell)$ as in~\eref{PwtPdfn_e}, let
\BE{RiowtPdfn_e}\io_{\wt{P}}^{\C}\!: 
\ov\cM_{\wt{P}}^{\C}\!\equiv\!\ov\cM_{g',|I\sqcup J|+1}\!\times\!\R\ov\cM_{g_0,|K|+1}
\lra \R\ov\cM_{g,\ell}\EE
be the immersion sending $([\cC'],[\cC_0])$ to the equivalence class of 
the marked curve obtained by identifying the marked points
the marked points~$z_{|I\sqcup J|+1}^+$ and~$z_{|I\sqcup J|+1}^-$ of~$\fD_{I\sqcup\{\ell+1\},J}(\cC')$
with the marked points~$z_{|K|+1}^+$ and~$z_{|K|+1}^-$ of~$\cC_0$ 
to form a conjugate pair of nodes and by re-ordering the remaining pairs of marked points 
according to the bijection~\eref{Sperm_e} with~$(I,J)$ replaced by~$(I\!\sqcup\!J,K)$.
If \hbox{$[\ell]\!=\!I\!\sqcup\!J$}, let
\BE{RioIJEdfn_e}\io_{g,(I,J)}^{\C}\!: \ov\cM_{(g-1)/2,\ell+2}\lra  \R\ov\cM_{g,\ell}
\quad\hbox{and}\quad
\io_{g,(I,J)}^E\!: \ov\cM_{g/2,\ell+1}\lra  \R\ov\cM_{g,\ell}\EE
be  the immersion sending $[\cC']$ to the equivalence class of the marked curve obtained
by identifying the marked points~$z_{\ell+1}^+$ and~$z_{\ell+2}^+$
of~$\fD_{I\sqcup\{\ell+1\},J\sqcup\{\ell+2\}}(\cC')$
with~$z_{\ell+1}^-$ and~$z_{\ell+2}^-$, respectively, to form
a conjugate pair of nodes and
 the immersion sending $[\cC']$ to the equivalence class of the marked curve obtained
by identifying the marked point~$z_{\ell+1}^+$ of~$\fD_{I\sqcup\{\ell+1\},J}(\cC')$
with~$z_{\ell+1}^-$ to form an $E$-node.
These three immersions are illustrated by the last three diagrams in Figure~\ref{Rimmers_fig}.\\

\noindent
For each $i\!\in\![\ell]$, let 
$$\cL_i\lra \ov\cM_{g,\ell}\qquad\hbox{and}\qquad \cL_i\lra \R\ov\cM_{g,\ell}$$
be the universal tangent complex line orbi-bundles at the marked points~$z_i$ and~$z_i^+$, 
respectively.
In either case, let
\BE{cLRdfn_e}\big(\cL_i\!\otimes\!\ov{\cL_i}\big)^{\R}=
\big\{rv\!\otimes_{\C}\!v\!\in\!\cL_i\!\otimes_{\C}\!\ov{\cL_i}\!:
v\!\in\!\cL_i,\,r\!\in\!\R\big\}.\EE
This real line bundle is canonically oriented by the standard orientation of~$\R$.
There are canonical isomorphisms
\BE{Riogldisom_e}\cN\io_{g,\ell}^{\C}\approx\cL_{\ell+1}\!\otimes\!\cL_{\ell+2}
\lra\R\ov\cM_{g-2,\ell+2}  ~~\hbox{and}~~
\cN\io_{g,\ell}^E\approx \big(\cL_{\ell+1}\!\otimes\!\ov{\cL_{\ell+1}}\big)^{\R}
\lra\R\ov\cM_{g-1,\ell+1}  \EE
of the normal bundles to the immersions in~\eref{Riogldfn_e}.
With the assumptions as in~\eref{RioPdfn_e}-\eref{RioIJEdfn_e},
\BE{RioPprp_e}\begin{aligned}
\cN\io_P^{\C}&\approx\pi_1^*\cL_{|I|+1}\!\otimes\!\pi_2^*\cL_{|J|+1}
\lra\ov\cM_P^{\C},
&\cN\io_{g,(I,J)}^{\C}&\approx\cL_{\ell+1}\!\otimes\!\ov{\cL_{\ell+2}}
 \lra\ov\cM_{(g-1)/2,\ell+1},\\
\cN\io_{\wt{P}}^{\C}&\approx\pi_1^*\cL_{|I\sqcup J|+1}\!\otimes\!\pi_2^*\cL_{|K|+1}
\lra\ov\cM_{\wt{P}}^{\C},
&\cN\io_{g,(I,J)}^E&\approx \big(\cL_{\ell+1}\!\otimes\!\ov{\cL_{\ell+1}}\big)^{\R}
 \lra\ov\cM_{g/2,\ell+1},
\end{aligned}\EE
where $\pi_1,\pi_2$ are the component projections from the domains in~\eref{RioPdfn_e}
and~\eref{RiowtPdfn_e}.
Via the isomorphisms in~\eref{Riogldisom_e} and~\eref{RioPprp_e},
the normal bundles to the six immersions in~\eref{Riogldfn_e}-\eref{RioIJEdfn_e} 
inherit orientations from the complex orientation of~$\cL_i$ and 
the canonical orientation of~$(\cL_{\ell+1}\!\otimes\!\ov{\cL_{\ell+1}})^{\R}$.\\

\noindent
Let $(X,\om,\phi)$ be a compact real symplectic manifold of real dimension~$2n$ 
with $n\!\not\in\!2\Z$.
Define
\begin{gather*}
H_2(X;\Z)_{\om}^{\phi}=\big\{B\!\in\!H_2(X;\Z)_{\om}\!:\phi_*(B)\!=\!-B\big\},
\quad 
H^*_{\pm}(X;\Q)=\big\{\mu\!\in\!H^*(X;\Q)\!:\phi^*\mu\!=\!\pm\!\mu\big\},\\
\cJ_{\om}^{\phi}=\big\{J\!\in\!\cJ_{\om}\!:\phi^*J\!=\!-J\big\}.
\end{gather*}
For $g,\ell\!\in\!\Z^{\ge0}$, $B\!\in\!H_2(X;\Z)$,  and $J\!\in\!\cJ_{\om}^{\phi}$,
we denote by $\ov\fM_{g,\ell}^{\phi}(B;J)$ 
the moduli space of stable real $J$-holomorphic degree~$B$ maps 
from closed connected, but possibly nodal, symmetric Riemann surfaces of arithmetic genus~$g$
with $\ell$~conjugate pairs of marked points.
This space is empty if \hbox{$B\!\not\in\!H_2(X;\Z)_{\om}^{\phi}$} or
$B\!=\!0$ and $g\!+\!\ell\!\le\!1$.\\

\noindent
For each $i\!=\!1,\ldots,\ell$, let 
$$\ev_i\!:\ov\fM_{g,\ell}^{\phi}(B;J)\lra X \qquad\hbox{and}\qquad
\psi_i\in H^2\big(\ov\fM_{g,\ell}^{\phi}(B;J);\Q\big)$$
be the natural evaluation map at the marked point~$z_i^+$ and
the Chern class of the universal cotangent line bundle for this marked point, 
respectively.
Let 
$$\wt\th_i\!:\ov\fM_{g,\ell}^{\phi}(B;J)\lra\ov\fM_{g,\ell}^{\phi}(B;J)$$
be the automorphism interchanging the marked points in the $i$-th conjugate pair.
It satisfies
\BE{wtthiprp_e} \ev_j\!\circ\!\wt\th_i=\begin{cases}
\phi\!\circ\!\ev_i,&\hbox{if}~j\!=\!i\,;\\
\ev_j,&\hbox{if}~j\!\neq\!i\,;\end{cases}
\qquad\hbox{and}\qquad
\wt\th_i^*\psi_j=\begin{cases}-\psi_i,&\hbox{if}~j\!=\!i\,;\\
\psi_j,&\hbox{if}~j\!\neq\!i\,.\end{cases}\EE
The symmetric group~$\bS_{\ell}$ acts on~$\ov\fM_{g,\ell}^{\phi}(B;J)$
by permuting pairs of conjugate marked points analogously to~\eref{bSellact_e}.
This action again satisfies~\eref{bSellact_e2b}.\\

\noindent
By \cite[Theorem~1.4]{RealGWsI}, a real orientation on~$(X,\om,\phi)$ endows
the moduli space $\ov\fM_{g,\ell}^{\phi}(B;J)$ with a virtual fundamental class
of dimension/degree 
\BE{RfMdim_e}\begin{split}
\dim\big[\ov\fM_{g,\ell}^{\phi}(B;J)\big]^{\vir}
&=(1\!-\!g)(n\!-\!3)\!+\!2\ell\!+\!\blr{c_1(X,\om),B}\\
&=\dim\R\ov\cM_{g,\ell}
\!+\!\big(n(1\!-\!g)\!+\!\blr{c_1(X,\om),B}\!\big)\in 2\Z\,;
\end{split}\EE
see also \cite[Section~3.3]{RealGWsGeom}.
This class is preserved by the $\bS_{\ell}$-action.
For \hbox{$a_1,\ldots,a_{\ell}\!\in\!\Z^{\ge0}$} and 
\hbox{$\mu_1,\ldots,\mu_{\ell}\!\in\!H^*(X;\Q)$},
let 
\BE{RGWdfn_e}
\blr{\tau_{a_1}(\mu_1),\ldots,\tau_{a_{\ell}}(\mu_{\ell})}_{\!g,B}^{\!\om,\phi} 
=\int_{[\ov\fM_{g,\ell}^{\phi}(B;J)]^{\vir}}\!\! 
\psi_1^{a_1}\!\big(\ev_1^*\mu_1\big)\ldots\psi_{\ell}^{a_{\ell}}\!\big(\ev_{\ell}^*\mu_{\ell}\big)\EE
be the associated real \sf{descendant GW-invariant}.
This number is independent of the choice of~$J\!\in\!\cJ_{\om}^{\phi}$.\\

\noindent
If $g\!+\!\ell\!\ge\!2$, let
\BE{Rffdfn_e}\ff:\ov\fM_{g,\ell}^{\phi}(B;J)\lra \R\ov\cM_{g,\ell}\EE
the natural forgetful morphism to the corresponding Deligne-Mumford moduli space.
It satisfies
\BE{Rffvp_e}\ff\!\circ\!\vp\!=\!\vp\!\circ\!\ff,\ff\!\circ\!\wt\th_i\!=\!\th_i\!\circ\!\ff:
\ov\fM_{g,\ell}^{\phi}(B;J)\lra \R\ov\cM_{g,\ell} \qquad\forall\,
\vp\!\in\!\bS_{\ell},\,i\!\in\![\ell].\EE
We denote by
$$\pi_{\R\ov\cM_{g,\ell}},\pi_{X^{\ell}}\!:
\R\ov\cM_{g,\ell}\!\times\!X^{\ell}\lra\R\ov\cM_{g,\ell},X^{\ell}$$
the component projection maps.
Using Poincar\'e Duality on~$\R\ov\cM_{g,\ell}$ and $\R\ov\cM_{g,\ell}\!\times\!X^{\ell}$,
we define 
\BE{RIdfn_e}
\cI_{g,\ell,B}^{\om,\phi}\!: 
H^*(X;\Q)^{\otimes\ell}\lra \wt{H}^*\big(\R\ov\cM_{g,\ell};\Q\big)
\quad\hbox{and}\quad
C_{g,\ell,B}^{\om,\phi}\in \wt{H}^*\big(\R\ov\cM_{g,\ell}\!\times\!X^{\ell};\Q\big)\EE
by requiring that 
\begin{gather*}
\int_{\R\ov\cM_{g,\ell}}\!\!\ga
\cI_{g,\ell,B}^{\om,\phi}\big(\mu_1,\ldots,\mu_{\ell}\big)=
\int_{[\ov\fM_{g,\ell}^{\phi}(B;J)]^{\vir}}\!\! (\ff^*\ga\big)
\big(\ev_1^*\mu_1\big)\ldots\!\big(\ev_{\ell}^*\mu_{\ell}\big) \quad\hbox{and}\\
\int_{\R\ov\cM_{g,\ell}\times X^{\ell}}\!\! 
\big(\pi_{\R\ov\cM_{g,\ell}}^*\!\ga\big)C_{g,\ell,B}^{\om,\phi}
\big(\pi_{X^{\ell}}^*(\mu_1\!\times\!\ldots\!\times\!\mu_{\ell})\!\big)
=\int_{[\ov\fM_{g,\ell}^{\phi}(B;J)]^{\vir}}\!\!(\ff^*\ga\big) 
\big(\ev_1^*\mu_1\big)\ldots\!\big(\ev_{\ell}^*\mu_{\ell}\big)
\end{gather*}
for all $\mu_i\!\in\!H^*(X;\Q)$  and $\ga\!\in\!H^*(\R\ov\cM_{g,\ell};\Q)$.
The linear maps~$\cI_{g,\ell,B}^{\om,\phi}$ and 
the correspondences $C_{g,\ell,B}^{\om,\phi}$ in~\eref{RIdfn_e} 
are independent of the choice of \hbox{$J\!\in\!\cJ_{\om}^{\phi}$}.
If $g\!<\!0$ or $g\!+\!\ell\!\le\!1$, we set $\cI_{g,\ell,B}^{\om,\phi}\!=\!0$.

\subsection{Properties of invariants I}
\label{RGWs_subs1}

\noindent
The real descendant GW-invariants~\eref{RGWdfn_e} satisfy the following properties:
\begin{enumerate}[label=$\R{\arabic*}\!\!$,ref=$\R{\arabic*}$,leftmargin=*]

\item\label{Reff1_it} ({\it Effectivity I}):
$\lr{\tau_{a_1}(\mu_1),\ldots,\tau_{a_{\ell}}(\mu_{\ell})\!}_{\!g,B}^{\om,\phi}\!=\!0$ 
if $B\!\not\in\!H_2(X;\Z)_{\om}^{\phi}$, or $B\!=\!0$ and $g\!+\!{\ell}\!\le\!1$,
or \hbox{$\mu_i\!\in\!H^*_{(-1)^{a_i}}\!(X;\Q)$} for some $i\!\in\![\ell]$;

\item\label{Rgrad1_it} ({\it Grading I}): 
$\lr{\tau_{a_1}(\mu_1),\ldots,\tau_{a_{\ell}}(\mu_{\ell})\!}_{\!g,B}^{\om,\phi}\!=\!0$ if
$$\sum_{i=1}^{\ell}\big(2a_i\!+\!|\mu_i|\big)\neq 
(1\!-\!g)(n\!-\!3)+2\ell+\blr{c_1(X,\om),B};$$

\item\label{Rstr_it}  ({\it String}): 
$\lr{\tau_{a_1}(\mu_1),\ldots,\tau_{a_{\ell}}(\mu_{\ell}),\tau_0(1)\!}_{\!g,B}^{\om,\phi}\!=\!0$;

\item\label{Rdil_it}  ({\it Dilaton}): 
$\blr{\tau_{a_1}(\mu_1),\ldots,\tau_{a_{\ell}}(\mu_{\ell}),\tau_1(1)\!}_{\!g,B}^{\om,\phi}
=2(g\!-\!1\!+\!{\ell})\lr{\tau_{a_1}(\mu_1),
\ldots,\tau_{a_{\ell}}(\mu_{\ell})\!}_{\!g,B}^{\om,\phi}$;

\item\label{RdivII_it}  ({\it Divisor I}): if $\mu_{\ell+1}\!\in\!H_-^2(X;\Q)$,
\begin{equation*}\begin{split}
&{}\hspace{-.5in}
\blr{\tau_{a_1}(\mu_1),\ldots,\tau_{a_{\ell}}(\mu_{\ell}),\tau_0(\mu_{\ell+1})\!}_{\!g,B}^{\om,\phi} 
=\lr{\mu_{\ell+1},B}\blr{\tau_{a_1}(\mu_1),\ldots,\tau_{a_{\ell}}(\mu_{\ell})\!}_{\!g,B}^{\om,\phi}\\
&\qquad+2\!\!\sum_{\begin{subarray}{c}1\le i\le\ell\\ a_i>0\end{subarray}}
\!\!\!\blr{\tau_{a_1}(\mu_1),\ldots,\tau_{a_{i-1}}(\mu_{i-1}),
\tau_{a_i-1}(\mu_i\mu_{\ell+1}),\tau_{a_{i+1}}(\mu_{i+1}),\ldots, 
\tau_{a_{\ell}}(\mu_{\ell})\!}_{\!g,B}^{\om,\phi}.
\end{split}\end{equation*}

\setcounter{saveenumi}{\arabic{enumi}}
\end{enumerate}
The {\it Effectivity} properties above in the first two cases and in Section~\ref{RGWs_subs2} 
follow immediately 
from the moduli space $\ov\fM_{g,\ell}^{\phi}(B;J)$ being empty 
if either \hbox{$B\!\not\in\!H_2(X;\Z)_{\om}^{\phi}$} or $B\!=\!0$ and $g\!+\!\ell\!\le\!1$.
The {\it Effectivity} property above in the third case follows from~\eref{wtthiprp_e}
and the fact that~$\wt\th_i$ reverses the orientation of 
the moduli space~$\ov\fM_{g,\ell}^{\phi}(B;J)$.
The ${\it Grading}$ properties above and in Section~\ref{RGWs_subs2}  
are consequences of~\eref{RfMdim_e}.
The vanishing in~\ref{Rstr_it} is immediate from the third case in~\ref{Reff1_it}.\\

\noindent
The proofs of \ref{Rdil_it} and \ref{RdivII_it} are similar to the complex case.
Suppose first that either $B\!\neq\!0$ or $g\!+\!\ell\!\ge\!2$ so that the forgetful
morphism
\BE{Rwtffell_e}\wt\ff_{\ell+1}^{\R}\!:
\ov\fM_{g,\ell+1}^{\phi}(B;J)\lra\ov\fM_{g,\ell}^{\phi}(B;J)\EE
dropping the last conjugate pair of marked points is well-defined.
For $i\!\in\![\ell]$, let 
\BE{EDipmdfn_e}D_i^{\pm} \subset \ov\fM_{g,\ell+1}^{\phi}(B;J)\EE
be the subspace of maps from domains $\Si$ so that one of the irreducible components~$\Si_i$ 
of~$\Si$
is~$\P^1$ which has precisely one node, carries only the marked points~$z_i^+$
and $z_{\ell+1}^{\pm}$, and is contracted by the~map.
The forgetful morphism~\eref{Rwtffell_e} restricts to an isomorphism
\BE{Dipmisom_e} D_i^{\pm}\approx \ov\fM_{g,\ell}^{\phi}(B;J)\,.\EE
The (virtual) normal bundle $\cN D_i^{\pm}$ of $D_i^{\pm}$ in 
$\ov\fM_{g,\ell+1}^{\phi}(B;J)$
is canonically isomorphic to the complex line bundle of the smoothings of 
the above node of~$\Si_i$, as in \cite[Lemma~5.2]{RealEnum}.
By \cite[Corollary~3.17]{RealGWsGeom},
the sign of the isomorphism~\eref{Dipmisom_e} with respect to the orientation on~$D_i^{\pm}$
determined by the orientations of $\ov\fM_{g,\ell+1}^{\phi}(B;J)$ and $\cN D_i^{\pm}$ 
is~$\pm1$.\\

\noindent
By the same reasoning as in the complex case,
\BE{psipull_e} \psi_i\big|_{D_i^{\pm}},\psi_{\ell+1}\big|_{D_i^{\pm}}=0,\quad
\psi_i=\big\{\wt\ff_{\ell+1}^{\R}\big\}^{\!*}\psi_i\!+\![D_i^+]^{\vir}\!+\![D_i^-]^{\vir}
\in H^2\big(\ov\fM_{g,\ell+1}^{\phi}(B;J);\Q\big);\EE
the meaning of the last identity in~\eref{psipull_e} is analogous to that 
of~\eref{CTRR_e}, as explained above Proposition~\ref{CTRR_prp}.
Under the identification~\eref{Dipmisom_e}, 
$$c_1(\cN D_i^{\pm})\approx-\psi_i, \quad 
\big\{\wt\ff_{\ell+1}^{\R}\!\big\}^{\!*}\psi_j\big|_{D_i^{\pm}}\approx\psi_j~~\forall\,j\!\in\![\ell],
\quad
\psi_j\big|_{D_i^{\pm}}\approx\psi_j~~\forall\,j\!\in\![\ell]\!-\!\{i\}\,.$$
Along with~\eref{psipull_e}, this gives
\begin{gather*}
\blr{\tau_{a_1}(\mu_1),\ldots,\tau_{a_{\ell}}(\mu_{\ell}),\tau_1(1)}_{\!g,B}^{\!\om,\phi}
=\int_{[\ov\fM_{g,\ell+1}^{\phi}(B;J)]^{\vir}}\!\!\!
\big\{\wt\ff_{\ell+1}^{\R}\big\}^{\!*}\!\big( 
\psi_1^{a_1}\!\big(\ev_1^*\mu_1\big)\ldots\psi_{\ell}^{a_{\ell}}\!\big(\ev_{\ell}^*\mu_{\ell}\big)
\!\big)\psi_{\ell+1},\\
\begin{split}
&\blr{\tau_{a_1}(\mu_1),\ldots,\tau_{a_{\ell}}(\mu_{\ell}),\tau_0(\mu)}_{\!g,B}^{\!\om,\phi} 
=\int_{[\ov\fM_{g,\ell+1}^{\phi}(B;J)]^{\vir}}\!\!\!
\big\{\wt\ff_{\ell+1}^{\R}\big\}^{\!*}\!\big( 
\psi_1^{a_1}\!\big(\ev_1^*\mu_1\big)\ldots\psi_{\ell}^{a_{\ell}}\!\big(\ev_{\ell}^*\mu_{\ell}\big)
\!\big)\big(\ev_{\ell+1}^{~*}\mu\big)\\
&\hspace{1in}
+\!\!\sum_{\begin{subarray}{c}1\le i\le\ell\\ a_i>0\end{subarray}}
\int_{[\ov\fM_{g,\ell}^{\phi}(B;J)]^{\vir}}\!\!
\bigg(\prod_{j=1}^{i-1}\!\psi_j^{a_j}\!\big(\ev_j^*\mu_j\big)\!\!\!\bigg)
\psi_i^{a_i-1}\!\big(\ev_i^*(\mu_i(\mu\!-\!\phi^*\mu)\!)\!\big)
\bigg(\prod_{j=i+1}^{\ell}\!\!\!\!\psi_j^{a_j}\!\big(\ev_j^*\mu_j\big)\!\!\!\bigg).
\end{split}
\end{gather*}

\vspace{.18in}

\noindent
The identities~\ref{Rdil_it} and~\ref{RdivII_it} follow from
the above two equations,
\BE{psipull_e5}\blr{c_1(\cL_{\ell+1}^{~*}|_{\Si}\!\otimes\!T\Si),\Si}=2\ell,
\qquad \lr{\ev_{\ell+1}^{~*}\mu,\Si}=\lr{\mu,B},\EE
where $\Si$ is a generic fiber of the forgetful morphism~$\wt\ff_{\ell+1}^{\R}$ 
in~\eref{Rwtffell_e} oriented by the position of the marked point~$z_{\ell+1}^+$
and \hbox{$\cL_{\ell+1}\!\lra\!\ov\fM_{g,\ell+1}^{\phi}(B;J)$} is the universal tangent line bundle
at~$z_{\ell+1}^+$,
and from the compatibility, as in~\eref{fiberint_e}, of the virtual fundamental classes for 
$\ov\fM_{g,\ell+1}^{\phi}(B;J)$ and $\ov\fM_{g,\ell}^{\phi}(B;J)$
with this fiber orientation.
We justify~\eref{psipull_e5} below.\\

\noindent
The second identity in~\eref{psipull_e5} follows from the fact that
the intersection of any degree~$B$ curve in~$X$
with a generic representative for the Poincare dual of~$\mu$ is $\lr{\mu,B}$.
The pairing 
\BE{psipull_e7}\cL_{\ell+1}^{~*}|_{\Si}\!\otimes\!T\Si\lra\C, \qquad \psi\!\otimes\!v\lra\psi(v),\EE
vanishes transversely at the marked points $z_i^{\pm}\!\in\!\Si$ with $i\!\in\![\ell]$
(which correspond to the intersections of~$\Si$ with~$D_i^{\pm}$).
It also vanishes on the real locus~$\Si^{\si}$ of~$\si$
(which corresponds to the intersection of~$\Si$ with the subspace of
$\ov\fM_{g,\ell+1}^{\phi}(B;J)$
of maps from domains~$\Si'$ so that one of the irreducible components of~$\Si'$
is~$\P^1$ which has precisely one node, carries only the marked points~$z_{\ell+1}^{\pm}$, 
and is contracted by the~map).
Each of the $2\ell$ points $z_i^{\pm}\!\in\!\Si$ is a positive zero
of the pairing~\eref{psipull_e7}. 
The fixed locus~$\Si^{\si}$ is a disjoint union of circles with a trivial normal bundle.
A small deformation of the section of $\cL_{\ell+1}^{~*}|_{\Si}\!\otimes\!T\Si$
given by~\eref{psipull_e7} does not vanish near~$\Si^{\si}$.
Thus, $\Si^{\si}$ does not contribute to the first number in~\eref{psipull_e5}.
This establishes the first identity in~\eref{psipull_e5}.\\

\noindent
The remaining cases of \ref{Rdil_it} and~\ref{RdivII_it} are $B\!=\!0$ and
either $(g,\ell)\!=\!(0,1)$ or $(g,\ell)\!=\!(1,0)$.
The right-hand sides of the equations in~\ref{Rdil_it} and~\ref{RdivII_it}
vanish in either case.
The left-hand sides are integrals against the virtual classes~of   
\BE{psipull_e9}\ov\fM_{0,2}^{\phi}(0;J)\approx\R\ov\cM_{0,2}\!\times\!X^{\phi}
\qquad\hbox{and}\qquad
\ov\fM_{1,1}^{\phi}(0;J)\approx\R\ov\cM_{1,1}\!\times\!X^{\phi}\,.\EE
With $\pi_1,\pi_2$ denoting the projections of the right-hand sides above
to the two factors,
$$\psi_i\approx\pi_1^*\psi_i  \qquad\hbox{and}\qquad 
\ev_i^*\mu_i\approx\pi_2^*\big(\mu_i|_{X^{\phi}}\!\big)$$
under these identifications. 
Since $\R\ov\cM_{0,2}\!\approx\!S^1$, $\psi_1,\psi_2\!=\!0$ if $(g,\ell)\!=\!(0,1)$.
Thus, the left-hand sides of the equations in~\ref{Rdil_it} and~\ref{RdivII_it} 
also vanish in this case.
Since the connected two-dimensional moduli space~$\R\ov\cM_{1,1}$ is not orientable,
the left-hand side of the equation in~\ref{Rdil_it} vanishes in the $(g,\ell)\!=\!(1,0)$
case as well.\\

\noindent
Since the obstruction bundle for $\ov\fM_{1,1}^{\phi}(0;J)$ is isomorphic to
$\pi_1^*(\bE^{\R})^*\!\otimes\!\pi_2^*TX^{\phi}$ under the second identification 
in~\eref{psipull_e9}, 
\BE{psipull_e11}\begin{split}
\blr{\tau_0(\mu)}_{\!1,0}^{\!\om,\phi}
&=\pm\blr{(\pi_2^*\mu)e(\pi_1^*(\bE^{\R})^*\!\otimes\!\pi_2^*TX^{\phi}),
\R\ov\cM_{1,1}\!\times\!X^{\phi}}\\
&=\pm\blr{e(\pi_1^*(\bE^{\R})^*\!\otimes\!\pi_2^*TX^{\phi}),
\R\ov\cM_{1,1}\!\times\!(M\!\cap\!X^{\phi})}
\end{split}\EE
if $M$ is a generic representative for the Poincar\'e dual of~$\mu$.
For dimensional reasons,
the restriction of~$TX^{\phi}$ to~$M\!\cap\!X^{\phi}$
contains a trivial rank~2 subbundle;
we denote its complement by~$V$.
By~\eref{psipull_e11},
$$\blr{\tau_0(\mu)}_{\!1,0}^{\!\om,\phi}
=\pm\blr{e\big(\!(\bE^{\R})^*\!\oplus\!(\bE^{\R})^*\big),\R\ov\cM_{1,1}}\blr{e(V),M\!\cap\!X^{\phi}}.$$
Since $\R\ov\cM_{1,1}$ is not orientable (resp.~$V$ is oriented and of odd rank), 
the first (resp.~second) factor on the right-hand side above vanishes. 
This establishes the remaining case of~\ref{RdivII_it}.

\subsection{Properties of invariants II}
\label{RGWs_subs2}

\noindent
For $\wt{P}\!\in\!\wt\cP(g,\ell)$ as in~\eref{PwtPdfn_e} and $B\!\in\!H_2(X;\Z)$, let
$$\cP_{\wt{P}}^{\phi}(B)=
\big\{\!(B',B_0)\!\in\!H_2(X;\Z)^2\!:
B'\!-\!\phi_*(B')\!+\!B_0\!=\!B\big\}.$$
For a complex line bundle $L\!\lra\!X$, 
$g\!\in\!\Z$, disjoint subsets $I,J\!\subset\!\Z^+$, 
and \hbox{$B\!\in\!H_2(X;\Z)$}, define
\begin{gather}\notag
\cI^{\om,\phi,L}_{g,(I,J),B}\!:
H^*(X;\Q)^{|I\sqcup J|}\lra H^*\big(\ov\cM_{g,|I\sqcup J|};\Q\big),\\
\label{IomphiLdfn_e}
\cI^{\om,\phi,L}_{g,(I,J),B}\big(\mu_1,\ldots,\mu_{|I\sqcup J|}\big)
=(-1)^{\lr{c_1(L),\phi_*(B)}}\cI^{\om}_{g,|I\sqcup J|,B}
\big(\mu_1',\ldots,\mu_{|I\sqcup J|}'\big),\\
\notag
\hbox{where}\qquad \mu_i'=\begin{cases}
\mu_i,&\hbox{if}~i\!\in\!\io_{I,J}^{\,-1}(I);\\
-\phi^*\mu_i,&\hbox{if}~i\!\in\!\io_{I,J}^{\,-1}(J).
\end{cases}\end{gather}

\noindent
Each moduli space $\R\ov\cM_{0,\ell}$ with $\ell\!\ge\!2$ is oriented as described
in \cite[Section~3.5]{RealGWsGeom}.
For the purposes of the $B\!=\!0$ case of~\ref{RfcI_it} and
the $g\!=\!0$ case of~\ref{RMapPt_it} below,
we identify $\wt{H}^*(\R\ov\cM_{0,\ell};\Q)$ with $H^*(\R\ov\cM_{0,\ell};\Q)$
via this orientation.
The identities in~\ref{RfcII_it}, \ref{RGenusRed_it}, and~\ref{RSplit_it} 
depend on the choices of orientations for 
a generic fiber of the forgetful morphism~$\ff_{\ell+1}^{\R}$
and for the normal bundle of the immersions in~\eref{Riogldfn_e}-\eref{RioIJEdfn_e};
these are specified in Section~\ref{RMS_subs}.\\

\noindent
The linear maps~\eref{RIdfn_e} satisfy analogues of Kontsevich-Manin's axioms of 
\cite[Section~2]{KM}:
\begin{enumerate}[label=$\R{\arabic*}\!\!$,ref=$\R{\arabic*}$,leftmargin=22pt]
\setcounter{enumi}{\arabic{saveenumi}}

\item\label{Reff2_it} ({\it Effectivity II}): 
$\cI_{g,\ell,B}^{\om,\phi}\!=\!0$ if $B\!\not\in\!H_2(X;\Z)_{\om}^{\phi}$;

\item\label{Rcov_it} ({\it Covariance}): 
the map $\cI_{g,\ell,B}^{\om,\phi}$ is $\bS_{\ell}$-equivariant and
$${}\hspace{-.3in}
\cI_{g,\ell,B}^{\om,\phi}\big(\mu_1,\ldots,\mu_{i-1},\phi^*\mu_i,\mu_{i+1},\ldots,\mu_{\ell}\big)
=\th_i^*\,\cI_{g,\ell,B}^{\om,\phi}(\mu_1,\ldots,\mu_{\ell})
\quad\forall\,\mu_1,\ldots,\mu_{\ell}\!\in\!H^*(X;\Q);$$

\item\label{Rgrad_it} ({\it Grading II}): $\cI_{g,\ell,B}^{\om,\phi}$ is homogeneous of degree
$(g\!-\!1)n\!-\!\lr{c_1(X,\om),B}$, i.e.
$$\big|\cI_{g,\ell,B}^{\om,\phi}(\mu)\big|=|\mu|+(g\!-\!1)n\!-\!\blr{c_1(X,\om),B}
\quad\forall~\mu\!\in\!H^*(X;\Q)^{\ell};$$

\item\label{RfcI_it} ({\it Fundamental Class I}): for all $\mu\!\in\!H^*(X;\Q)$,
$$\cI_{0,2,B}^{\om,\phi}(\mu,1)
=\begin{cases} 0, &\hbox{if}~B\!\neq\!0;\\ 
-\lr{\mu,[X^{\phi}]}, &\hbox{if}~B\!=\!0;\end{cases}$$

\item\label{RfcII_it} ({\it Fundamental Class II}): if $g\!+\!\ell\!\ge\!2$ and 
$\mu\!\in\!H^*(X;\Q)^{\ell}$,  
$${}\hspace{-.5in}
\cI_{g,\ell+1,B}^{\om,\phi}(\mu,1)
=\big(\ff_{\ell+1}^{\R}\big)^{\!*}\!\big(\cI_{g,\ell,B}^{\om,\phi}(\mu)\!\big);$$

\item\label{RdivI_it}  ({\it Divisor II}):  
if $g\!+\!\ell\!\ge\!2$, $\mu\!\in\!H^*(X;\Q)^{\ell}$, $\mu_{\ell+1}\!\in\!H^2(X;\Q)$, and
$\ga\!\in\!H^*(\R\ov\cM_{g,\ell};\Q)$, 
$$\hspace{-.5in}
\int_{\R\ov\cM_{g,\ell+1}}\!\!(\ff_{\ell+1}^{\R})^{\!*}\!\ga\,
\cI_{\!g,\ell+1,B}^{\om,\phi}
\big(\mu,\mu_{\ell+1}\big)=
\lr{\mu_{\ell+1},B}\!\int_{\R\ov\cM_{g,\ell}}\!\!\ga
\cI_{g,\ell,B}^{\om,\phi}(\mu);$$

\item\label{RMapPt_it} ({\it Mapping to Point}):
for all $\mu_1,\ldots,\mu_{\ell}\!\in\!H^*(X;\Q)$,
$${}\hspace{-.5in}
\cI_{g,{\ell},0}^{\om,\phi}\big(\mu_1,\ldots,\mu_{\ell}\big)
=\begin{cases} 
-\lr{\mu_1\!\ldots\!\mu_{\ell},[X^{\phi}]}, &\hbox{if}~g\!=\!0,\,\ell\!\ge\!2;\\
\pm\lr{\mu_1\!\ldots\!\mu_{\ell},[X^{\phi}]}e(\bE^{\R}), 
&\hbox{if}~n\!=\!1,\,g\!\in\!2\Z^+;\\
0,&\hbox{otherwise};
\end{cases}$$

\item\label{RGenusRed_it} ({\it Genus Reduction}): for all $\mu\!\in\!H^*(X;\Q)^{\ell}$,
\begin{equation*}\begin{split}
\hspace{-.2in}
\big(\io_{g,\ell}^{\C}\big)^{\!*}
\big(\cI_{g,\ell,B}^{\om,\phi}(\mu)\!\big)
&=\sum_{i,j=1}^N\!g^{ij}\cI_{g-2,\ell+2,B}^{\om,\phi}\big(\mu,e_i,e_j\big),\\
\big(\io_{g,\ell}^E\big)^{\!*}
\big(\cI_{g,\ell,B}^{\om,\phi}(\mu)\!\big)
&=(-1)^{g+|\mu|}\cI_{g-1,\ell+1,B}^{\om,\phi}\big(\mu,\PD_X^{-1}\big([X^{\phi}]_X\big)\!\big);
\end{split}\end{equation*}

\item\label{RSplit_it} ({\it Splitting}): for all $\mu\!\in\!H^*(X;\Q)^{\ell}$,
$P\!\in\!\cP(g\!-\!1,\ell)$ as in~\eref{PwtPdfn_e},
$\wt{P}\!\in\!\wt\cP(g,\ell)$ as in~\eref{PwtPdfn_e}, and
partitions \hbox{$[\ell]\!=\!I\!\sqcup\!J$},
\begin{equation*}\begin{split}
{}\hspace{-.3in}
\big(\io_P^{\C}\big)^{\!*}\big(\cI_{g,\ell,B}^{\om,\phi}(\mu)\!\big)&
=(-1)^{\ve_n(P,\mu)}\!\!\!\!\!\!\!\!\!\!\!\!\sum_{\begin{subarray}{c}B_1,B_2\in H_2(X;\Z)\\
B_1+B_2=B \end{subarray}}\!\sum_{i,j=1}^N\!g^{ij}
\cI_{g_1,|I|+1,B_1}^{\om,\phi}\!\big(\mu_I,e_i\big)\!\times\!
\cI_{g_2,|J|+1,B_2}^{\om,\phi}\!\big(e_j,\mu_J\big),\\
{}\hspace{-.3in}
\big(\io_{\wt{P}}^{\C}\big)^{\!*}\big(\cI_{g,\ell,B}^{\om,\phi}(\mu)\!\big)
&=(-1)^{\ve(\wt{P},\mu)}\!\!\!\!\!\!\!\!\!\!\!\!\sum_{(B',B_0)\in\cP_{\wt{P}}^{\phi}(B)}\!
\sum_{i,k=1}^N\!g^{ik}
\cI_{g',(I\sqcup\{\ell+1\},J),B'}^{\om,\phi,L}\!\big(\mu_{I\sqcup J},e_i\big)\!\times\!
\cI_{g_0,|K|+1,B_0}^{\om,\phi}\!\big(e_k,\mu_K\big),\\
{}\hspace{-.3in}
\big(\io_{g,(I,J)}^{\C}\big)^{\!*}\big(\cI_{g,\ell,B}^{\om,\phi}(\mu)\!\big)
&=\!\!\!
\sum_{\begin{subarray}{c}B'\in H_2(X;\Z)\\ B'-\phi_*(B')=B \end{subarray}}\!
\sum_{i,j=1}^N\!g^{ij}
\cI_{(g-1)/2,(I\sqcup\{\ell+1\},J\sqcup\{\ell+2\}),B'}^{\om,\phi,L}
\big(\mu,e_i,e_j\big),\\
{}\hspace{-.3in}
\big(\io_{g,(I,J)}^E\big)^{\!*}\big(\cI_{g,\ell,B}^{\om,\phi}(\mu)\!\big)
&=(-1)^{|\mu|}\!\!\!\!\!\!\!\!\!\!
\sum_{\begin{subarray}{c}B'\in H_2(X;\Z)\\ B'-\phi_*(B')=B \end{subarray}}
\!\!\!\!\!\!\!\!\!
\cI_{g/2,(I\sqcup\{\ell+1\},J),B'}^{\om,\phi,L}\big(\mu,\PD_X^{-1}\big([X^{\phi}]_X\big)\!\big);
\end{split}\end{equation*}

\item\label{RMotivic_it} ({\it Motivic Axiom}):
$\cI_{g,\ell,B}^{\om,\phi}$ is induced by the correspondence~$C_{g,\ell,B}^{\om,\phi}$.

\end{enumerate}

\vspace{.15in}

\noindent
The first property in~\ref{Rcov_it} follows from the $\bS_{\ell}$-invariance
of the twisted fundamental class of~$\R\ov\cM_{g,\ell}$
and of the virtual fundamental class of~$\ov\fM_{g,\ell}^{\phi}(B;J)$,
the first identity in~\eref{bSellact_e2b}, and the first identity in~\eref{Rffvp_e}.
The second property in~\ref{Rcov_it} follows~from
\BE{thiFCprp_e}\th_{i*}\big([\R\ov\cM_{g,\ell}]\big)=-[\R\ov\cM_{g,\ell}], \qquad
\wt\th_{i*}\big(\big[\ov\fM_{g,\ell}^{\phi}(B;J)\big]^{\vir}\big)
=-\big[\ov\fM_{g,\ell}^{\phi}(B;J)\big]^{\vir},\EE
the first identity in~\eref{wtthiprp_e}, and the second identity in~\eref{Rffvp_e}.
The property~\ref{RfcII_it} is a consequence of the compatibility, as in~\eref{fiberint_e},
of the virtual fundamental classes for~$\ov\fM_{g,\ell+1}^{\phi}(B;J)$
and~$\ov\fM_{g,\ell}^{\phi}(B;J)$ 
with the orientation~of a generic fiber of the forgetful morphism~$\wt\ff_{\ell+1}^{\R}$
in~\eref{Rwtffell_e} by the position of the marked point~$z_{\ell+1}^+$.
Since the intersection of any degree~$B$ curve in~$X$
with a generic representative for the Poincare dual of $\mu\!\in\!H^2(X;\Q)$ is $\lr{\mu,B}$,
this compatibility also implies~\ref{RdivI_it}.
Similarly to the complex case, \ref{RMotivic_it} is immediate 
from~\eref{cupcap_e}, \eref{PDdfn_e},  and
the definitions of~$\cI_{g,\ell,B}^{\om,\phi}$ 
and~$C_{g,\ell,B}^{\om,\phi}$ after~\eref{RIdfn_e}.
We establish~\ref{RGenusRed_it} and~\ref{RSplit_it} in Section~\ref{RSplitPf_subs}.

\begin{proof}[{\bf{\emph{Proof of~\ref{RfcI_it} and~\ref{RMapPt_it},\,$\mathbf{g\!=\!0}$}}}]
By definition, 
\BE{RfcI_e3}\int_{\R\ov\cM_{0,2}}\!\!\ga
\cI_{0,2,B}^{\om,\phi}\big(\mu,1\big)
= \int_{[\ov\fM_{0,2}^{\phi}(B;J)]^{\vir}}\!\! (\ff^*\ga\big)
\big(\ev_1^*\mu\big)\,.\EE
If $B\!\neq\!0$, $\ev_1^*\mu\!=\!\wt\ff_2^{\R\,*}\ev_1^*\mu$ and a generic fiber of 
$$\wt\ff_2^{\R}\!:\ov\fM_{0,2}^{\phi}(B;J) \lra\ov\fM_{0,1}^{\phi}(B;J)$$
is~$\P^1$.
Since $\R\ov\cM_{0,2}\!\approx\!\R\P^1$, the integral on the right-hand side
of~\eref{RfcI_e3} thus vanishes for dimensional reasons.
This establishes the $B\!\neq\!0$ case of~\eref{RfcI_it}.\\

\noindent
Let $\ell\!\ge\!2$.
By \cite[Corollary~3.19]{RealGWsGeom}, the natural isomorphism
$$\ov\fM_{0,\ell}^{\phi}(0;J)\approx \R\ov\cM_{0,\ell}\!\times\!X^{\phi}$$
is orientation-reversing with respect to the orientations on~$\R\ov\cM_{0,\ell}$
and $\ov\fM_{0,\ell}^{\phi}(0;J)$ specified in \cite[Section~3.5]{RealGWsGeom}.
Along with~\eref{FCprod_e}, this implies the $B\!=\!0$ case of~\ref{RfcI_it} and 
the $g\!=\!0$ case of~\ref{RMapPt_it}.
\end{proof}

\begin{proof}[{\bf{\emph{Proof of~\ref{RMapPt_it},\,$\mathbf{g\!\ge\!1}$}}}]
There is a natural identification
\BE{RMB0split_e} \ov\fM_{g,\ell}^{\phi}(0;J)\approx X^{\phi}\!\times\!\R\ov\cM_{g,\ell}\,.\EE
We denote by 
$\pi_1,\pi_2\!:\ov\fM_{g,\ell}^{\phi}(0;J)\!\lra\!X^{\phi},\R\ov\cM_{g,\ell}$ 
the two projections.
The kernels of the linearizations~$D_{(TX,\nd\phi)}$ of the $\dbar_J$-operators on 
the space of real maps from symmetric surfaces to~$(X,\phi)$ form 
the vector bundle $\pi_1^*TX^{\phi}$ over the moduli space 
on the left-hand side of~\eref{RMB0split_e}.
The cokernels of these linearizations form 
a vector orbi-bundle isomorphic to the real part
\BE{RB0Obs_e}\big(\pi_1^*TX\!\otimes\!\pi_2^*\bE^*\big)^{\!\R}\approx 
\pi_1^*TX^{\phi}\!\otimes\!\pi_2^*(\bE^{\R})^*\
\lra X^{\phi}\!\times\!\R\ov\cM_{g,\ell}\EE
of the complex vector bundle $\pi_1^*TX\!\otimes\!\pi_2^*\bE^*$ 
with the conjugation induced by the involutions on~$X$ and~$\bE$.\\

\noindent
For a symmetric, possibly nodal, Riemann surface $(\Si,\si)$,
we denote by $\dbar_{(\Si,\si)}$ the real Cauchy-Riemann operator on
the real line pair $(\Si\!\times\!\C,\si\!\times\!\fc)$ over~$(\Si,\si)$
induced by the standard $\dbar$-operator on the smooth functions on~$\Si$;
see \cite[Section~2]{RealGWsGeom}.
The determinants of the operators~$\dbar_{(\Si,\si)}$ form a real line orbi-bundle~$\det\dbar_{\C}$
over~$\R\ov\cM_{g,\ell}$; see~\cite{detLB}.
Since the kernel of each operator~$\dbar_{(\Si,\si)}$ consists of constant $\R$-valued functions,
there is a canonical homotopy class of isomorphisms
\BE{RcMisom_e0}   \det\dbar_{\C}\approx \La_{\R}^{\top}(\bE^{\R})\EE
of real line orbi-bundles over~$\R\ov\cM_{g,\ell}$. 
By \cite[Propositions~5.9/6.1]{RealGWsI}, there is also a canonical homotopy class of isomorphisms
\BE{RcMisom_e} \La^{\top}_{\R}\big(T(\R\ov\cM_{g,\ell})\!\big)\approx\det\dbar_{\C}\EE
of real line orbi-bundles over~$\R\ov\cM_{g,\ell}$.\\

\noindent
The determinants of the linearizations of the $\dbar_J$-operators on 
the space of real maps from symmetric surfaces to~$(X,\phi)$   
form a real line orbi-bundle $\det D_{(TX,\nd\phi)}$ 
over $\ov\fM_{g,\ell}^{\phi}(0;J)$; see \cite[Section~4.3]{RealGWsI}.
The forgetful morphism~\eref{Rffdfn_e} induces an isomorphism
\BE{RfMisom_e}\La_{\R}^{\top}\big(T\ov\fM_{g,\ell}^{\phi}(0;J)\!\big)
\approx \big(\!\det D_{(TX,\nd\phi)}\big)\!\otimes\!
\ff^*\!\big(\La^{\top}_{\R}(T(\R\ov\cM_{g,\ell})\!\big)\!\big)\EE
of real line orbi-bundles over $\ov\fM_{g,\ell}^{\phi}(0;J)$;
the line orbi-bundle on the left-hand side is the top exterior power of 
the virtual tangent bundle for the moduli space $\ov\fM_{g,\ell}^{\phi}(0;J)$.
A real orientation~$(L,[\psi],\fs)$ on~$(X,\om,\phi)$ determines a homotopy class of isomorphisms
\BE{ROisom_e}  \det D_{(TX,\nd\phi)}\approx  
\ff^*\!\big(\!(\det\dbar_{\C})^{\otimes n}\big)\EE
of real line orbi-bundles over $\ov\fM_{g,\ell}^{\phi}(0;J)$ via~\eref{RBPisom_e}
for each stable map~$u$ representing an element of $\ov\fM_{g,\ell}^{\phi}(0;J)$;
see \cite[Section~3.2]{RealGWsGeom}.
By \cite[Lemma~3.1]{RealGWsGeom} with \hbox{$\rk_{\C}L\!=\!1$} and \hbox{$\deg L\!=\!0$}, 
the homotopy class of isomorphisms in~\eref{ROisom_e} is determined by the orientation of~$X^{\phi}$
induced by~$(L,[\psi],\fs)$ if and only if $g(g\!-\!1)/2$ is even.\\

\noindent
Combining \eref{RcMisom_e}-\eref{ROisom_e} with the identification of 
the cokernels of the operators~$D_{(TX,\nd\phi)}$ in~\eref{RB0Obs_e}, 
we obtain a homotopy class of isomorphisms
\BE{RMB0vfc_e0}\begin{split} 
&\La_{\R}^{\top}\!\big(T(X^{\phi}\!\times\!\R\ov\cM_{g,\ell})\!\big)
\!\otimes\!
\big(\La_{\R}^{\top}\!\big(\pi_1^*TX^{\phi}\!\otimes\!\pi_2^*(\bE^{\R})^*\big)\!\big)^*
\equiv 
\La_{\R}^{\top}\big(T\ov\fM_{g,\ell}^{\phi}(0;J)\!\big)\\
&\hspace{1in}
\approx \big(\!\det D_{(TX,\nd\phi)}\big)\!\otimes\!
\ff^*\!\big(\La^{\top}_{\R}(T(\R\ov\cM_{g,\ell})\!\big)\!\big)
\approx \ff^*\!\big((\det\dbar_{\C})^{\otimes(n+1)}\big)
\end{split}\EE
of real line orbi-bundles over $\ov\fM_{g,\ell}^{\phi}(0;J)$.
Since $n\!+\!1\!\in\!2\Z$, the last line bundle in~\eref{RMB0vfc_e0} is canonically oriented.
Thus, 
\BE{RMB0vfc_e}
\big[\ov\fM_{g,\ell}^{\phi}(0;J)\big]^{\vir}=
\pm e\big(\pi_1^*TX^{\phi}\!\otimes\!\pi_2^*(\bE^{\R})^*\big)\!\cap\!
\big[X^{\phi}\!\times\!\R\ov\cM_{g,\ell}\big].\EE
The cap product above is taken with respect to the relative orientation of 
the vector orbi-bundle~\eref{RB0Obs_e} determined via~\eref{RMB0vfc_e0}.
This relative orientation depends on
the choice of a coherent system of determinant line bundles as in~\cite{detLB},
and so does the sign in~\eref{RMB0vfc_e}.\\

\noindent
If $g\!\not\in\!2\Z$, the rank of $\pi_1^*TX^{\phi}\!\otimes\!\pi_2^*(\bE^{\R})^*$
is odd.
The Euler class in~\eref{RMB0vfc_e} thus vanishes in this case.
If $g\!\ge\!2$, $\pi_1^*TX^{\phi}\!\otimes\!\pi_2^*(\bE^{\R})^*$ is pulled back from
$X^{\phi}\!\times\!\R\ov\cM_{g,0}$.
Since
$$\rk\big(\pi_1^*TX^{\phi}\!\otimes\!\pi_2^*(\bE^{\R})^*\!\big)>
\dim\!\big(X^{\phi}\!\times\!\R\ov\cM_{g,0}\big)
\qquad\hbox{if}~n\!>\!3,\,g\!\ge\!2,$$
the Euler class in~\eref{RMB0vfc_e} also vanishes if $n\!>\!3$ and $g\!\ge\!2$
(just as happens in the complex case).
If $n\!=\!1,3$, the vector bundle $TX^{\phi}$ over~$X^{\phi}$ is trivial.
Since 
$$3\,\rk(\bE^{\R})>\dim\R\ov\cM_{g,0} ~~\hbox{if}~g\!\ge\!2,$$
the Euler class in~\eref{RMB0vfc_e} also vanishes if $n\!=\!3$ and $g\!\ge\!2$.\\

\noindent
In the remaining $n\!=\!1$, $g\!\in\!2\Z^+$ case of~\ref{RMapPt_it}, 
\eref{RMB0vfc_e0} and~\eref{RMB0vfc_e} reduce~to
\begin{gather}\label{Rdeg0isom_e}
\La_{\R}^{\top}\!\big(T(X^{\phi}\!\times\!\R\ov\cM_{g,\ell})\!\big)
\!\otimes\!\La_{\R}^{\top}\!\big(\pi_2^*(\bE^{\R})^*\big)
\approx \La_{\R}^{\top}\big(T\ov\fM_{g,\ell}^{\phi}(0;J)\!\big)
\approx \ff^*\!\big(\!(\det\dbar_{\C})^{\otimes2}\big),\\
\label{RMB0vfc_e2}
\big[\ov\fM_{g,\ell}^{\phi}(0;J)\big]^{\vir}=
\pm (-1)^{g/2}
\Big([X^{\phi}]\!\times\!\big(e((\bE^{\R})^*)\!\cap\!\big[\R\ov\cM_{g,\ell}\big]\big)\!\!\Big).
\end{gather} 
The cap product above is taken with respect to the relative orientation of 
the vector orbi-bundle~\eref{bERdfn_e} determined by the homotopy class of 
isomorphisms
$$\La_{\R}^{\top}(\bE^{\R})\approx \det\dbar_{\C}\approx
 \La^{\top}_{\R}\big(T(\R\ov\cM_{g,\ell})\!\big)$$
induced by~\eref{RcMisom_e0} and~\eref{RcMisom_e}.
The sign in~\eref{RMB0vfc_e2} depends on 
the choice of a coherent system of determinant line bundles as in~\cite{detLB}.
By~\eref{FCprod_e} and~\eref{cupcap_e}, 
the sign in the middle case of~\ref{RMapPt_it} is opposite to 
the overall sign in~\eref{RMB0vfc_e2}.
\end{proof}

\begin{rmk}\label{RMapPt_rmk}
By the second part of Section~5.3 in the 5th arXiv version of~\cite{detLB},
the choice of a coherent system of determinant line bundles
determines the sign between the virtual fundamental class 
of a smooth moduli space with an obstruction bundle and the cap product of 
the Euler class of the obstruction bundle with the fundamental class of the moduli space
with respect to the relative orientation induced by the orientation of the moduli space
as in~\eref{RMB0vfc_e}.
A coherent system of determinant line bundles also determines 
the homotopy class of the isomorphisms~\eref{RcMisom_e} and
the homotopy class of the isomorphisms 
$$\La_{\R}^{\top}\!\big(T(X^{\phi}\!\times\!\R\ov\cM_{g,\ell})\!\big)
\!\otimes\!
\La_{\R}^{\top}\!\big(\pi_1^*TX^{\phi}\!\otimes\!\pi_2^*(\!(\bE^{\R})^*)\!\big)
\approx \big(\!\det D_{(TX,\nd\phi)}\big)\!\otimes\!
\ff^*\!\big(\La^{\top}_{\R}(T\R\ov\cM_{g,\ell})\!\big)$$
induced by the identification in~\eref{RMB0vfc_e0} and the isomorphism in~\eref{RfMisom_e}.
By Proposition~5.10 in the 5th arXiv version of~\cite{detLB},
the sign in~\eref{RMB0vfc_e}
for the ``base" coherent system of determinant line bundles given by \cite[(4.10)]{detLB}
is~plus.
The sign~$\pm$ in~\eref{RMB0vfc_e2} is plus as well for this system.
In general, the sign~$\pm$ in~\eref{RMB0vfc_e2} is the same as the sign of
the component \hbox{$A_{1-g,g}\!\in\!\R^*$} corresponding, as in Theorem~2 in Section~3.4
of~\cite{detLB}, to the coherent system of determinant line bundles used.
The signs of~$A_{1-g,g}$ also correspond to the two relative orientations of 
the vector orbi-bundle~\eref{bERdfn_e} determined by the homotopy class 
of isomorphisms~\eref{RcMisom_e}.
As explained in \cite[Section~3.5]{RealGWsGeom}, 
the construction of this homotopy class in~\cite{RealGWsI} 
involves a somewhat arbitrary sign choice which can be fixed
systematically from an orientation of~$\R\ov\cM_{0,2}$.
\end{rmk}

\begin{rmk}\label{CanonOrient_rmk}
The properties \ref{Reff1_it}-\ref{RMotivic_it} are stated for 
the orientations on the moduli spaces~$\ov\fM_{g,\ell}^{\phi}(B;J)$
induced by a real orientation~$(L,[\psi],\fs)$ on~$(X,\om,\phi)$ via~\eref{RBPisom_e}
as in \cite[Section~3.3]{RealGWsGeom}.
These properties remain valid  with the modifications described below
if $\wt\phi$ is a conjugation on~$L$ lifting~$\phi$ and
the orientations on the moduli spaces are instead
induced via~\eref{RBPisom_e3} as in~\cite{RealGWsI}.
The homotopy class of isomorphisms in~\eref{ROisom_e} is then always determined by 
the orientation of~$X^{\phi}$ induced by~$(L,[\psi],\fs)$, and thus $(-1)^{g/2}$
should be omitted from~\eref{RMB0vfc_e2}.
By Remark~\ref{CanonOrient_rmk2},
the right-hand side in the first equation in~\ref{RGenusRed_it} should be negated,
while $g$ in the sign exponent in the second equation in~\ref{RGenusRed_it}
should be replaced with \hbox{$\lr{c_1(X,\om),B}/2\!-\!1$}.
Each summand on the right-hand side of the first equation in~\ref{RSplit_it} should be
multiplied by $(-1)^{\ve_{g_1,g_2}(B_1,B_2)}$ with 
$$\ve_{g_1,g_2}(B_1,B_2)=\big(g_1\!-\!1\!+\!\lr{c_1(X,\om),B_1}/2\big)
\big(g_2\!-\!1\!+\!\lr{c_1(X,\om),B_2}/2\big),$$
and the entire right-side should be negated.
For the purposes of the last three equations in~\ref{RSplit_it},
the leading sign should be dropped from the definition of~$\cI_{g,(I,J),B}^{\om,\phi,L}$
in~\eref{IomphiLdfn_e};
this makes~$\cI_{g,(I,J),B}^{\om,\phi,L}$ independent of the complex line bundle~$L$,
which can thus be omitted from the superscript. 
The leading sign exponents, in front of the sums,
in these three equations should changed to $(-1)^{\ve(\wt{P},\mu)+g'}$,
$(-1)^{(g+1)/2}$, and~$(-1)^{|\mu|+g/2}$, respectively.
The sum of~$\ve_n(P)$ and~$\ve_{g_1,g_2}(B_1,B_2)$ is even, 
i.e.~it does not contribute to the overall sign in the first equation in~\ref{RSplit_it}, 
if $n\!\cong\!3$ mod~4 and $c_1(X,\om)$ is divisible by~4 on the classes in~$H_2(X;\Z)$
representable by real maps.
This is in particular the case if $X$ is a Calabi-Yau threefold 
or the complex projective space~$\P^3$.
\end{rmk}

\subsection{Reconstruction of invariants}
\label{RRecon_subs}

\noindent
We call a collection of homomorphisms~$\cI_{0,\ell,B}^{\om,\phi}$
as in~\eref{RIdfn_e} with $\ell\!\ge\!2$ and $B\!\in\!H_2(X;\Z)$ which satisfies
the properties \ref{Reff2_it}-\ref{RfcII_it}, \ref{RdivI_it} with $\ga\!=\!1$, 
and \ref{RMapPt_it}-\ref{RMotivic_it} 
an \sf{extension} of the tree-level system $\{I_{0,\ell,B}^{\om}\}$ 
of GW-classes for~$(X,\om)$.
Theorem~\ref{Rrecont_thm1} below is the real analogue of Proposition~\ref{Crecont_prp1}.
According to~it, an extension of a tree-level system of GW-classes for~$(X,\om)$
can be reconstructed from a small subcollection of the full collection.
Similarly to Proposition~\ref{Crecont_prp1}, 
Theorem~\ref{Rrecont_thm1} follows readily from the fact that 
$H^*(\R\ov\cM_{0,\ell};\Q)$ is generated as a ring by certain ``boundary" classes.

\begin{thm}\label{Rrecont_thm1}
Let $(X,\om,\phi)$ be a compact real symplectic manifold of dimension~$2n$
and $\{\cI_{0,\ell,B}^{\om}\}$ be a tree-level system of GW-classes for~$(X,\om)$.
A $\phi$-extension $\{\cI_{0,\ell,B}^{\om,\phi}\}$ of $\{\cI_{0,\ell,B}^{\om}\}$
is determined by the~numbers
$$\int_{\R\ov\cM_{0,{\ell}}}\!\!\!\cI_{0,\ell,B}^{\om,\phi}(\mu)\in\Q
\qquad\hbox{with}\quad
\ell\in\Z^{\ge0},\,B\!\in\!H_2(X;\Z)_{\om}^{\phi},\,\mu\!\in\!H_-^*(X;\Q)^{\ell},$$
i.e.~by the codimension~0 classes $I_{0,\ell,B}^{\om,\phi}(\mu)$ 
with $\mu\!\in\!H_-^*(X;\Q)^{\ell}$.
If in addition $H_-^*(X;\Q)$ is generated as a ring by~$H_-^2(X;\Q)$ and $H^*_+(X;\Q)$, 
then this extension is determined by the codimension~0 classes 
$\cI_{0,2,B}^{\om,\phi}(\mu_1,\mu_2)$ with 
\BE{Rrecont_e2}\begin{split}
&\hspace{.8in} B\in H_2(X;\Z)_{\om}^{\phi},\quad \lr{c_1(X,\om),B}\le n\!+\!1,\\ 
&\mu_1\in H_-^*(X;\Q), \quad
|\mu_1|=n\!+\!\blr{c_1(X,\om),B}\!-\!1, \quad \mu_2\in H_-^2(X;\Q).
\end{split}\EE
\end{thm}

\begin{proof} By \cite[Theorem~2.2]{RDMhomol}, 
the cohomology ring $H^*(\R\ov\cM_{0,\ell};\Q)$ is generated by 
the Poincar\'e duals of the images of the immersions~$\io_{\wt{P}}^{\C}$
in~\eref{RiowtPdfn_e} with $\wt{P}\!\in\!\wt\cP(0,\ell)$ 
and~$\io_{0,(I,J)}^E$ in~\eref{RioIJEdfn_e} with \hbox{$[\ell]\!=\!I\!\sqcup\!J$}.
Along with Poincar\'e Duality for~$\R\ov\cM_{0,\ell}$, this implies~that
$$\bigcap_{\wt{P}\in\wt\cP(0,\ell)}\hspace{-.2in}\big(\!\ker\io_{\wt{P}}^{\C*}\big)
\cap\bigcap_{[\ell]=I\sqcup J}\!\!\!\!\!\!\big(\!\ker\io_{0,(I,J)}^{\,E*}\big)
=H^{2\ell-3}\big(\R\ov\cM_{0,\ell};\Q\big)\,,$$
with $\io_{\wt{P}}^{\C*}$ and $\io_{0,(I,J)}^{\,E*}$ denoting 
the pullback homomorphisms on the rational cohomology.
Thus, every positive-codimension cohomology class~$\cI_{0,\ell,B}^{\om,\phi}(\mu)$ on~$\R\ov\cM_{0,\ell}$
is determined by its pullbacks
\BE{Rrecont_e3}\begin{aligned}
\io_{\wt{P}}^{\C*}\cI_{0,\ell,B}^{\om,\phi}(\mu)&\in H^*\big(\ov\cM_{\wt{P}}^{\C};\Q\big)
&\quad&\hbox{with}~~\wt{P}\!\in\!\wt\cP(0,\ell)  \qquad\hbox{and}\\
\io_{0,(I,J)}^{E\,*}\cI_{0,\ell,B}^{\om,\phi}(\mu)&\in H^*\big(\ov\cM_{0,\ell+1};\Q\big)
&\quad&\hbox{with}~[\ell]\!=\!I\!\sqcup\!J\,.
\end{aligned}\EE
By the last identity in~\ref{RSplit_it}, each of the latter pullbacks is 
a linear combination of classes $\cI_{0,\ell+1,B'}^{\om}(\mu')$ determined by 
the tree-level system~$\{\cI_{0,\ell,B}^{\om}\}$. 
By the second identity in~\ref{RSplit_it}, 
each of the former pullbacks in~\eref{Rrecont_e3} 
to a nonempty domain is a linear combination of the products 
\hbox{$\cI_{0,\ell',B'}^{\om}(\mu')\!\times\!\cI_{0,\ell_0,B_0}^{\om,\phi}(\mu_0)$}
with $\ell_0\!<\!\ell$.
By induction, this implies that a $\phi$-extension $\{\cI_{0,\ell,B}^{\om,\phi}\}$ 
of $\{\cI_{0,\ell,B}^{\om}\}$ is determined  
by the codimension~0 classes $I_{0,\ell,B}^{\om,\phi}(\mu)$ 
with $\mu\!\in\!H^*(X;\Q)^{\ell}$.
By the second property in~\ref{Rcov_it} and the first identity in~\eref{thiFCprp_e},
the codimension~0 classes $I_{0,\ell,B}^{\om,\phi}(\mu_1,\ldots,\mu_{\ell})$
with $\mu_i\!\in\!H^*_+(X;\Q)$ for some $i\!\in\![\ell]$ vanish.
This establishes the first claim.\\

\noindent
Suppose $H^*_-(X;\Q)$ is generated as a ring by~$H_-^2(X;\Q)$ and $H^*_+(X;\Q)$.
By the proof of \cite[Corollary~2.4(2)]{RealEnum}, 
which depends only on a homology relation between the images of the immersions of~$\io_{\wt{P}}^{\C}$,
the second identity in~\ref{RSplit_it},
\ref{CfcI_it}, \ref{RfcI_it}, \ref{RdivI_it}, and the above vanishing property,
every codimension~0 class $I_{0,\ell,B}^{\om,\phi}(\mu)$ is
then determined by the codimension~0 classes~$I_{0,\ell,B}^{\om,\phi}(\mu_1,\mu_2)$
satisfying~\eref{Rrecont_e2} and the complex codimension~0 classes~$I_{0,\ell,B}^{\om}(\mu')$. 
Combining this with the first claim, we obtain the second claim.
\end{proof}

\noindent
A manifestation of the first statement of Theorem~\ref{Rrecont_thm1} is the WDVV-type equation 
for real GW-invariants established in~\cite{RealEnum}.
Define
$$\fd\!:H_2(X;\Z)\lra H_2(X;\Z), \qquad \fd(B')=B'\!-\!\phi_*(B').$$
For $[\ell]\!=\!I\!\sqcup\!J$ and $\mu\!\in\!H^*(X;\Q)^{\ell}$, let 
$$\cW_{I,J}(\mu)=(-1)^{\ve((I,J),\mu)}2^{|J|}\,.$$
According to \cite[Theorem~2.1]{RealEnum},
\BE{RWDVV_e}\begin{split}
&\sum_{\begin{subarray}{c}[\ell]=I\sqcup J\\ 2\in I,\,1,3\in J\end{subarray}}
\!\!\sum_{\begin{subarray}{c}B_0,B'\in H_2(X;\Z)\\ B_0+\fd(B')=B\end{subarray}}
\sum_{i,j=1}^N\!\cW_{I,J}(\mu)\blr{\mu_I,e_i}_{\!0,B_0}^{\!\om,\phi} 
g^{ij} \blr{e_j,\mu_J}_{\!0,B'}^{\!\om}\\
&\hspace{1.2in}=
\sum_{\begin{subarray}{c}[\ell]=I\sqcup J\\ 3\in I,\,1,2\in J\end{subarray}}
\!\!\sum_{\begin{subarray}{c}B_0,B'\in H_2(X;\Z)\\ B_0+\fd(B')=B\end{subarray}}
\sum_{i,j=1}^N\cW_{I,J}(\mu)\blr{\mu_I,e_i}_{\!0,B_0}^{\!\om,\phi} g^{ij} 
\blr{e_j,\mu_J}_{\!0,B'}^{\!\om} 
\end{split}\EE
for all $\ell\!\in\!\Z^+$, $B\!\in\!H_2(X;\Z)$, and 
$\mu\!\equiv\!(\mu_1,\ldots,\mu_{\ell})$ with 
$\mu_1\!\in\!H_+^*(X;\Q)$, 
and \hbox{$\mu_2,\ldots,\mu_{\ell}\!\in\!H_-^*(X;\Q)$}.
The full collection of relations~\eref{RWDVV_e} is equivalent to 
the compatibility of the homomorphism~$\fR_{\phi}$ on the extended 
quantum cohomology of~$(X,\om)$ defined at the end of \cite[Section~7]{RealEnum}
with the quantum product.
This collection completely determines all real genus~0 GW-invariants
$\lr{\ldots}_{0,B}^{\om,\phi}$ of $(X,\om)\!=\!(\P^n,\om_{\FS})$ with $n\!\not\in\!2\Z$
from the basic input $\lr{\pt}_{0,L}^{\om,\phi}\!=\!\pm1$, 
i.e.~the number of real lines in~$\P^n$ through a conjugate pair of points.\\

\noindent
Theorem~\ref{Rrecont_thm1} is a \sf{reconstruction} result in real GW-theory.
Another result is Theorem~\ref{RTRR_thm} below that reduces 
the $g\!=\!0$ descendant GW-invariants~\eref{RGWdfn_e} to the \sf{primary} ones,
i.e.~those with $a_i\!=\!0$ for all $i\!\in\![\ell]$.
Its proof is similar to that of Proposition~\ref{CTRR_prp}.
We give it under the assumption that
the strata of $\ov\fM_{0,\ell}^{\phi}(B;J)$ are of the expected dimension.
It adapts readily to semi-positive symplectic manifolds via real analogues of
Ruan-Tian's global inhomogeneous perturbations as in~\cite{RealRT}
and with some technical care to arbitrary real symplectic manifolds
with real orientations 
via the virtual fundamental class constructions of~\cite{LT,FO}.
The zero-contribution arguments in the last two paragraphs of the proof of 
Theorem~\ref{RTRR_thm} are similar to the proofs of \cite[Corollary~3.4(2)]{RealEnum2}
and \cite[Lemma~6.2]{RealEnum}, respectively.\\

\noindent
For $\ell\!\in\!\Z^+$, $i\!\in\![\ell]$ distinct, and
$B,B_0,B'\!\in\!H_2(X;\Z)$ with $B_0\!+\!\fd(B')\!=\!B$, we denote~by
\BE{RD2dfn_e} \R D_{B_0,B'i}\subset\ov\fM_{0,\ell}^{\phi}(B;J)\EE
the subspace of real maps~$u$ from domains $\Si^+\!\cup\!\Si_0\!\cup\!\Si^-$ such that 
\begin{enumerate}[label=$\bullet$,leftmargin=*]

\item $\Si^+,\Si_0,\Si^-$ are connected genus~0 curves 
with $\Si_0$ sharing a node with~$\Si^+$ and another node with~$\Si^-$,

\item the involution~$\si$ on the domain preserves~$\Si_0$ and interchanges
$\Si^+$ and~$\Si^-$,

\item the first point~$z_i^+$ in the $i$-th conjugate pair of marked points
lies on~$\Si^+$, and

\item the restrictions of~$u$ to~$\Si_0$ and~$\Si^+$ are of degrees~$B_0$ and~$B'$,
respectively.

\end{enumerate}
Under ideal circumstances, $\R D_{B_0,B'i}$ is a union 
of smooth codimension~2 submanifolds in~$\ov\fM_{0,\ell}^{\phi}(B;J)$
intersecting transversely.
The (virtual) normal bundle of each of these submanifolds is naturally isomorphic 
to the complex line bundle of the smoothings of the node 
shared by~$\Si_0$ and~$\Si^+$, as in \cite[Lemma~5.2]{RealEnum}.
These submanifolds thus inherit orientations from that of~$\ov\fM_{0,\ell}^{\phi}(B;J)$. 
By~\cite{LT,FO}, $\R D_{B_0,B'i}$ carries a natural virtual fundamental class.
The meaning of the equation~\eref{RTRR_e} below is that the integral
of the product of the left-hand side with a descendant cohomology class~$\eta$,  
as in the integrand in~\eref{RGWdfn_e}, against $[\ov\fM_{0,\ell}^{\phi}(B;J)]^{\vir}$
equals the integral
of the product of the first term on the right-hand side with~$\eta$ 
against $[\ov\fM_{0,\ell}^{\phi}(B;J)]^{\vir}$ plus
the integral of~$\eta$ against the weighted sum of~$[\R D_{B_0,B'i}]^{\vir}$.

\begin{thm}\label{RTRR_thm}
Let $(X,\om,\phi)$ be a compact real symplectic manifold of dimension $2n$ with $n\!\not\in\!2\Z$
and $(L,[\psi],\fs)$ be a real orientation on~$(X,\om,\phi)$.
For $\ell\!\in\!\Z^+$, $B\!\in\!H_2(X;\Z)$, $i\!\in\![\ell]$,
and $\mu\!\in\!H^2(X;\Q)$,
\BE{RTRR_e}\lr{\mu,B}\psi_i=-2\,\ev_i^*\mu
+\!\sum_{\begin{subarray}{c}B_0,B'\in H_2(X;\Z)\\ B_0+\fd(B')=B\end{subarray}}
\hspace{-.3in}\lr{\mu,B_0}\!\big[\R D_{B_0,B'i}\big]^{\vir}
\in H^2\big(\ov\fM_{0,\ell}^{\phi}(B;J);\Q\big)\,.\EE
\end{thm}

\begin{proof}
We can assume that $\om(B)\!>\!0$ and $\phi_*B\!=\!-B$.
By the linearity and continuity of both sides of~\eref{RTRR_e} in~$\mu$,
we can also assume that $\phi^*\mu\!=\!-\mu$ and that
$\mu$ can be represented by a (generic) pseudocycle $\io\!:M\!\lra\!X$.
Let $d\!=\!\lr{\mu,B}$ and $\hb\!\in\!\R^+$ be the minimal value of $\lr{\om,u_*[\P^1]}$
for a non-constant $J$-holomorphic map $u\!:\P^1\!\lra\!X$.\\

\noindent
We denote by $\wt\fM$ the set of representatives $u\!:\Si\!\lra\!X$
for the elements of $\ov\fM_{0,\ell}^{\phi}(B;J)$
such~that 
\begin{enumerate}[label=$\bullet$,leftmargin=*]

\item $z_i^+\!=\!0$ on the irreducible component $\P^1_{i;+}\!\subset\!\Si$ 
containing it;

\item if $z_i^-\!\in\!\P^1_{i;+}$, then $z_i^-\!=\!\i$;

\item if $z_i^-\!\not\in\!\P^1_{i;+}$, then 
\begin{enumerate}[label=$\circ$,leftmargin=*]

\item the node of~$\P^1_{i;+}$ separating
it from the irreducible component~$\P^1_{i;-}$
of~$\Si$ containing~$z_i^-$ is the point $\i\!\in\!\P^1_{i;+}$;

\item the interior of the unit disk $\D\!\subset\!\P_{i;+}^1$ centered at $z_i^+\!=\!0$
contains no marked points other than~$z_i^+$ and no nodes;

\item either $\int_{\D}u^*\om\!=\!\hb/2$ and $\D\!\subset\!\P_{i;+}^1$ 
contains no marked points other than~$z_i^+$ and no nodes or
$\int_{\D}u^*\om\!<\!\hb/2$ and $\D\!\subset\!\P_{i;+}^1$ 
contains a marked point other than~$z_i^+$ or a node.\\

\end{enumerate}
\end{enumerate}

\noindent
We call two elements of $\wt\fM$ \sf{equivalent} if they differ by a reparametrization
of the domain which commutes with the involution on the domain 
and restricts to the identity on the irreducible component~$\P^1_{i;+}$
containing~$z_i^+$.
Gromov's convergence induces a topology on the set~$\wh\fM$ of the resulting equivalence
classes of elements of~$\wt\fM$.
The action of~$S^1$ on~$\P^1_{i;+}$ induces a continuous action on~$\wh\fM$ so~that 
$\ov\fM_{0,\ell}^{\phi}(B;J)$ is the quotient $\wh\fM/S^1$ and
$$L\!\equiv\!\wh\fM\!\times_{S^1}\!\C\lra \ov\fM_{0,\ell}^{\phi}(B;J)$$
is the universal {\it tangent} line bundle for the marked point~$z_i^+$.
Below we define a ``meromorphic" section of~$L^{\otimes d}$ and determine its zero/pole locus.\\

\noindent
A generic element $u\!:\P^1\!\lra\!X$ of $\wh\fM$ intersects $M$ transversely 
at points 
$$y_1^+(u),\ldots,y_{d_+}^+(u)\in\C^* \qquad\hbox{and}\qquad 
y_1^-(u),\ldots,y_{d_-}^-(u)\in\C^*$$
positively and negatively, respectively, so that $d\!=\!d_+\!-\!d_-$.
The resulting element
$$s\big([u]\big)\equiv 
\big[u,y_1^+(u)\ldots y_{d_+}^+(u)\big/y_1^-(u)\ldots y_{d_-}^-(u)\big]
\in L^{\otimes d}$$
depends only on the image $[u]\!\in\!\ov\fM_{0,\ell}^{\phi}(B;J)$ of $u\!\in\!\wh\fM$.
This construction induces a section of~$L^{\otimes d}$ over generic elements
of $\ov\fM_{0,\ell}^{\phi}(B;J)$.
This section extends to a ``meromorphic" section over all of~$\ov\fM_{0,\ell}^{\phi}(B;J)$.\\

\noindent
The section $s$ has a zero (resp.~pole) whenever $u$ meets $M$ positively (resp.~negatively)
at $z_i^+\!=\!0$ and a pole (resp.~zero) whenever $u$ meets~$M$ positively (resp.~negatively)
at $z_i^-\!=\!\i$.
It also has a pole (resp.~zero) whenever $u\!\in\!\R D_{B_0,B'i}$  meets~$M$ at any point of~$\Si_0$ 
(in the terminology around~\eref{RD2dfn_e}) positively (resp.~negatively).
The same is the case if $u\!\in\!\R D_{B_0,B'i}$  meets~$M$ at any point of~$\Si^-$
or the domain of~$u$ consists of two components, $\P^1_{i;+}$ and $\P^1_{i;-}$,
interchanged by the involution and 
$u$ meets~$M$ at any point of~$\P^1_{i;-}$.
We denote the two sets of such elements by~$\R D_{B_0,B'i}^-(M)$
and $\R E_i^-(M)$, respectively.
Thus,
\BE{RTRR_e9}\begin{split}
s^{-1}(0)=&\big(\ev_i^{-1}(M)\!-\!\{\phi\!\circ\!\ev_i\}^{-1}(M)\!\big)
-\!\!\!\sum_{\begin{subarray}{c}B_0,B'\in H_2(X;\Z)\\ B_0+\fd(B')=B\end{subarray}}
\hspace{-.3in}\lr{\mu,B_0}\R D_{B_0,B'i}\\
&\hspace{.5in}
\cup\!\!\!\bigcup_{\begin{subarray}{c}B_0,B'\in H_2(X;\Z)\\ B_0+\fd(B')=B\end{subarray}}
\hspace{-.35in}\big(\R D_{B_0,B'i}^-(M)\!\cup\!\R E_i^-(M)\!\big)
\subset \ov\fM_{0,\ell}^{\phi}(B;J)\,.
\end{split}\EE

\vspace{.2in}

\noindent
The rows in the commutative diagram
$$\xymatrix{\ov\fM_{0,\ell}^{\phi}(B;J) \ar[rr]^>>>>>>>>>>>{\phi\circ\ev_i} \ar[d]^{\id}&& 
X\ar[d]^{\phi} && \ar[ll]_{\io} M \ar[d]^{\id}\\ 
\ov\fM_{0,\ell}^{\phi}(B;J) \ar[rr]^>>>>>>>>>>>>>{\ev_i} &&
X && \ar[ll]_{\phi\circ\io} M}$$
induce two fiber-product orientations on
$$\ov\fM_{0,\ell}^{\phi}(B;J)\,_{\phi\circ\ev_i}\!\times_{\io}M
=\ov\fM_{0,\ell}^{\phi}(B;J)\,_{\ev_i}\!\times_{\phi\circ\io}M\,.$$
Since the diffeomorphism~$\phi$ of~$X$ is orientation-reversing,
these two orientations are opposite.
Since \hbox{$\phi^*\mu\!=\!-\mu$} and $\phi_*[M]\!=\![M]$, 
the difference in~\eref{RTRR_e9} contributes $2\,\ev_i^*{\mu}$ to 
\BE{RTRR_e11}-\lr{\mu,B}\psi_i=c_1(L^{\otimes d})
\in H^2\big(\ov\fM_{0,\ell}^{\phi}(B;J);\Q\big).\EE
The contribution of the sum in~\eref{RTRR_e9} is exactly as in 
the complex case of Proposition~\ref{CTRR_prp}.
We show below that the subsets $\R D_{B_0,B'i}^-(M)$ and $\R E_i^-(M)$ 
of $s^{-1}(\i)$ do not contribute to~\eref{RTRR_e11} and thus establish~\eref{RTRR_e}.\\

\noindent
After intersecting with a cycle corresponding to integration against 
a descendant cohomology class~$\eta$ on $\ov\fM_{0,\ell}^{\phi}(B;J)$,
as in the integrand in~\eref{RGWdfn_e},
we can assume that $\R D_{B_0,B'i}^-(M)$ is a finite collection of points.
Let $d'\!=\!\lr{\mu,B'}$, $u_0\!\in\!\R D_{B_0,B'i}^-(M)$, and
$$y_1^+,\ldots,y_{d_+}^+\in\C^* \qquad\hbox{and}\qquad 
y_1^-,\ldots,y_{d_-}^-\in\C^*,$$
with $d'\!=\!d_+\!-\!d_-$, be the points at which the restriction of~$u_0$ 
to~$\Si^-$ intersects~$M$ positively and negatively, respectively.
The complex smoothing parameter~$c$ corresponding to the node shared by~$\Si_0$ and~$\Si^+$
also smooths out  the node shared by~$\Si_0$ and~$\Si^-$ according to 
the smoothing parameter~$\ov{c}$, i.e.~a point $1/w\!\in\!\Si^-$ 
in the domain of~$u$ corresponds to the point $\sim\!1/|c|^2w$ in
the domain of the associated smoothed out map~$u_c$ meeting~$M$.
The section~$s$ on a neighborhood of~$u_0$ can thus be approximated by the~map
$$\C\lra\C, \qquad c\lra \big(y_1^-\!\ldots\!y_{d_-}^-/y_1^+\!\ldots\!y_{d_+}^+\big)|c|^{-2d'}\,.$$
It follows that $s$ can be deformed to a section with no zeros or poles  on this neighborhood.
This implies that $\R D_{B_0,B'i}^-(M)$ does not contribute to~$\lr{\eta,[s^{-1}(0)]}$.\\

\noindent
If $\R E_i^-(M)\!\neq\!\eset$, then $d\!\in\!2\Z$.
After intersecting with a cycle corresponding to integration against 
a descendant cohomology class~$\eta$ on $\ov\fM_{0,\ell}^{\phi}(B;J)$,
we can assume that $\R E_i^-(M)$ is a finite collection of circles.
The normal bundle to~$\R E_i^-(M)$ is the trivial line bundle.
The section~$s$ on a neighborhood of a circle $S^1\!\subset\!\R E_i^-(M)$
can be approximated by a~map
$$S^1\!\times\!\R\lra\C, \qquad (u,r)\lra c(u)r^{-d/2}\,,$$
for some continuous map $c\!:S^1\!\lra\!\C^*$.
The signed number of zeros of a deformation of~$s$ on this neighborhood is
$$-(d/2)\blr{e\big(\!(S^1\!\times\!\C)/(S^1\!\times\!c\R)\!\big),S^1}=0\,;$$
see Propositions~2.18A and~B in~\cite{g0pr}.
Thus, $\R E_i^-(M)$ does not contribute to~~$\lr{\eta,[s^{-1}(0)]}$ either.
\end{proof}

\noindent
As in the complex case in Section~\ref{CRecon_subs}, 
the properties of real GW-invariants of Sections~\ref{RGWs_subs1} and~\ref{RGWs_subs2}
can be reformulated in terms of differential equations on 
generating functions for these invariants.
We continue with the notation at the end of Section~\ref{CRecon_subs}.
For the purposes of~\eref{RWDVV_e2} below,  we assume each basis element~$e_j$ for~$H^*(X;\Q)$
lies in either~$H^*_+(X;\Q)$ or $H^*_-(X;\Q)$.
Let
$$\Phi^{\om,\phi}(\ft,q) =\sum_{B\in H_2(X)}\sum_{\ell\ge 0} \frac{1}{\ell!} 
\Bigg(\sum_{\begin{subarray}{c}B'\in H_2(X;\Z)\\ \fd(B')=B\end{subarray}} 
\blr{\underset{\ell}{\underbrace{\ft,\ldots,\ft}}}_{\!0,B'}^{\!\om}
\Bigg)q^B\,.$$
The (\sf{primary}) \sf{real genus~0 GW-potential} and 
the \sf{descendant real genus~g GW-potential} of $(X,\om,\phi)$ are the formal power series
\begin{equation*}\begin{split}
\Om^{\om,\phi}(\ft,q) &=\sum_{B\in H_2(X)}\sum_{\ell\ge 0} \frac{1}{\ell!} 
\blr{\underset{\ell}{\underbrace{\ft/2,\ldots,\ft/2}}}_{\!0,B}^{\!\om,\phi} q^B
\qquad\hbox{and}\\
\cF_g^{\om,\phi}(\wt\ft,q) &=\sum_{B\in H_2(X;\Z)}\sum_{\ell\ge0} \frac{1}{\ell!} 
\blr{\underset{\ell}{\underbrace{\wt\ft/2,\ldots,\wt\ft/2}}}_{\!g,B}^{\!\om,\phi} q^B\,,
\end{split}\end{equation*}
respectively.
The \sf{full descendant GW-potential} is the formal power~series
$$\cF^{\om,\phi}(\wt\ft,q) = \sum_{g\ge0}\cF_g^{\om,\phi}(\wt\ft,q)\la^{g-1}.$$
Thus, $\Om^{\om,\phi}$ is the coefficient of $\la^{-1}$ in $\cF^{\om,\phi}$ with $t_{0j}\!=\!t_j$ 
and $t_{aj}\!=\!0$ for $a\!>\!0$.\\

\noindent
The real string relation 
\ref{Rstr_it} is equivalent to the differential equation
$$\frac{\prt\cF^{\om,\phi}}{\prt t_{01}}  = 0\,.$$
The real dilaton relation 
\ref{Rdil_it} is equivalent to the differential equation
$$\frac{\prt\cF^{\om,\phi}}{\prt t_{11}}=
\bigg\{\la\frac{\prt}{\prt\la} +  \sum_{a=0}^\infty \sum_{j=1}^N t_{aj}
\frac{\prt }{\prt t_{aj}}\bigg\} \cF^{\om,\phi} \,.$$
The relations~\eref{RWDVV_e} are equivalent to the WDVV-type differential equations
\BE{RWDVV_e2}\begin{split}
&\sum_{j,k=1}^N \!\prt_{t_{i_1}}\prt_{t_{i_2}}\prt_{t_j}\Phi^{\om,\phi}\cdot g^{jk}
\prt_{t_k}\prt_{t_{i_3}}\Om^{\om,\phi} \\
&\hspace{1.5in}=(-1)^{|e_{i_2}||e_{i_3}|}
\sum_{j,k=1}^N \!\prt_{t_{i_1}}\prt_{t_{i_3}}\prt_{t_j}\Phi^{\om,\phi}\cdot g^{jk} 
\prt_{t_k}\prt_{t_{i_2}}\Om^{\om,\phi}
\end{split}\EE
with $i_1,i_2,i_3\!=\!1,2,\ldots,N$ such that $e_{i_1}\!\in\!H^*_+(X;\Q)$
and $e_{i_2},e_{i_3}\!\in\!H^*_-(X;\Q)$.

\section{Genus reduction and splitting}
\label{GenRedSpl_sec}

\noindent
Let $(X,\om,\phi)$ and $n$ be as in Theorem~\ref{main_thm}.
For the remainder of this paper,  we fix $J\!\in\!\cJ_{\om}^{\phi}$ and omit it from 
the notation for the moduli spaces of maps.
We define the \sf{arithmetic genus}~$g$ of a closed, possibly nodal and disconnected, 
Riemann surface~$(\Si,\fj)$ by 
$$g\!-\!1=\sum_{i=1}^m(g_i\!-\!1)$$
if $g_1,\ldots,g_m$ are the arithmetic genera of the topological components of~$\Si$.

\subsection{Comparisons of orientations}
\label{CompOrient_subs}

\noindent
For $g\!\in\!\Z$, $\ell\!\in\!\Z^{\ge0}$, and \hbox{$B\!\in\!H_2(X;\Z)$}, 
we denote by $\ov\fM_{g,\ell}^{\phi;\bu}(B)$
the moduli space of stable real $J$-holomorphic degree~$B$ maps 
from closed, possibly nodal and disconnected, symmetric Riemann surfaces of arithmetic genus~$g$
with $\ell$~conjugate pairs of marked points.
Thus,
$$\ov\fM_{g,\ell}^{\phi}(B)\subset\ov\fM_{g,\ell}^{\phi;\bu}(B)$$
is a union of topological components.
For each $i\!\in\![\ell]$, let 
$$\ev_i\!:\ov\fM_{g,\ell}^{\phi;\bu}(B)\lra X \qquad\hbox{and}\qquad
\cL_i\lra \ov\fM_{g,\ell}^{\phi;\bu}(B)$$
be the natural evaluation map and the universal tangent line orbi-bundle, respectively,
at the marked point~$z_i^+$.
Let 
$$\big(\cL_i\!\otimes\!\ov{\cL_i}\big)^{\R}=
\big\{rv\!\otimes_{\C}\!v\!\in\!\cL_i\!\otimes_{\C}\!\ov{\cL_i}\!:
v\!\in\!\cL_i,\,r\!\in\!\R\big\}.$$
This real line bundle over $\ov\fM_{g,\ell}^{\phi;\bu}(B)$ is canonically oriented
by the standard orientation of~$\R$.\\

\noindent
For each element $P\!\in\!\cP(g\!-\!1,\ell)$ as in~\eref{PwtPdfn_e}, let
\BE{RioPdfn_e0}\wt\io_{P;0}^{\C}\!: 
\ov\fM_P^{\C}(B)\!\equiv\!\bigsqcup_{\begin{subarray}{c}B_1,B_2\in H_2(X;\Z)\\
B_1+B_2=B \end{subarray}}\!\!\!\!\!\!\!\!\!\!\!\!
\ov\fM_{g_1,|I|+1}^{\phi}(B_1)\!\times\!\ov\fM_{g_2,|J|+1}^{\phi}(B_2)
\lra \ov\fM_{g-2,\ell+2}^{\phi;\bu}(B)\EE
be the open embedding obtained by taking the disjoint union of the two maps
and by re-ordering the pairs of marked points 
according to the bijection~\eref{Sperm_e} with~$(I,J)$ replaced by 
$(I\!\sqcup\!\{\ell\!+\!1\},J\!\sqcup\!\{\ell\!+\!2\})$.
For disjoint subsets $I,J\!\subset\!\Z^+$ and a map
$\u\!\equiv\!(\cC,u)$ from a marked curve~$\cC$ with the underlying Riemann surface~$(\Si,\fj)$
as in~\eref{fDIJdfn_e1}, define
$$\fD_{I,J}(\u)=\big(\fD_{I,J}(\cC),\fD(u)\!:\wh\Si\!\lra\!X\big) \qquad\hbox{with}\quad
\fD(u)|_{\{+\}\times\Si}=u,~~\fD(u)|_{\{-\}\times\Si}=\phi\!\circ\!u.$$
For each $\wt{P}\!\in\!\wt\cP(g,\ell)$ as in~\eref{PwtPdfn_e}, let
\BE{RiowtPdfn_e0}\wt\io_{\wt{P};0}^{\,\C}\!: 
\ov\fM_{\wt{P}}^{\C}(B)\!\equiv\!\bigsqcup_{(B',B_0)\in\cP_{\wt{P}}^{\phi}(B)}\!\!\!\!\!\!\!\!\!\!\!
\ov\fM_{g',|I\sqcup J|+1}(B')\!\times\!\ov\fM_{g_0,|K|+1}^{\phi}(B_0)
\lra \ov\fM_{g-2,\ell+2}^{\phi;\bu}(B)\EE
be the open embedding sending $([\u'],[\u_0])$ to the equivalence class of 
the marked map obtained by taking the disjoint union 
of~$\fD_{I\sqcup\{\ell+1\},J}(\u')$
and~$\u_0$ and by re-ordering the pairs of marked points 
according to the bijection~\eref{Sperm_e} with~$(I,J)$ replaced by
$(I\!\sqcup\!J\!\sqcup\!\{\ell+1\},K\!\sqcup\!\{\ell+2\})$.
If \hbox{$[\ell]\!=\!I\!\sqcup\!J$}, let
\BE{RiowtPEdfn_e0}\begin{split}
\wt\io_{g,(I,J);0}^{\,\C}\!: 
\ov\fM_{(g-1)/2,\ell+2}^{\C}(B)\equiv\!
\bigsqcup_{\begin{subarray}{c}B'\in H_2(X;\Z)\\ B'-\phi_*(B')=B \end{subarray}}\!\!\!\!\!\!\!\!\!\!
\ov\fM_{(g-1)/2,\ell+2}(B')&\lra \ov\fM_{g-2,\ell+2}^{\phi;\bu}(B) \quad\hbox{and}\\
\wt\io_{g,(I,J);0}^{\,E}\!: 
\ov\fM_{g/2,\ell+1}^E(B)\equiv\!
\bigsqcup_{\begin{subarray}{c}B'\in H_2(X;\Z)\\ B'-\phi_*(B')=B \end{subarray}}\!\!\!\!\!\!\!\!\!\!
\ov\fM_{g/2,\ell+1}(B')&\lra \ov\fM_{g-1,\ell+1}^{\phi;\bu}(B)
\end{split}\EE
be the open embedding sending $[\u']$ to the equivalence class 
of~$\fD_{I\sqcup\{\ell+1\},J\sqcup\{\ell+2\}}(\u')$ and 
the open embedding sending $[\u']$ to the equivalence class of~$\fD_{I\sqcup\{\ell+1\},J}(\u')$.
If $g\!\in\!2\Z$ (resp.~$g\!\not\in\!2\Z$), 
we define $\ov\fM_{(g-1)/2,\ell+2}^{\C}(B)$ (resp.~$\ov\fM_{g/2,\ell+1}^E(B)$)
to be the empty~set.\\

\noindent
Define
\begin{equation*}\begin{split}
\ov\fM_{g-2,\ell+2}'^{\phi;\bu}(B)
&=\big\{[\u]\!\in\!\ov\fM_{g-2,\ell+2}^{\phi;\bu}(B)\!:\,
\ev_{\ell+1}([\u])\!=\!\ev_{\ell+2}([\u])\!\big\},\\
\ov\fM_{g-2,\ell+2}''^{\phi;\bu}(B)
&=\big\{[\u]\!\in\!\ov\fM_{g-1,\ell+1}^{\phi;\bu}(B)\!:\,
\ev_{\ell+1}([\u])\!\in\!X^{\phi}\!\big\}.
\end{split}\end{equation*}
The short exact sequences
\begin{equation*}\begin{split}
0&\lra T\ov\fM_{g-2,\ell+2}'^{\phi;\bu}(B)\lra 
T\ov\fM_{g-2,\ell+2}^{\phi;\bu}(B)\!\big|_{\ov\fM_{g-2,\ell+2}'^{\phi;\bu}(B)} 
\lra \ev_{\ell+1}^{\,*}TX\lra 0,\\
0&\lra T\ov\fM_{g-1,\ell+1}''^{\phi;\bu}(B)\lra 
T\ov\fM_{g-1,\ell+1}^{\phi;\bu}(B)\!\big|_{\ov\fM_{g-1,\ell+1}''^{\phi;\bu}(B)} 
\lra \ev_{\ell+1}^{\,*}\cN_XX^{\phi}\lra 0
 \end{split}\end{equation*}
induce isomorphisms
\BE{SubIsom_e}\begin{split}
\La_{\R}^{\top}\big(T\big(\ov\fM_{g-2,\ell+2}^{\phi;\bu}(B)\!\big)\!\big)
\!\big|_{\ov\fM_{g-2,\ell+2}'^{\phi;\bu}(B)}
&\approx 
\La_{\R}^{\top}\big(T\big(\ov\fM_{g-2,\ell+2}'^{\phi;\bu}(B)\!\big)\!\big)
\!\otimes\! \ev_{\ell+1}^*\big(\La_{\R}^{\top}(TX)\!\big),\\
\La_{\R}^{\top}\big(T\big(\ov\fM_{g-1,\ell+1}^{\phi;\bu}(B)\!\big)\!\big)
\!\big|_{\ov\fM_{g-1,\ell+1}''^{\phi;\bu}(B)}
&\approx 
\La_{\R}^{\top}\big(T\big(\ov\fM_{g-1,\ell+1}''^{\phi;\bu}(B)\!\big)\!\big)
\!\otimes\! \ev_{\ell+1}^*\big(\La_{\R}^{\top}(\cN_XX^{\phi})\!\big)
\end{split}\EE
of real line bundles over $\ov\fM_{g-2,\ell+2}'^{\phi;\bu}(B)$
and $\ov\fM_{g-1,\ell+1}''^{\phi;\bu}(B)$, respectively.\\

\noindent
We denote by 
\BE{Riogldfn_e2}
\wt\io_{g,\ell}'\!:\ov\fM_{g-2,\ell+2}'^{\phi;\bu}(B)\lra\ov\fM_{g,\ell}^{\phi;\bu}(B)
\quad\hbox{and}\quad
\wt\io_{g,\ell}''\!:\ov\fM_{g-1,\ell+1}''^{\phi;\bu}(B)\lra\ov\fM_{g,\ell}^{\phi;\bu}(B)\EE
the immersion obtained by identifying the marked points~$z_{\ell+1}^+$ and~$z_{\ell+1}^-$ 
of the domain of each map with~$z_{\ell+2}^+$ and~$z_{\ell+2}^-$, respectively,
to form a conjugate pair of nodes
and the immersion obtained by identifying the marked point~$z_{\ell+1}^+$ 
of the domain of each map with~$z_{\ell+1}^-$ to form an $E$-node.
The first immersion is generically $4\!:\!1$ onto its image,
while the second is generically $2\!:\!1$.
There are canonical isomorphisms
\BE{fMcNioisom_e}\begin{split}
\cN\wt\io_{g,\ell}'&\equiv 
\frac{\wt\io_{g,\ell}'^{\,*}T(\ov\fM_{g,\ell}^{\phi;\bu}(B)\!)}
{\nd\wt\io_{g,\ell}'(T(\ov\fM_{g-2,\ell+2}'^{\phi;\bu}(B)\!)\!)}
\approx \cL_{\ell+1}\!\otimes_{\C}\!\cL_{\ell+2},\\
\cN\wt\io_{g,\ell}''&\equiv 
\frac{\wt\io_{g,\ell}''^{\,*}T(\ov\fM_{g,\ell}^{\phi;\bu}(B)\!)}
{\nd\io_{g,\ell}''(T(\ov\fM_{g-1,\ell+1}''^{\phi;\bu}(B)\!)\!)}
\approx \big(\cL_{\ell+1}\!\otimes\!\ov{\cL_{\ell+1}}\big)^{\R}\,.
\end{split}\EE 
They induce isomorphisms
\BE{CompOrient_e0}\begin{split} 
\wt\io_{g,\ell}'^{\,*}\big(\La_{\R}^{\top}
\big(T(\ov\fM_{g,\ell}^{\phi;\bu}(B)\!)\!\big)\!\big)
&\approx \La_{\R}^{\top}\big(T\big(\ov\fM_{g-2,\ell+2}'^{\phi;\bu}(B)\!\big)\!\big)
\!\otimes\! \La_{\R}^2\big(\cL_{\ell+1}\!\otimes_{\C}\!\cL_{\ell+2}\big),\\
\wt\io_{g,\ell}''^{\,*}\big(\La_{\R}^{\top}
\big(T(\ov\fM_{g,\ell}^{\phi;\bu}(B)\!)\!\big)\!\big)
&\approx \La_{\R}^{\top}\big(T\big(\ov\fM_{g-1,\ell+1}''^{\phi;\bu}(B)\!\big)\!\big)
\!\otimes\! \big(\cL_{\ell+1}\!\otimes\!\ov{\cL_{\ell+1}}\big)^{\R}
\end{split}\EE
of real line bundles over $\ov\fM_{g-2,\ell+2}'^{\phi;\bu}(B)$
and $\ov\fM_{g-1,\ell+1}''^{\phi;\bu}(B)$, respectively.\\

\noindent
The symplectic orientation on~$X$ and the orientation on~$X^{\phi}$ induced by
a real orientation~$(L,[\psi],\fs)$ on~$(X,\om,\phi)$  
determines an orientation on~$\cN_XX^{\phi}$ via the short exact sequence
$$0\lra TX^{\phi}\lra TX|_{X^{\phi}}\lra \cN_XX^{\phi}\lra0$$
of real vector bundles over~$X^{\phi}$.
By \cite[Section~3.3]{RealGWsGeom}, $(L,[\psi],\fs)$  
endows the moduli space $\ov\fM_{g,\ell}^{\phi;\bu}(B)$ with an orientation and
a virtual fundamental class of dimension~\eref{RfMdim_e}.
The space~$\ov\fM_{g,\ell}^{\phi;\bu}(B)$ naturally splits into unions
of topological components, such as $\ov\fM_{g,\ell}^{\phi}(B)$
and the images of~$\wt\io_{P;0}^{\,\C}$ with $P\!\in\!\cP(g\!+\!1,\ell\!-\!2)$,
$\wt\io_{\wt{P};0}^{\,\C}$ with $\wt{P}\!\in\!\wt\cP(g\!+\!2,\ell\!-\!2)$,
$\wt\io_{g+2,(I,J);0}^{\,\C}$ with \hbox{$[\ell\!-\!2]\!=\!I\!\sqcup\!J$},
and $\wt\io_{g+1,(I,J);0}^{\,E}$ with \hbox{$[\ell\!-\!1]\!=\!I\!\sqcup\!J$}.
By the first statements of \cite[Propositions~3.12/3]{RealGWsGeom}, 
the signs~$(-1)^{\ve}$ of the open embeddings in~\eref{RioPdfn_e0}-\eref{RiowtPEdfn_e0}
with respect to the orientations induced by $(L,[\psi],\fs)$ 
and the complex orientations on the complex moduli space $\ov\fM_{g',\ell'}(B')$ 
are given~by
\BE{sgnwtio_e}\begin{split}
&\ve_{P;0}^{\C}=\ve_n(P), 
\qquad \ve_{\wt{P};0}^{\C}
\big|_{\ov\fM_{g',|I\sqcup J|+1}(B')\times\ov\fM_{g_0,|K|+1}^{\phi}(B_0)}
=\blr{c_1(L),\phi_*(B')\!}\!+\!|J|,\\
&\hspace{.6in}\ve_{g,(I,J);0}^{\,\C}\big|_{\ov\fM_{(g-1)/2,\ell+2}(B')}
=\blr{c_1(L),\phi_*(B')\!}\!+\!|J|\!+\!1, \\
&\hspace{.8in}\ve_{g,(I,J);0}^{\,E}\big|_{\ov\fM_{g/2,\ell+1}(B')}
=\blr{c_1(L),\phi_*(B')\!}\!+\!|J|,
\end{split}\EE
respectively.\\

\noindent
Along with the canonical orientations on~$\cL_{\ell+1}\!\otimes_{\C}\!\cL_{\ell+2}$
and~$(\cL_{\ell+1}\!\otimes\!\ov{\cL_{\ell+1}})^{\R}$,
$(L,[\psi],\fs)$ also determines orientations on~$\ov\fM_{g-2,\ell+2}'^{\phi;\bu}(B)$
and~$\ov\fM_{g-1,\ell+1}''^{\phi;\bu}(B)$ via~\eref{SubIsom_e} and~\eref{CompOrient_e0}.
By the first statement of \cite[Proposition~3.14]{RealGWsGeom}, 
the orientations on~$\ov\fM_{g-2,\ell+2}'^{\phi;\bu}(B)$
defined via the first isomorphisms in~\eref{SubIsom_e} and~\eref{CompOrient_e0} are the same.
By the first statement of \cite[Proposition~3.16]{RealGWsGeom}, 
the orientations on~$\ov\fM_{g-1,\ell+1}''^{\phi;\bu}(B)$
defined via the second isomorphisms in~\eref{SubIsom_e} and~\eref{CompOrient_e0} are opposite.

\subsection{Proof of~\ref{RGenusRed_it} and~\ref{RSplit_it}}
\label{RSplitPf_subs}

\noindent
Fix a tuple $\mu\!\equiv\!(\mu_i)_{i\in[\ell]}$ of elements of $H^*(X;\Q)$.
Let  $\ga'\!\in\!H^*(\R\ov\cM_{g-1,\ell+1};\Q)$ and
$$\ov\fM_{g-1,\ell+1}''^{\phi}(B)=\ov\fM_{g-1,\ell+1}''^{\phi;\bu}(B)
\!\cap\!\ov\fM_{g-1,\ell+1}^{\phi}(B).$$
By the definition of~$\cI_{g,\ell,B}^{\om,\phi}$ below~\eref{RIdfn_e} and~\eref{cupcap_e},
\BE{SplitE_e}\begin{split}
&\int_{\R\ov\cM_{g-1,\ell+1}}\!\!\ga'
\cI_{g-1,\ell+1,B}^{\om,\phi}\big(\mu,\PD_X^{-1}\big([X^{\phi}]_X\big)\!\big)\\
&\hspace{1in}=
\int_{[\ov\fM_{g-1,\ell+1}^{\phi}(B)]^{\vir}}\!\! (\ff^*\ga')
\big(\ev_1^*\mu_1\big)\ldots\!\big(\ev_{\ell}^*\mu_{\ell}\big)
\big(\ev_{\ell+1}^{\,*}\big(\PD_X^{-1}\big([X^{\phi}]_X\big)\!\big)\!\big)\\
&\hspace{1in}
=\int_{[\ov\fM_{g-1,\ell+1}''^{\phi}(B)]^{\vir}}\! (\ff^*\ga')
\wt\io_{g,\ell}''^{\,*}
\big(\!(\ev_1^*\mu_1)\ldots\!(\ev_{\ell}^*\mu_{\ell})\!\big)\,.
\end{split}\EE
The last virtual fundamental class above is taken with respect to the orientation
on~$\ov\fM_{g-1,\ell+1}''^{\phi}(B)$ induced
by~$(L,[\psi],\fs)$ via the second isomorphism in~\eref{SubIsom_e}.
By~\eref{pushintprp_e} with $r\!=\!1$ and~\ref{Rgrad_it},
\BE{SplitE_e2}\begin{split}
&\int_{\R\ov\cM_{g-1,\ell+1}}\!\!\ga'
\big(\io_{g,\ell}^E\big)^{\!*}\!
\big(\cI_{g,\ell,B}^{\om,\phi}(\mu)\!\big)
=(-1)^{g-1+|\mu|}\!\!\int_{\R\ov\cM_{g,\ell}}\!\!\big(\io_{g,\ell}^E\big)\!_*\!(\ga')
\cI_{g,\ell,B}^{\om,\phi}(\mu)\\
&\hspace{1.2in}=(-1)^{g-1+|\mu|}\!\!\int_{[\ov\fM_{g,\ell}^{\phi}(B;J)]^{\vir}}\! 
\big(\ff^*\!\big(\!(\io_{g,\ell}^E)_*(\ga')\!\big)\!\big)
(\ev_1^*\mu_1)\ldots\!(\ev_{\ell}^*\mu_{\ell})\,.
\end{split}\EE

\vspace{.18in}

\noindent
The restriction~$\wt\io_{g,\ell}^E$ of $\wt\io_{g,\ell}''$ to~$\ov\fM_{g-1,\ell+1}''^{\phi}(B)$
is the fiber product of the forgetful morphism~$\ff$ in~\eref{Rffdfn_e}
and the immersion~$\io_{g,\ell}^E$ in~\eref{Riogldfn_e}.
The bundle homomorphism 
$$\nd\ff\!: \cN\wt\io_{g,\ell}^E\!=\!\cN\wt\io_{g,\ell}''\big|_{\ov\fM_{g-1,\ell+1}''^{\phi}(B)}
\lra  \big\{\ff|_{\ov\fM_{g-1,\ell+1}''^{\phi}(B)}\big\}^{\!*}\cN\io_{g,\ell}^E$$
induced by the differential of~$\ff$, with $\ff$ on the right-hand side 
as in~\eref{Rffdfn_e} with~$(g,\ell)$ replaced by \hbox{$(g\!-\!1,\ell\!+\!1)$}, 
is generically an isomorphism that intertwines the two canonical orientations.
Thus, the integral on the right-hand side of~\eref{SplitE_e2} equals to the right-hand side 
in~\eref{SplitE_e} with the orientation on~$\ov\fM_{g-1,\ell+1}''^{\phi}(B)$
induced by~$(L,[\psi],\fs)$ via the second isomorphism in~\eref{CompOrient_e0}.
Along with the first statement of \cite[Proposition~3.16]{RealGWsGeom},
this implies the second claim in~\ref{RGenusRed_it}.\\

\noindent
Let $[\ell]\!=\!I\!\sqcup\!J$ and $g\!\in\!\Z^{\ge0}$.
The~spaces 
\begin{equation*}\begin{split}
\ov\fM_{(g-1)/2,\ell+1}'(B)&\equiv\big\{[\u]\!\in\!\ov\fM_{(g-1)/2,\ell+2}^{\C}(B)\!:\,
\ev_{\ell+1}([\u])\!=\!\phi\big(\ev_{\ell+1}([\u])\!\big)\!\big\}\qquad\hbox{and}\\
\ov\fM_{g/2,\ell+1}''(B)&\equiv\big\{[\u]\!\in\!\ov\fM_{g/2,\ell+1}^E(B)\!:\,
\ev_{\ell+1}([\u])\!\in\!X^{\phi}\!\big\}
\end{split}\end{equation*}
inherit orientations from the complex orientations of 
$\ov\fM_{(g-1)/2,\ell+2}^{\C}(B)$ and $\ov\fM_{g/2,\ell+1}^E(B)$
via the analogues of the isomorphisms in~\eref{SubIsom_e}.
Furthermore,
$$\wt\io_{g,(I,J);0}^{\,\C}\big(\ov\fM_{(g-1)/2,\ell+2}'(B)\!\big)
\subset \ov\fM_{g-2,\ell+2}'^{\phi;\bu}(B)
\quad\hbox{and}\quad
\wt\io_{g,(I,J);0}^{\,E}\big(\ov\fM_{g/2,\ell+1}''(B)\!\big)
\subset \ov\fM_{g-1,\ell+1}''^{\phi;\bu}(B).$$

\vspace{.18in}

\noindent
Suppose in addition $g\!\in\!2\Z^{\ge0}$ and $\ga'\!\in\!H^*(\ov\cM_{g/2,\ell+1};\Q)$.
The composition~$\wt\io_{g,(I,J)}^{\,E}$ of the restriction of~$\wt\io_{g,(I,J);0}^{\,E}$ 
to~$\ov\fM_{g/2,\ell+1}''(B)$ with~$\wt\io_{g,\ell}''$ 
is the fiber product of the forgetful morphism~$\ff$ in~\eref{Rffdfn_e}
and the immersion~$\io_{g,(I,J)}^E$ in~\eref{RioIJEdfn_e}.
The bundle homomorphism 
$$\nd\ff\!: \cN\wt\io_{g,(I,J)}^{\,E}\!\approx\!
\wt\io_{g,(I,J);0}^{\,E\,*}\cN\wt\io_{g,\ell}''\big|_{\ov\fM_{g/2,\ell+1}''(B)}
\lra  \big\{\ff|_{\ov\fM_{g/2,\ell+1}''(B)}\big\}^{\!*}\cN\io_{g,(I,J)}^E$$
induced by the differential of~$\ff$, with $\ff$ on the right-hand side as in~\eref{ffdfn_e} 
with~$(g,\ell)$ replaced by \hbox{$(g/2,\ell\!+\!1)$}, 
is generically an isomorphism that intertwines the two canonical orientations.
By the first statements of \cite[Propositions~3.13/6]{RealGWsGeom},
the orientation on $\ov\fM_{g/2,\ell+1}''(B)$
induced from the canonical orientation of~$\cN\wt\io_{g,(I,J)}^{\,E}$
via the analogue of the second isomorphism in~\eref{CompOrient_e0} differs from
the above orientation by~$-1$ to the power of~$\ve_{g,(I,J);0}^{\,E}\!+\!1$.
Along with~\eref{SplitE_e} and~\eref{SplitE_e2} with
$$\R\ov\cM_{g-1,\ell+1}, \quad \cI_{g-1,\ell+1,B}^{\om,\phi}, \quad 
\ov\fM_{g-1,\ell+1}^{\phi}(B),\quad \ov\fM_{g-1,\ell+1}''^{\phi}(B), \quad
\wt\io_{g,\ell}'', ~~\hbox{and}~~ \io_{g,\ell}^E$$ 
replaced~by 
$$\ov\cM_{g/2,\ell+1}, \quad \cI_{g/2,(I,J),B'}^{\om,\phi,L}, \quad 
\ov\fM_{g/2,\ell+1}(B'), 
\quad \ov\fM_{g/2,\ell+1}''(B)\!\cap\!\ov\fM_{g/2,\ell+1}(B'),
\quad\wt\io_{g,(I,J)}^E, ~~\hbox{and}~~ \io_{g,(I,J)}^E$$
respectively, this implies the last claim in~\ref{RSplit_it}.\\

\noindent
Let  $\ga'\!\in\!H^*(\R\ov\cM_{g-2,\ell+2};\Q)$ and
$$\ov\fM_{g-2,\ell+2}'^{\phi}(B)=\ov\fM_{g-2,\ell+2}'^{\phi;\bu}(B)
\!\cap\!\ov\fM_{g-2,\ell+2}^{\phi}(B).$$
By the definition of~$\cI_{g,\ell,B}^{\om,\phi}$ below~\eref{RIdfn_e},
\eref{DeXPD_e}, and~\eref{cupcap_e},
\BE{SplitC_e}\begin{split}
&\sum_{i,j=1}^N\!g^{ij}\!\!\!\int_{\R\ov\cM_{g-2,\ell+2}}\!\!\ga'
\cI_{g-2,\ell+2,B}^{\om,\phi}(\mu,e_i,e_j)\\
&\hspace{.5in}=
\int_{[\ov\fM_{g-2,\ell+2}^{\phi}(B)]^{\vir}}\!\! (\ff^*\ga')
\big(\ev_1^*\mu_1\big)\ldots\!\big(\ev_{\ell}^*\mu_{\ell}\big)
\big(\!\big\{\ev_{\ell+1}\!\times\!\ev_{\ell+2}\big\}^{\!*}\big(\PD_{X^2}^{-1}(\De_X)\!\big)\!\big)\\
&\hspace{.5in}
=\int_{[\ov\fM_{g-2,\ell+2}'^{\phi}(B)]^{\vir}}\! (\ff^*\ga')
\wt\io_{g,\ell}'^{\,*}
\big(\!(\ev_1^*\mu_1)\ldots\!(\ev_{\ell}^*\mu_{\ell})\!\big)\,.
\end{split}\EE
The last virtual fundamental class above is taken with respect to the orientation
on~$\ov\fM_{g-2,\ell+2}'^{\phi}(B)$ induced
by~$(L,[\psi],\fs)$ via the first isomorphism in~\eref{SubIsom_e}.
By~\eref{pushintprp_e} with $r\!=\!2$,
\BE{SplitC_e2}\begin{split}
\int_{\R\ov\cM_{g-2,\ell+2}}\!\!\ga'
\big(\io_{g,\ell}^{\C}\big)^{\!*}\!\big(\cI_{g,\ell,B}^{\om,\phi}(\mu)\!\big)
&=\int_{\R\ov\cM_{g,\ell}}\!\!\big(\io_{g,\ell}^{\C}\big)\!_*\!(\ga')
\cI_{g,\ell,B}^{\om,\phi}(\mu)\\
&=\int_{[\ov\fM_{g,\ell}^{\phi}(B;J)]^{\vir}}\! 
\big(\ff^*\!\big(\!(\io_{g,\ell}^{\C})_*(\ga')\!\big)\!\big)
(\ev_1^*\mu_1)\ldots\!(\ev_{\ell}^*\mu_{\ell})\,.
\end{split}\EE

\vspace{.18in}

\noindent
The restriction~$\wt\io_{g,\ell}^{\C}$ of $\wt\io_{g,\ell}'$ to~$\ov\fM_{g-2,\ell+2}'^{\phi}(B)$
is the fiber product of the forgetful morphism~$\ff$ in~\eref{Rffdfn_e}
and the immersion~$\io_{g,\ell}^{\C}$ in~\eref{Riogldfn_e}.
The bundle homomorphism 
$$\nd\ff\!: \cN\wt\io_{g,\ell}^{\C}\!=\!\cN\wt\io_{g,\ell}'\big|_{\ov\fM_{g-2,\ell+2}'^{\phi}(B)}
\lra  \big\{\ff|_{\ov\fM_{g-2,\ell+2}'^{\phi}(B)}\big\}^{\!*}\cN\io_{g,\ell}^{\C}$$
induced by the differential of~$\ff$, with $\ff$ on the right-hand side 
as in~\eref{Rffdfn_e} with~$(g,\ell)$ replaced by \hbox{$(g\!-\!2,\ell\!+\!2)$}, 
is $\C$-linear with respect to the identifications in~\eref{Riogldisom_e} and~\eref{fMcNioisom_e}
and is generically an isomorphism.
Thus, the right-hand side of~\eref{SplitC_e2} equals to the right-hand side 
in~\eref{SplitC_e} with the orientation on~$\ov\fM_{g-2,\ell+2}'^{\phi}(B)$
induced by~$(L,[\psi],\fs)$ via the first isomorphism in~\eref{CompOrient_e0}.
Along with the first statement of \cite[Proposition~3.14]{RealGWsGeom},
this implies the first claim in~\ref{RGenusRed_it}.\\

\noindent
Suppose $[\ell]\!=\!I\!\sqcup\!J$, $g\!\in\!\Z^+\!-\!2\Z$, and 
$\ga'\!\in\!H^*(\ov\cM_{(g-1)/2,\ell+2};\Q)$.
The composition~$\wt\io_{g,(I,J)}^{\,\C}$ of the restriction of~$\wt\io_{g,(I,J);0}^{\,\C}$ 
to~$\ov\fM_{(g-1)/2,\ell+1}'(B)$ with~$\wt\io_{g,\ell}'$ 
is the fiber product of the forgetful morphism~$\ff$ in~\eref{Rffdfn_e}
and the immersion~$\io_{g,(I,J)}^{\C}$ in~\eref{RioIJEdfn_e}.
The bundle homomorphism 
$$\nd\ff\!: \cN\wt\io_{g,(I,J)}^{\,\C}\!\approx\!
\wt\io_{g,(I,J);0}^{\,\C\,*}\cN\wt\io_{g,\ell}'\big|_{\ov\fM_{(g-1)/2,\ell+2}'(B)}
\lra  \big\{\ff|_{\ov\fM_{(g-1)/2,\ell+2}'(B)}\big\}^{\!*}\cN\io_{g,(I,J)}^{\C}$$
induced by the differential of~$\ff$, with $\ff$ on the right-hand side as in~\eref{ffdfn_e} 
with~$(g,\ell)$ replaced by \hbox{$(\!(g\!-\!1)/2,\ell\!+\!2)$}, 
is $\C$-linear with respect to the identifications in~\eref{RioPprp_e} and~\eref{fMcNioisom_e}
and is generically an isomorphism.
By the first statements of \cite[Propositions~3.13/4]{RealGWsGeom},
the orientation on $\ov\fM_{(g-1)/2,\ell+2}'(B)$
induced from the complex orientation of~$\cN\wt\io_{g,(I,J)}^{\,\C}$
via the analogue of the first isomorphism in~\eref{CompOrient_e0} differs from
the orientation induced via the analogue of the first isomorphism in~\eref{SubIsom_e}
by~$-1$ to the power of~$\ve_{g,(I,J);0}^{\,\C}$.
Along with~\eref{SplitC_e} and~\eref{SplitC_e2} with
$$\R\ov\cM_{g-2,\ell+2}, \quad \cI_{g-2,\ell+2,B}^{\om,\phi}, \quad 
\ov\fM_{g-2,\ell+2}^{\phi}(B),\quad \ov\fM_{g-2,\ell+2}'^{\phi}(B), \quad
\wt\io_{g,\ell}', ~~\hbox{and}~~ \io_{g,\ell}^{\C}$$ 
replaced~by 
$$\ov\cM_{(g-1)/2,\ell+2}, ~~ \cI_{(g-1)/2,(I,J),B'}^{\om,\phi,L}, ~~ 
\ov\fM_{(g-1)/2,\ell+2}(B'), 
~~ \ov\fM_{(g-1)/2,\ell+1}'(B)\!\cap\!\ov\fM_{(g-1)/2,\ell+2}(B'),$$
$\wt\io_{g,(I,J)}^{\C}$, and $\io_{g,(I,J)}^{\C}$,
respectively, this implies the third claim in~\ref{RSplit_it}.\\

\noindent
For $I\!\equiv\!\{i_1,\ldots,i_m\}\!\subset\![\ell]$ with $i_1\!<\!\ldots\!<\!i_m$, let
$$\ev_I^*\mu=\big(\ev_1^{\,*}\mu_{i_1}\big)\!\ldots\!\big(\ev_m^{\,*}\mu_{i_m}\big).$$
For $P\!\in\!\cP(g\!-\!1,\ell)$ as in~\eref{PwtPdfn_e}, define
\begin{equation*}\begin{split}
\ov\fM_P'^{\C}(B)&=\big\{\!\big([\u_1],[\u_2]\big)\!\in\!\ov\fM_P^{\C}(B)\!:
\ev_{|I|+1}([\u_1])\!=\!\ev_{|J|+1}([\u_2])\!\big\},\\
\big[\ov\fM_P^{\C}(B)\big]^{\vir}
&=\sum_{\begin{subarray}{c}B_1,B_2\in H_2(X;\Z)\\ B_1+B_2=B \end{subarray}}
\!\!\!\!\!\!\!\!\!\!\!\!\big[\ov\fM_{g_1,|I|+1}^{\phi}(B_1)\big]^{\vir}\!\times\!
\big[\ov\fM_{g_2,|J|+1}^{\phi}(B_2)\big]^{\vir},\\
\big[\ov\fM_P'^{\C}(B)\big]^{\vir}
&=\big\{\ev_{|I|+1}\!\times\!\ev_{|J|+1}\big\}^{\!*}\!\big(\PD_{X^2}^{-1}(\De_X)\!\big)
\!\cap\!\big[\ov\fM_P^{\C}(B)\big]^{\vir}.
\end{split}\end{equation*}
For $\wt{P}\!\in\!\wt\cP(g,\ell)$ as in~\eref{PwtPdfn_e}, we similarly define
\begin{equation*}\begin{split}
\ov\fM_{\wt{P}}'^{\C}(B)&=\big\{\!\big([\u'],[\u_0]\big)\!\in\!\ov\fM_{\wt{P}}^{\C}(B)\!:
\ev_{|I\sqcup J|+1}([\u'])\!=\!\ev_{|K|+1}([\u_0])\!\big\},\\
\big[\ov\fM_{\wt{P}}^{\C}(B)\big]^{\vir}
&=\sum_{(B',B_0)\in\cP_{\wt{P}}^{\phi}(B)}\!\!\!\!\!\!\!\!\!\!\!
\big[\ov\fM_{g',|I\sqcup J|+1}(B')\big]^{\vir}\!\times\!
\big[\ov\fM_{g_0,|K|+1}^{\phi}(B_0)\big]^{\vir},\\
\big[\ov\fM_{\wt{P}}'^{\C}(B)\big]^{\vir}
&=\big\{\ev_{|I\sqcup J|+1}\!\times\!\ev_{|K|+1}\big\}^{\!*}\!\big(\PD_{X^2}^{-1}(\De_X)\!\big)
\!\cap\!\big[\ov\fM_{\wt{P}}^{\C}(B)\big]^{\vir}\,.
\end{split}\end{equation*}

\vspace{.18in}

\noindent
With $P\!\in\!\cP(g\!-\!1,\ell)$ as in~\eref{PwtPdfn_e}, let 
$$\ga_1\!\in\!H^*(\R\ov\cM_{g_1,|I|+1};\Q) \quad\hbox{and}\quad
\ga_2\!\in\!H^*(\R\ov\cM_{g_2,|J|+1};\Q).$$
By the definition of~$\cI_{g,\ell,B}^{\om,\phi}$ below~\eref{RIdfn_e},
\eref{FCprod_e}, \ref{Rgrad_it}, \ref{Rcov_it}, \eref{DeXPD_e}, and~\eref{cupcap_e},
\BE{SplitCP_e}\begin{split}
&\sum_{\begin{subarray}{c}B_1,B_2\in H_2(X;\Z)\\ B_1+B_2=B \end{subarray}}
\sum_{i,j=1}^N\!g^{ij}\!\!\!\int_{\R\ov\cM_{g_1,|I|+1}\times\R\ov\cM_{g_2,|J|+1}}\!\!\!\!\!
\big(\ga_1\!\times\!\ga_2\big)
\big(\cI_{g_1,|I|+1,B_1}^{\om,\phi}(\mu_I,e_i)\!\times\!\cI_{g_2,|J|+1,B_2}^{\om,\phi}(e_j,\mu_J)\!\big)\\
&=(-1)^{(g_1-1)|\ga_2|}\!\!\!\!
\int_{[\ov\fM_P^{\C}(B)]^{\vir}}\!\!\!\!\big(\!(\ff^*\ga_1)\!\times\!(\ff^*\ga_2)\!\big)
\big(\!(\ev_I^*\mu)\!\times\!(\ev_J^*\mu)\!\big)
\big(\!\big\{\ev_{|I|+1}\!\times\!\ev_{|J|+1}\big\}^{\!*}\!\big(\PD_{X^2}^{-1}(\De_X)\!\big)\!\big)\\
&=(-1)^{(g_1-1)|\ga_2|+\ve(P,\mu)}\!\!\!\!\int_{[\ov\fM_P'^{\C}(B)]^{\vir}}
\!\!\!\!\big(\!\{\ff\!\times\!\ff\}^*(\ga_1\!\times\!\ga_2)\!\big)
\wt\io_P^{\C\,*}\big(\!(\ev_1^*\mu_1)\ldots\!(\ev_{\ell}^*\mu_{\ell})\!\big)\,.
\end{split}\EE
By~\eref{pushintprp_e} with $r\!=\!2$,
\BE{SplitCP_e2}\begin{split}
\int_{\R\ov\cM_{g_1,|I|+1}\times\R\ov\cM_{g_2,|J|+1}}\!\!(\ga_1\!\times\!\ga_2)
\big(\io_P^{\C}\big)^{\!*}\!\big(\cI_{g,\ell,B}^{\om,\phi}(\mu)\!\big)
=\int_{\R\ov\cM_{g,\ell}}\!\!\big(\io_P^{\C}\big)\!_*\!(\ga_1\!\times\!\ga_2)
\cI_{g,\ell,B}^{\om,\phi}(\mu)&\\
\hspace{1in}=\int_{[\ov\fM_{g,\ell}^{\phi}(B;J)]^{\vir}}\! 
\big(\ff^*\!\big(\!(\io_P^{\C})_*(\ga_1\!\times\!\ga_2)\!\big)\!\big)
(\ev_1^*\mu_1)\ldots\!(\ev_{\ell}^*\mu_{\ell})&\,.
\end{split}\EE

\vspace{.18in}

\noindent
The image of $\ov\fM_P'^{\C}(B)$ under $\wt\io_{P;0}^{\C}$ is contained 
in $\ov\fM_{g-2,\ell+2}'^{\phi;\bu}(B)$.
The composition~$\wt\io_P^{\C}$ of the restriction of~$\wt\io_{P;0}^{\C}$ 
to~$\ov\fM_P'^{\C}(B)$ with~$\wt\io_{g,\ell}'$ 
is the fiber product of the forgetful morphism~$\ff$ in~\eref{Rffdfn_e}
and the immersion~$\io_P^{\C}$ in~\eref{RioPdfn_e}.
The bundle homomorphism 
$$\nd\ff\!: \cN\wt\io_P^{\C}\!\approx\!
\wt\io_{P;0}^{\C\,*}\cN\wt\io_{g,\ell}'\big|_{\ov\fM_P'^{\C}(B)}
\lra  \big\{\ff\!\times\!\ff|_{\ov\fM_P'^{\C}(B)}\big\}^{\!*}\cN\io_P^{\C}$$
induced by the differential of~$\ff$, 
with $\ff$ on the right-hand side as in~\eref{Rffdfn_e} 
with~$(g,\ell)$ replaced by \hbox{$(g_1,|I|\!+\!1)$} and \hbox{$(g_2,|J|\!+\!1)$},
is $\C$-linear with respect to the identifications in~\eref{RioPprp_e} and~\eref{fMcNioisom_e}
and is generically an isomorphism.
Thus, the right-hand side of~\eref{SplitCP_e2} equals to the last integral
in~\eref{SplitCP_e} with the orientation on~$\ov\fM_P'^{\C}(B)$
induced from the complex orientation of~$\cN\wt\io_P^{\C}$
via the analogue of the first isomorphism in~\eref{CompOrient_e0}.
Along with the first statements of \cite[Propositions~3.12/4]{RealGWsGeom},
this implies the first claim in~\ref{RSplit_it}.\\

\noindent
Suppose now that $\wt{P}\!\in\!\wt\cP(g,\ell)$ is as in~\eref{PwtPdfn_e}, 
$$\ga_1\!\in\!H^*(\ov\cM_{g',|I\sqcup J|+1};\Q) \quad\hbox{and}\quad 
\ga_2\!\in\!H^*(\R\ov\cM_{g_0,|K|+1};\Q).$$
The image of $\ov\fM_{\wt{P}}'^{\C}(B)$ under $\wt\io_{\wt{P};0}^{\C}$ is 
again contained  in $\ov\fM_{g-2,\ell+2}'^{\phi;\bu}(B)$.
The composition~$\wt\io_{\wt{P}}^{\C}$ of the restriction of~$\wt\io_{\wt{P};0}^{\C}$ 
to~$\ov\fM_{\wt{P}}'^{\C}(B)$ with~$\wt\io_{g,\ell}'$ 
is the fiber product of the forgetful morphism~$\ff$ in~\eref{Rffdfn_e}
and the immersion~$\io_{\wt{P}}^{\C}$ in~\eref{RiowtPdfn_e}.
The bundle homomorphism 
$$\nd\ff\!: \cN\wt\io_{\wt{P}}^{\C}\!\approx\!
\wt\io_{\wt{P};0}^{\C\,*}\cN\wt\io_{g,\ell}'\big|_{\ov\fM_{\wt{P}}'^{\C}(B)}
\lra  \big\{\ff\!\times\!\ff|_{\ov\fM_{\wt{P}}'^{\C}(B)}\big\}^{\!*}\cN\io_{\wt{P}}^{\C}$$
induced by the differential of~$\ff$, 
with $\ff$ on the right-hand side as in~\eref{ffdfn_e} 
with~$(g,\ell)$ replaced by \hbox{$(g',|I\!\sqcup\!J|\!+\!1)$} and
as in~\eref{Rffdfn_e} with~$(g,\ell)$ replaced by \hbox{$(g_0,|K|\!+\!1)$},
is $\C$-linear with respect to the identifications in~\eref{RioPprp_e} and~\eref{fMcNioisom_e}
and is generically an isomorphism.
By the first statements of \cite[Propositions~3.12/3/4]{RealGWsGeom},
the orientation~on 
$$\ov\fM_{\wt{P}}'^{\C}(B)\!\cap\!
\big(\ov\fM_{g',|I\sqcup J|+1}(B')\!\times\!\ov\fM_{g_0,|K|+1}^{\phi}(B_0)\!\big)
\subset\ov\fM_{\wt{P}}^{\C}(B)$$
induced from the complex orientation of~$\cN\wt\io_{\wt{P}}^{\C}$
via the analogue of the first isomorphism in~\eref{CompOrient_e0} differs from
the orientation induced via the analogue of the first isomorphism in~\eref{SubIsom_e}
by~$-1$ to the power of~$\lr{c_1(L),\phi_*(B')\!}\!+\!|J|$.
Along with~\eref{SplitCP_e} and~\eref{SplitCP_e2} with
\begin{gather*}
\sum_{\begin{subarray}{c}B_1,B_2\in H_2(X;\Z)\\B_1+B_2=B \end{subarray}}, \quad
\R\ov\cM_{g_1,|I|+1}, \quad \R\ov\cM_{g_2,|J|+1}, \quad
\cI_{g_1,|I|+1,B_1}^{\om,\phi}, \quad \cI_{g_2,|J|+1,B_2}^{\om,\phi}, \\
I, \quad J, \quad g_1, \quad \ov\fM_P^{\C}(B), \quad \ov\fM_P'^{\C}(B), \quad
\wt\io_P^{\C}, \quad\hbox{and}\quad \io_P^{\C}
\end{gather*}
replaced~by 
\begin{gather*}
\sum_{(B',B_0)\in\cP_{\wt{P}}^{\phi}(B)}, \quad
\ov\cM_{g',|I\sqcup J|+1}, \quad \R\ov\cM_{g_0,|K|+1}, \quad
\cI_{g,|I\sqcup J|+1,B'}^{\om,\phi,L}, \quad \cI_{g_0,|K|+1,B_0}^{\om,\phi}, \\
I\!\sqcup\!J, \quad K, \quad 2g'\!-\!1, \quad \ov\fM_{\wt{P}}^{\C}(B), \quad 
\ov\fM_{\wt{P}}'^{\C}(B), \quad \wt\io_{\wt{P}}^{\C}, \quad\hbox{and}\quad \io_{\wt{P}}^{\C},
\end{gather*}
respectively, this implies the second claim in~\ref{RSplit_it}.

\begin{rmk}\label{CanonOrient_rmk2}
Suppose $\wt\phi$ is a conjugation on~$L$ lifting~$\phi$ and
the moduli spaces $\ov\fM_{g,\ell}^{\phi}(B)$ are oriented via~\eref{RBPisom_e3}
as in~\cite{RealGWsI}.
By the second statement of \cite[Proposition~3.14]{RealGWsGeom}, 
the orientations on~$\ov\fM_{g-2,\ell+2}'^{\phi;\bu}(B)$
defined via the first isomorphisms in~\eref{SubIsom_e} and~\eref{CompOrient_e0} 
are then opposite.
By the second statement of \cite[Proposition~3.16]{RealGWsGeom}, 
the orientations on~$\ov\fM_{g-1,\ell+1}''^{\phi;\bu}(B)$
defined via the second isomorphisms in~\eref{SubIsom_e} and~\eref{CompOrient_e0} 
are the same if and only~if \hbox{$g\!+\!\lr{c_1(X,\om),B}/2$} is even.
By the second statements of \cite[Propositions~3.12/3]{RealGWsGeom}, 
the sign exponents in~\eref{sgnwtio_e} become
\begin{gather*}
\ve_{P;0}^{\C}=\frac{n\!-\!1}2(g_1\!-\!1)(g_2\!-\!1)
\!+\!\big(g_1\!-\!1\!+\!\lr{c_1(X,\om),B_1}/2\big)\big(g_2\!-\!1\!+\!\lr{c_1(X,\om),B_2}/2\big),\\
\ve_{\wt{P};0}^{\C}
\big|_{\ov\fM_{g',|I\sqcup J|+1}(B')\times\ov\fM_{g_0,|K|+1}^{\phi}(B_0)}=g'\!-\!1\!+\!|J|,\\
\ve_{g,(I,J);0}^{\,\C}\big|_{\ov\fM_{(g-1)/2,\ell+2}(B')}=(g\!-\!1)/2\!+\!|J|,\quad
\ve_{g,(I,J);0}^{\,E}\big|_{\ov\fM_{g/2,\ell+1}(B')}=g/2\!-\!1\!+\!|J|.
\end{gather*}
\end{rmk}

\vspace{.6in}

\noindent
{\it  Institut de Math\'ematiques de Jussieu - Paris Rive Gauche,
Universit\'e Pierre et Marie Curie, 
4~Place Jussieu,
75252 Paris Cedex 5,
France\\
penka.georgieva@imj-prg.fr}\\

\noindent
{\it Department of Mathematics, Stony Brook University, Stony Brook, NY 11794\\
azinger@math.stonybrook.edu}\\

\end{document}